\documentclass[12pt,a4paper]{article}

\usepackage{fullpage}
\usepackage{psfrag}
\usepackage{amsmath}
\usepackage{fancybox}
\usepackage{amsfonts}
\usepackage{amsthm}
\usepackage{amssymb}
\usepackage{comment}
\usepackage{graphicx}
\usepackage[normalem]{ulem} 
\usepackage{tikz}
\usepackage{enumerate}
\usepackage{MnSymbol}

 \theoremstyle{plain}
    \newtheorem{theorem}{Theorem}[section]
    \newtheorem{corollary}[theorem]{Corollary}
    \newtheorem{lemma}[theorem]{Lemma}
    \newtheorem{proposition}[theorem]{Proposition}
    \newtheorem{assumption}[theorem]{Assumption}
 \theoremstyle{definition}
    \newtheorem{definition}[theorem]{Definition}
    
 \theoremstyle{remark}
    \newtheorem{example}[theorem]{Example}
 
 \usepackage[english]{babel}

\usepackage[%
   bookmarks=true,
   pagebackref=false, 
   colorlinks,
   linkcolor=blue,
   anchorcolor=blue,
   citecolor=blue,
   filecolor=blue,
   urlcolor=blue
   ]{hyperref}

\usepackage{bm}
\usepackage{xspace}

\newcommand{\D}{\displaystyle}
\usepackage{color}
\definecolor{miverde}{RGB}{128,255,0}
\definecolor{minaranja}{RGB}{255,128,1}

\DeclareMathOperator{\supp}{supp}

\DeclareMathOperator{\Span}{span}
\DeclareMathOperator{\diam}{diam}
\DeclareMathOperator{\Div}{div}

\newcommand\NN{\mathbb N}
\newcommand\AAA{\mathcal A}
\newcommand\BB{\mathcal B}
\newcommand\BBB{B}
\newcommand\MM{\mathcal M}

\newcommand\QQ{\mathcal Q}
\newcommand\HH{\mathcal H}
\newcommand\VV{\mathcal S}

\newcommand\EE{\mathcal E}
\newcommand\RRR{\mathcal R}
\newcommand{\Res}{\mathbf R}
\newcommand\CC{\mathcal C}
\newcommand\OO{\mathcal O}
\newcommand\NNN{\mathcal N}

\newcommand{\RR}{\mathbb R}

\newcommand{\II}{\mathcal I}

\usepackage{authblk}
\newcommand{\norm}[1]{{\left\vert\kern-0.25ex\left\vert\kern-0.25ex\left\vert #1 
    \right\vert\kern-0.25ex\right\vert\kern-0.25ex\right\vert}}

\begin{document}

\title{\bf A posteriori error estimators\\ for hierarchical B-spline discretizations}
\author[1,2]{Annalisa Buffa} 
\author[3]{Eduardo M. Garau\thanks{Corresponding author. Email: egarau@santafe-conicet.gov.ar}}
\affil[1]{\small \'Ecole Polytechnique F\'ed\'erale de Lausanne, School of Basic Sciences, MATHICSE-MNS, Lausanne, Switzerland.}
\affil[2]{\small Istituto di Matematica Applicata e Tecnologie Informatiche 
`E. Magenes' (CNR), Pavia, Italy.}
\affil[3]{\small Universidad Nacional del Litoral, Consejo Nacional de Investigaciones Cient\'ificas y T\'ecnicas, FIQ, Santa Fe, Argentina.}

\maketitle

\begin{abstract} In this article we develop function-based a posteriori error estimators for the solution of linear second order elliptic problems considering hierarchical spline spaces for the Galerkin discretization.
  We prove a global upper bound for the energy error. The theory hinges on some weighted Poincar\'e type inequalities, where the B-spline basis functions are the weights appearing in the norms. Such inequalities are derived following the lines in~\cite{Veeser-Verfurth}, where the case of standard finite elements is considered. Additionally, we present numerical experiments that show the efficiency of the error estimators independently of the degree of the splines used for discretization, together with an adaptive algorithm guided by these local estimators that yields optimal meshes and rates of convergence, exhibiting an excellent performance. 
\end{abstract}

\begin{quote}\small
\end{quote}

\section{Introduction}

The design of reliable and efficient a posteriori error indicators for guiding local refinement when solving numerically partial differential equations is essential, both for defining a robust and automatic adaptive procedure and for ensuring to find suitable approximations of the desired solution without exceeding the limits of available softwares.

The main idea behind a posteriori error estimation is to build a properly locally refined mesh in order to equidistribute the approximation error. Whereas for standard finite element methods several intuitive ways of refining locally a mesh are clear and broadly analysed, for isogeometric methods~\cite{Hughes_Cottrell_Bazilevs,IGA-book}, the development
of efficient and robust strategies to get suitably locally refined meshes constitutes a challenging 
problem because the tensor product structure of B-splines~\cite{DeBoor,Schumi} is broken. Different alternatives have been proposed in order to tackle this situation, such as hierarchical splines, T-splines, LR-splines or PHT-splines. Among them, hierarchical splines based on the construction presented in~\cite{Vuong_giannelli_juttler_simeon} (see also the previous work~\cite{Kraft}) are probably the easiest to define and to implement for their use in the context of isogeometric methods.

Regarding a posteriori error estimation using hierachical spline spaces for discretizations in isogeometric methods, up to this moment we can only mention the results from~\cite{BGi15}, where residual based error indicators has been proposed. In that case, the authors considered an approach following the standard techniques in classical finite elements for deriving element-based a posteriori error estimators using the truncated basis for hierachical splines introduced in~\cite{Giannelli2012485,GJS14}. The presented proof for the reliability of such estimators needs to assume some restrictions over the hierachical meshes, with the purpose of controlling the overlap of the truncated basis function supports.

We notice that although truncation is indeed a possible strategy to recover partition of unity, this procedure requires a specific construction that entails complicated basis function supports, that may be non convex and/or not connected, and their use may produce a non negligible overhead with an adaptive strategy. Thus, in this article, we consider the hierachical basis without truncation recovering the simplicity of basis function supports, that in this case are boxes. Moreover, taking into account the hierachical space defined in~\cite{BG15}, we can also recover the partition of unity.  

The main goal of this article is to obtain simple residual type a posteriori error estimators for linear second order elliptic problems using discretizations in hierarchical spline spaces. As pointed out in~\cite{BG15}, where the hierarchical basis can be obtained simply through parent-children relations, the design of estimators associated to basis functions (instead of elements) seems to be more suitable or natural for guiding adaptive refinements. We derive reliable function-based a posteriori error indicators without any restrictions over the hierachical mesh configurations, i.e., we are able to bound the error in energy norm by our global a posteriori error indicator. In order to prove such upper bound for the error, it is key the use of some Poincar\'e type inequalities where the B-splines are weight functions and also that their supports are, obviously, convex sets. In this point it is important to remark that a posteriori error estimations in the context of classical finite element methods have been widely studied in the available literature, see for example,~\cite{Babuska-Rheinboldt-78} and~\cite{MNS-stars}. On a hand, our approach can be considered as a generalization to high order splines of some existent Poincar\'e type inequalities; and on the other hand, in order to do that, we follow closely the lines from~\cite{Veeser-Verfurth}, where specific Poincar\'e type inequalities are proved and the classical barycentric coordinate functions appear as weight functions.

This article is organized as follows. In Section~\ref{S:setting} we briefly introduce the variational formulation of the elliptic problem that we consider, and in Section~\ref{S:discretization} we describe precisely the hierachical spline spaces that we use for its Galerkin discretization. Next, we state and prove some Poincar\'e type inequalities that have B-splines as weight functions in Section~\ref{S:Poincare}, which are used in Section~\ref{S:estimators} to derive function-based a posteriori error estimators and to prove that such estimators constitute an upper bound for the energy error. In Section~\ref{S:refinement} we analyse a reduction property of our estimators after refinement of the hierachical mesh. Finally, in Section~\ref{S:aigm} we propose an adaptive algorithm guided by our estimators and illustrate its behaviour through several numerical tests, showing that the global estimator is efficient and the algorithm experimentally converges with the optimal rate.

\section{Problem setting}\label{S:setting}

For simplicity, we consider the following linear elliptic problem on the parametric domain $\Omega= [0,1]^d\subset \RR^d$, $d=2,3,\dots$,
\begin{equation}\label{E:strong form}
\left\{
\begin{aligned}
-\Div \big(\AAA\nabla u\big) + \bm{b} \cdot \nabla u + cu&= f\qquad &
& \text{in }\Omega\\
u&= 0\qquad & &\text{on }\partial \Omega
  \end{aligned}
\right.
\end{equation}
where $\AAA \in W^{1,\infty}(\Omega;\RR^{d\times d})$ is uniformly symmetric positive
definite over $\Omega$, i.e., 
there exist constants $0 <\gamma_1 \leq \gamma_2$ such that 
\begin{equation}\label{E:A uniformly spd}
\gamma_1 |\xi|^2 \leq \xi^{T}\AAA(x)\xi \leq \gamma_2 |\xi|^2,  \quad \forall x \in
\Omega,\  \xi \in \RR^{d}, 
\end{equation}
$\bm{b} \in W^{1,\infty}(\Omega;\RR^{d})$, $c \in  L^{\infty}(\Omega)$.
We assume that $c -\frac12 \Div(\bm{b}) \geq 0$. 

We say that $u\in H^1_0(\Omega):=W^{1,2}_0(\Omega)$ is a weak solution of~\eqref{E:strong form} if 
\begin{equation}\label{E:weak form}
\BBB[u,v]=F(v),\quad\forall\,v\in H^1_0(\Omega),
\end{equation}
where $\BB:H^1_0(\Omega)\times H^1_0(\Omega)\to \RR$ is the bounded bilinear
form given by
\begin{equation*}\label{E:bilinear form}
\BBB[u,v]:=\int_\Omega \AAA\nabla u\cdot \nabla v + \bm{b} \cdot \nabla u \,v + c\,u\,v,
\end{equation*}
and $F:H^1_0(\Omega)\to \RR$ is the lineal functional defined by
$$
 F(v):= \int_\Omega fv.
$$

Taking into account~\eqref{E:A uniformly spd} and that $c -\frac12 \Div(\bm{b}) \geq 0$ , it is easy to check that $B$ is coercive, that is,
\begin{equation}\label{E:coercivity}
 \gamma_1\|\nabla v\|_{L^2(\Omega)}^2\le \BBB[v,v],\qquad\forall\,v\in H^1_0(\Omega).
\end{equation}
Thus, as a consequence of the Lax-Milgram theorem, we have that problem~\eqref{E:weak form} is well posed.

\section{Discretization using hierarchical spline spaces}\label{S:discretization}

In this section we revise briefly the definitions of univariate and tensor product B-splines that we use to build a basis for a hierachical spline space like those from~\cite{Kraft,Vuong_giannelli_juttler_simeon}. Then, we state the discrete formulation of problem~\eqref{E:weak form} when considering such spaces.  

\paragraph{Univariate B-spline bases}
Let $\Xi_{p,n}:=\{\xi_j\}_{j=1}^{n+p+1}$ be a $p$-open knot vector, i.e., a \emph{sequence} such that
\begin{equation*}\label{E:open knot vector}
 0=\xi_1=\dots=\xi_{p+1}<\xi_{p+2}\le\dots\le\xi_{n}<\xi_{n+1}=\dots=\xi_{n+p+1}=1,
\end{equation*}
 where the two positive integers $p$ and $n$ denote a given polynomial degree, and the corresponding number of B-splines defined over the subdivision $\Xi_{p,n}$, respectively. Here, $n\ge p+1$.
We also introduce the \emph{set} $Z_{p,n}:=\{\zeta_j\}_{j=1}^{\tilde n}$ of breakpoints (i.e., knots without repetitions), and denote by $m_j$ the multiplicity of the breakpoint $\zeta_j$, such that
$$\Xi_{p,n}=\{\underbrace{\zeta_1,\dots,\zeta_1,}_{m_1 \text{ times}}\underbrace{\zeta_2,\dots,\zeta_2,}_{m_2 \text{ times}}\dots \underbrace{\zeta_{\tilde n},\dots,\zeta_{\tilde n}}_{m_{\tilde n} \text{ times}}\},$$
with $\D\sum_{i=1}^{\tilde n} m_i = n+p+1$. Note that the two extreme knots are repeated $p+1$ times, i.e., $m_1=m_{\tilde n}=p+1$. We assume that an internal knot can be repeated at most $p+1$ times, that is, $m_j\le p+1$, for $j=2,\dots,{\tilde n}-1$. 

Let $\BB(\Xi_{p,n}) := \{b_1,b_2,\dots,b_n\}$ be the B-spline basis~\cite{DeBoor,Schumi} associated to the knot vector $\Xi_{p,n}$. In particular, we remark that the \emph{local knot vector} of $b_j$ is given by $\{\xi_{j},\dots,\xi_{j+p+1}\}$,
which is a subsequence of $p+2$ consecutive knots of $\Xi_{p,n}$; and that the support of $b_j$, denoted by $\supp b_j$, is the closed interval $[\xi_{j},\xi_{j+p+1}]$. Additionally, the B-spline basis $\BB(\Xi_{p,n})$ is in fact a basis for the space $\VV_{p,n}$ of the piecewise polynomials of degree $p$ over the mesh $\II(\Xi_{p,n}):=\{[\zeta_j,\zeta_{j+1}]\,|\,j=1,\dots,{\tilde n}-1\}$, that have $r_j:=p-m_j$ continuous derivatives at the breakpoint $\zeta_j$, for $j= 1,\dots,{\tilde n}$. If $r_j = -1$ for some $j$, the splines in $\VV_{p,n}$ can be discontinuous at $\zeta_j$.

\paragraph{Tensor product B-spline bases}
Let $d\ge 1$. In order to define a tensor product $d$-variate spline function space on the parametric domain $\Omega:=[0,1]^d\subset \RR^d$, we consider ${\bf p}:=(p_1,p_2,\dots,p_d)$ the vector of polynomial degrees with respect to each coordinate direction and ${\bf n}:=(n_1,n_2,\dots,n_d)$, where $n_i\ge p_i+1$. For $i=1,2,\dots,d$, let $\Xi_{p_i,n_i}:=\{\xi_j^{(i)}\}_{j=1}^{n_i+p_i+1}$ be a $p_i$-open knot vector, i.e.,
\begin{equation*}
 0=\xi_1^{(i)}=\dots=\xi_{p_i+1}^{(i)}<\xi_{p_i+2}^{(i)}\le\dots\le\xi_{n_i}^{(i)}<\xi_{n_i+1}^{(i)}=\dots=\xi_{n_i+p_i+1}^{(i)}=1,
\end{equation*}
where the two extreme knots are repeated $p_i+1$ times and any internal knot can be repeated at most $p_i+1$ times. We denote by $\VV_{\bf p,n}$ the tensor product spline space spanned by the B-spline basis $\BB_{\bf p,n}$  defined as the tensor product of the univariate B-spline bases $\BB(\Xi_{p_1,n_1}),\ldots,\BB(\Xi_{p_d,n_d})$. More precisely, $\beta\in\BB_{\bf p,n}$ if and only if 
\begin{equation}\label{E:tensor product beta}
 \beta(x)=\beta_1(x_1)\dots\beta_j(x_j)\dots\beta_d(x_d),
\end{equation}
where $\beta_j\in \BB(\Xi_{p_j,n_j})$, for $j=1,2,\dots,d$, and $x_j$ denotes the $j$-th 
component of $x\in\RR^d$. We notice that the support of $\beta$, denoted by $\omega_\beta$, is a box in $\RR^d$ given by
\begin{equation}\label{E:omega beta}
 \omega_\beta:=\supp \beta = \supp\beta_1\times\dots\times \supp\beta_j\times\dots\times \supp\beta_d.
\end{equation}
Finally, the associated Cartesian grid $\QQ_{\bf p,n}$ consists of the cells $Q= I_1\times\dots\times I_d$, where $I_i$ is an element (closed interval) of the $i$-th univariate mesh $\II(\Xi_{p_i,n_i})$, for $i=1,\dots,d$. 
 
\paragraph{Sequence of tensor product spline spaces}
 
In order to define a hierachical structure, we assume that there exists an underlying sequence of tensor product $d$-variate spline spaces $\{\VV_\ell\}_{\ell\in\NN_0}$, where $\VV_\ell$ is called \emph{the space of level $\ell$}, such that
\begin{equation}\label{E:tensor product spaces}
 \VV_0\subset \VV_1\subset \VV_2\subset\VV_3\subset\dots.
\end{equation}
Each of these spaces are indeed obtained from a tensorization of univariate spline spaces as we explain now.

Let ${\bf p}:=(p_1,p_2,\dots,p_d)$ denote the chosen vector of polynomial degrees for the univariate splines in each coordinate direction. For $\ell\in\NN_0$, $\VV_\ell:=\VV_{{\bf p},{\bf n}_\ell}$ is the tensor product spline space and $\BB_\ell:= \BB_{{\bf p},{\bf n}_\ell}$ is the corresponding B-spline basis, that we call the set of \emph{B-splines of level $\ell$}, for some ${\bf n}_\ell=(n_1^\ell,n_2^\ell,\dots,n_d^\ell)$. In order to guarantee~\eqref{E:tensor product spaces}, we assume that if $\xi$ is a knot in $\Xi_{p_i,n_i^\ell}$ with multiplicity $m$, then $\xi$ is also a knot in $\Xi_{p_i,n_i^{\ell+1}}$ with multiplicity at least $m$, for $i=1,\dots,d$ and $\ell\in\NN_0$. Furthermore, we denote by $\QQ_\ell:=\QQ_{{\bf p},{\bf n}_\ell}$ the corresponding Cartesian mesh, and we say that $Q\in\QQ_\ell$ is a \emph{cell of level $\ell$}. We note that we assume that the cells are closed sets.

B-splines possess several important properties, such as non-negativity, partition of unity, local linear independence and local support, that make them suitable for design and analysis, see \cite{IGA-book,DeBoor, Schumi} for details. Moreover, we have that B-splines of level $\ell$ can be written as linear combinations of B-splines of level $\ell+1$ with non-negative coefficients, which is known as \emph{two-scale relation}. More specifically, if $\BB_\ell =\{\beta_{i,\ell}\mid i=1,\dots,N_\ell\}$, where $N_\ell$ is the dimension of the space $\VV_\ell$, for $\ell\in\NN_0$; this property can be stated as follows:
\begin{equation} \label{E:two scale relation}
 \beta_{i,\ell} = \sum_{k=1}^{N_{\ell+1}}
c_{k,\ell+1}(\beta_{i,\ell}) \, \beta_{k,\ell+1}, 
\qquad \forall\,\beta_{i,\ell} \in\BB_\ell,
\end{equation}
with $c_{k,\ell+1}(\beta_{i,\ell}) \ge 0$. We notice that, due to the local linear independence of B-splines, only a limited number of the coefficients $c_{k,\ell+1}(\beta_{i,\ell})$ are different from zero. Taking into account~\eqref{E:two scale relation} we define the set of children of $\beta_{i,\ell}$, denoted by $\CC(\beta_{i,\ell})$, consisting of the functions $\beta_{k,\ell+1}$ such that $c_{k,\ell+1}(\beta_{i,\ell}) \ne 0$, that is,
$$
\CC(\beta_{i,\ell}):=\{ \beta_{k,\ell+1}\in\BB_{\ell+1}\,\mid\,c_{k,\ell+1}(\beta_{i,\ell}) \ne 0\}.
$$
As we will see later on, in cases of interest such as subsequent levels obtained by dyadic refinement, the number of children is bounded and the bound depends solely on the degree ${\bf p}$.

\paragraph{Hierarchical spline space}

In \cite{BG15} the authors considered a particular subspace of the hierarchical space presented in~\cite{Kraft,Vuong_giannelli_juttler_simeon}, which still enjoys of good local approximation properties and leads to simple refinement schemes, because it is defined in a way that focus on the relation between functions. In order to define precisely a basis for the hierachical space that we consider, we first need to fix a hierarchy of subdomains of $\Omega=[0,1]^d$ as in the next definition, which in turn provides the different levels in the multilevel structure.

\begin{definition} \label{def:hierarchy of subdomains}[Hierarchy of subdomains]
 Let $n \in \NN$ be arbitrary. We say that the set ${\bf\Omega}_n := \{\Omega_0,\Omega_1,\dots,\Omega_n\}$ is a hierarchy of subdomains of $\Omega$ of depth $n$ if
\begin{equation*}
 \Omega = \Omega_0 \supset \Omega_1 \supset \dots \supset 
\Omega_{n-1}\supset \Omega_n = \emptyset,
\end{equation*}
and each subdomain $\Omega_\ell$ is the union of cells of level $\ell-1$, for $\ell=1,\dots,n-1$.
\end{definition}

We now are in position of introducing the basis for the hierachical space.

\begin{definition}\label{D:hierachical basis}[Hierarchical basis]
Let $\{\VV_\ell\}_{\ell\in\NN_0}$ be a sequence of spaces like~\eqref{E:tensor product spaces} with the corresponding B-spline bases $\{\BB_\ell\}_{\ell\in\NN_0}$, and ${\bf\Omega}_n := \{\Omega_0,\Omega_1,\dots,\Omega_n\}$ a hierarchy of subdomains of depth $n$. We define the \emph{hierarchical basis} $\HH :=  \HH_{n-1}$ computed with the following recursive algorithm:
\begin{equation*}
\left \{
\begin{array}{l}
 \HH_0 := \BB_0, \\
\displaystyle  \HH_{\ell+1} := \{\beta \in  \HH_\ell \mid \supp \beta \not \subset \Omega_{\ell+1} \} \cup \bigcup_{\substack{\beta \in  \HH_\ell \\ \supp \beta \subset \Omega_{\ell+1}}} \CC(\beta),\quad \ell=0,\dots,n-2.
\end{array}
\right.
\end{equation*} 
\end{definition}

An interesting property of the hierachical basis $\HH$ is that the coefficients for writing the unity are strictly positive. That is, we have 
 \begin{equation}\label{E:partition of the unity in the hierarchical space}
\sum_{\beta\in { \HH}} a_\beta \beta (x) = 1, \qquad\text{for } x \in  \Omega,
\end{equation}
with $a_\beta > 0$ (see~\cite[Theorem 5.2]{BG15}). 

In the following we say that $\beta$ is an \emph{active function} if $\beta\in\HH$, it is an \emph{active function of level $\ell$} if $\beta \in\HH \cap \BB_\ell$, and it is a \emph{deactivated function of level $\ell$} if $\beta \in\HH_\ell \setminus \HH_{\ell+1}$. Moreover, $\HH_\ell \cap \BB_\ell$ is the union of active and deactivated functions of level $\ell$.

We remark that, unlike in the definition given in~\cite{Vuong_giannelli_juttler_simeon} where a B-spline of level $\ell+1$ is added to $\HH_{\ell+1}$ if its support is contained in $\Omega_{\ell+1}$, in Definition~\ref{D:hierachical basis} B-splines of level $\ell+1$ are added only if they are children of a deactivated function of level $\ell$.

The hierarchical spline basis $\HH$ is associated to an underlying \emph{hierarchical mesh} $\QQ \equiv \QQ({\bf\Omega}_n)$, given by
\begin{equation*}
\QQ:= \bigcup_{\ell = 0}^{n-1} \{ Q\in\QQ_\ell\,\mid\, Q\subset \Omega_\ell 
\,\wedge\, Q\not\subset \Omega_{\ell+1}\}.
\end{equation*}
Analogously, we say that $Q$ is an \emph{active cell} if $Q \in \QQ$, and it is an \emph{active cell of level $\ell$} if $Q \in \QQ_\ell \cap \QQ$. We will also say that $Q$ is a \emph{deactivated cell of level $\ell$} if $Q \in \QQ_\ell$ and $Q \subset \Omega_{\ell+1}$. 

Finally, we notice that a B-spline of level $\ell$ is active if all the active cells within its support are of level $\ell$ or higher, and at least one of such cells is actually of level $\ell$. A B-spline is deactivated when all the cells of its level within the support are deactivated. 

\paragraph{Discretization of the variational problem}

In order to discretize problem~\eqref{E:weak form} we consider a hierarchy of subdomains ${\bf \Omega}_n$ of $\Omega$ and the corresponding spline space $\VV(\QQ):=\Span\HH$ with the hierarchical basis $\HH$ and the mesh $\QQ$ as defined above. Now, we define the discrete space $\VV_0\equiv\VV_0(\QQ)$ by $$\VV_0 := \{V\in\VV(\QQ)\,|\,V_{|_{\partial \Omega}}\equiv 0\}.$$
Thus, the discrete counterpart of~\eqref{E:weak form} consists in finding $U\in \VV_0$ such that
\begin{equation}\label{E:disc prob}
B[U,V]=F(V),\quad\forall\,V\in\VV_0.
\end{equation}

\section{Weighted Poincar\'e type inequalities}\label{S:Poincare}

In this section we briefly revise some basic notions about weighted Sobolev spaces and then we state weighted Poincar\'e type inequalities (see Theorems~\ref{T:Poincare inequality} and~\ref{T:Friedrichs} below) that will be needed for proving the reliability of the function-based a posteriori error estimators to be presented in the next section.

\subsection{Some definitions about weighted Sobolev spaces}

Let $A\subset\RR^d$ be a bounded domain with
Lipschitz boundary. If $\rho$ is a nonnegative locally integrable function, we denote by
$L^2(A, \rho)$ the space of measurable functions $u$ such that
\begin{equation*}
\| u \|_{L^2(A,\rho)} 
:= \left(\int_{A}|u(x)|^2 \rho(x) dx\right)^\frac12 < \infty.
\end{equation*}
Notice that $L^2(A, \rho)$ is a Hilbert space
equipped with the scalar product
\begin{equation*}
\langle u,v\rangle_{A,\rho} := \int_{A}u(x)v(x) \rho(x) dx.
\end{equation*}

We also define the weighted Sobolev space $H^1(A,\rho)$ of weakly differentiable functions $u$ such that 
$\| u \|_{H^1(A,\rho)} < \infty$, where
\begin{equation*}
\|u\|_{H^1(A,\rho)}^2:=\|u\|_{L^2(A,\rho)}^2+\|\nabla u\|_{L^2(A,\rho)}^2. 
\end{equation*}
Finally, $H^1_0(A,\rho)$ is the closure of $C^\infty_0(A)$ in $H^1(A,\rho)$.

\subsection{A weighted Poincar\'e inequality}

Before stating the main result of this section, we recall the definition of concave functions.

\begin{definition}[Concave function]\label{D:concave function}
A function $f$ defined on a convex set $A\subset\RR^d$ is \emph{concave on $A$} if for any $\alpha$, $0<\alpha<1$, there holds
$$f(\alpha x+(1-\alpha)y)\ge \alpha f(x)+(1-\alpha)f(y),\qquad\forall\,x,y\in A.$$
\end{definition}

The weighted Poincar\'e inequality stated in~\cite{Chua-Wheeden-06} holds for a weight $\rho$ such that $\rho^s$ is a concave function on its support, for some $s>0$. Thus, in view of Theorem~\ref{T:B-spline concave} below, which states that this is the case when $\rho$ is a multivariate B-spline basis function, the following result is an immediate consequence of~\cite[Thm. 
1.1 and Thm. 1.2]{Chua-Wheeden-06} (see also~\cite[Lemma 
5.2]{Veeser-Verfurth}).

\begin{theorem}[Weighted Poincar\'e inequality]\label{T:Poincare inequality}
If $\beta$ is a tensor product B-spline basis function, then 
\begin{equation*}
 \|v-c_\beta\|_{L^2(\omega_\beta,\beta)}\le \frac{1}{\pi}\diam(\omega_\beta)\|\nabla 
v\|_{L^2(\omega_\beta,\beta)},\qquad\forall\,v\in H^1(\omega_\beta,\beta),
\end{equation*}
where $c_\beta:= \frac{\int_{\omega_\beta}v\,\beta}{\int_{\omega_\beta}\beta}$.
\end{theorem}

In order to prove that $\beta^s$ is a concave function on its support, for some $s>0$, when $\beta$ is a tensor product B-spline basis function, we first notice that the result holds for univariate B-splines thanks to the Brunn--Minkowski inequality, as explained in~\cite[Section 2]{concave}. More precisely, the following result holds.

\begin{lemma}\label{L:concave}
 Let $\beta$ be a univariate B-spline basis function of degree $p$. Then, $\beta^{\frac1p}$ is concave on its support.
\end{lemma}

On the other hand, the following \emph{generalized Cauchy-Schwarz 
inequality} can be proved by mathematical induction:
\begin{equation}\label{E:generalized CS inequality}
 (a_1^d+b_1^d)^\frac1d(a_2^d+b_2^d)^\frac1d\dots (a_d^d+b_d^d)^\frac1d \geq 
a_1a_2\dots a_d+b_1b_2\dots b_d,
\end{equation}
for all nonnegative numbers $a_1,a_2,\dots,a_d,b_1,b_2,\dots,b_d$.

Now, as a consequence of~\eqref{E:generalized CS inequality} we have following 
result.

\begin{lemma}\label{L:product of concave functions}
 If $f_1,f_2,\dots,f_d$ are nonnegative concave functions on a convex set 
$A\subset \RR^d$ then $(f_1f_2\dots f_d)^{\frac1d}$ is concave on $A$.
\end{lemma}

\begin{proof}
 The assertion of this lemma follows from~\eqref{E:generalized CS inequality} taking $a_i=(\alpha f_i(x))^\frac1d$ and $b_i=((1-\alpha)f_i(y))^\frac1d$, for $i=1,\dots,d$, and regarding Definition~\ref{D:concave function}.
\end{proof}

Finally, using Lemma~\ref{L:concave} and the last lemma we can prove the following result.

\begin{theorem}\label{T:B-spline concave}
 Let $\beta$ be a tensor product $d$-variate B-spline basis function as in~\eqref{E:tensor product beta}. If the univariate B-splines that define $\beta$ are of degree $p$, then 
$\beta^{\frac{1}{pd}}$ is concave on its support.
\end{theorem}

\begin{proof}
 From~\eqref{E:tensor product beta} we have that $\beta(x)= \beta_1(x_1)\dots\beta_j(x_j)\dots \beta_d(x_d)$, where $\beta_j$ 
are univariate B-splines for $j=1,2,\dots,d$, and $x_j$ denotes the $j$-th 
component of $x\in\RR^d$. Now, by Lemma~\ref{L:concave} we have that 
$f_j(x):=\beta_j(x_j)^\frac1p$ is concave, for $j=1,2,\dots,d$. 
Finally, applying Lemma~\ref{L:product of concave functions} we obtain that 
$(\beta_1(x_1)^\frac1p\dots\beta_j(x_j)^\frac1p\dots 
\beta_d(x_d)^\frac1p)^\frac1d=\beta(x)^{\frac{1}{pd}}$ is concave on its support.
 \end{proof}

\subsection{A weighted Friedrichs inequality}

When considering Dirichlet boundary conditions as in problem~\eqref{E:strong form}, it is useful to have a suitable weighted 
Poincar\'e-Friedrichs inequality in the case that the weight function is a B-spline which does not vanish on a part of the 
boundary of $\Omega$. More precisely, we have the following result, which generalizes~\cite[Lemma 5.1]{Veeser-Verfurth} to high order splines.

\begin{theorem}[Weighted Friedrichs inequality]\label{T:Friedrichs}
Let $\beta$ be a tensor product B-spline basis function such that $\beta_{|_{\partial \Omega}}\not\equiv 0$. Then, there exists a constant $C_F>0$, independent of $\beta$, such that
\begin{equation}\label{E:Friedrichs inequality}
 \|v\|_{L^2(\omega_\beta,\beta)}\le C_F\diam(\omega_\beta)\|\nabla 
v\|_{L^2(\omega_\beta,\beta)},
\end{equation}
for all $v\in H^1(\omega_\beta,\beta)$ satisfying $v_{|_{\Gamma_{\beta}}}\equiv 0$, where $\Gamma_{\beta}:=\partial \Omega\cap\partial \omega_\beta$ is a set with positive $(d-1)$-dimensional Lebesgue measure.
\end{theorem}

The Friedrichs inequality stated in the last theorem can be proved following the same steps from~\cite{Veeser-Verfurth}. In particular, such inequality can be obtained as a consequence of the weighted Poincar\'e inequality given in Theorem~\ref{T:Poincare inequality} and a suitable weighted trace inequality. 

The following result is a trace theorem where B-splines are used as weight functions and can be seen as an extension of~\cite[Proposition 4.3]{Veeser-Verfurth}; its proof follows exactly the same lines, but we include it here for completeness.

\begin{proposition}[Weighted trace theorem]\label{P:weighted trace theorem}
Let $\beta$ be a tensor product B-spline basis function given by
\begin{equation*}
 \beta(x)=\beta_1(x_1)\dots\beta_j(x_j)\dots\beta_d(x_d),
\end{equation*}
where $\beta_j$ are univariate B-splines of a fixed degree $p$, for $j=1,2,\dots,d$, and $x_j$ denotes the $j$-th 
component of $x\in\RR^d$. Assume that $\beta_{|_{\partial \Omega}}\not\equiv 0$, and let $Q\subset\supp\beta$ be a cell of the associated Cartesian grid that has a side $S\subset\partial\Omega$ such that $\beta_{|_{S}}\not\equiv 0$. Then, 
\begin{equation}\label{E:aux averages}
  \frac{\int_S w\beta}{\int_S \beta} - \frac{\int_Q w\beta}{\int_Q \beta} = \frac{1}{p+1}\frac{\int_Q \gamma_{Q,S}\cdot \nabla w \,\beta}{\int_Q \beta},\qquad\forall\,w\in W^{1,1}(Q),
 \end{equation}
where $\gamma_{Q,S}(x):=(x_i-a_i){\bf e}_i$, for $x = (x_1,\dots,x_d)\in Q$. Here, $a = (a_1,\dots,a_d)$ is any vertex of $Q$ that does not belong to $S$ and $i$ denotes the coordinate direction given by the unit vector ${\bf e}_i$ that is orthogonal to the side $S$. 
\end{proposition}

\begin{proof}
In this proof we use the symbol $\strokedint_A f$ to denote the average of a function $f$ over a set $A$, that is, $\frac{1}{|A|}\int_A f$.

If $v\in W^{1,1}(Q)$, as an immediate consequence of Gauss divergence theorem (see also~\cite[Proposition 4.2]{Veeser-Verfurth}) to the vector field $v\gamma_{Q,S}$, it follows that
\begin{equation}\label{E:averages}
 \strokedint_S v -\strokedint_Q v = \strokedint_Q \gamma_{Q,S}\cdot\nabla v.
\end{equation}
Now, let $w\in W^{1,1}(Q)$. Applying~\eqref{E:averages} with $v=w\beta$ we obtain
\begin{equation}\label{E:averages 2}
\strokedint_S w\beta -\strokedint_Q w\beta = \strokedint_Q \gamma_{Q,S}\cdot\nabla (w\beta)=\strokedint_Q \gamma_{Q,S}\cdot\nabla w\,\beta+\strokedint_Q \gamma_{Q,S}\cdot\nabla\beta\,w.
\end{equation}
Let $i$ be the coordinate direction given by the unit vector ${\bf e}_i$ that is orthogonal to the side $S$. Taking into account that $\beta_{|_{S}}\not\equiv 0$ and that $S\subset\partial\Omega$, we have that $\beta_i(x_i) = \frac{|S|^p}{|Q|^p}|x_i-a_i|^p$, where  $a = (a_1,\dots,a_d)$ is any vertex of $Q$ that does not belong to $S$, which in turn implies that 
\begin{equation*}
  \gamma_{Q,S}(x)\cdot \nabla \beta(x) = p\beta(x),\qquad\forall\,x\in Q.
 \end{equation*}
Thus, from~\eqref{E:averages 2} we obtain that
$$\strokedint_S w\beta -(p+1)\strokedint_Q w\beta =\strokedint_Q \gamma_{Q,S}\cdot\nabla w\,\beta.$$
Finally,~\eqref{E:aux averages} follows from the last equation dividing both sides by $\int_S \beta$, and taking into account that $\strokedint_S \beta=(p+1)\strokedint_Q \beta$.
\end{proof}

We finish this section using the last proposition for proving the weighted Friedrichs inequality stated in Theorem~\ref{T:Friedrichs}.

\begin{proof}[Proof of Theorem~\ref{T:Friedrichs}]
Let $c\in\RR$ and let $v\in H^1(\omega_\beta,\beta)$ such that $v_{|_{\Gamma_{\beta}}}\equiv 0$. Then, 
 \begin{align}
  \|v\|_{L^2(\omega_\beta,\beta)}&\le \|v-c\|_{L^2(\omega_\beta,\beta)}+\|c\|_{L^2(\omega_\beta,\beta)}\notag\\
&\le \|v-c\|_{L^2(\omega_\beta,\beta)}+\frac{\int_{\omega_\beta} \beta}{\int_{\Gamma_{\beta}} \beta} \int_{\Gamma_{\beta}} |v-c|\beta.\label{E:aux Friedrichs}
 \end{align}
Let $S\subset \Gamma_{\beta}$ and $Q_S$ be a cell that has $S$ as a side. Applying Proposition~\ref{P:weighted trace theorem} with $w=|v-c|$, taking into account that $\max_{x\in Q} |\gamma_{Q,S}(x)|\le \frac{|Q|}{|S|}$ and using H\"older inequality we have
\begin{align*}
 \frac{\int_S |v-c|\beta}{\int_S \beta}&\le \frac{\int_{Q_S} |v-c|\beta}{\int_{Q_S} \beta}+\frac{1}{p+1}\frac{|Q|}{|S|}\frac{\int_{Q_S} |\nabla v|\beta}{\int_{Q_S} \beta}\\
&\le \frac{1}{\left(\int_{Q_S} \beta\right)^{\frac12}}\left( \|v-c\|_{L^2(Q_S,\beta)}+ \frac{1}{p+1}\frac{|Q|}{|S|}\|\nabla v\|_{L^2(Q_S,\beta)}\right).
\end{align*}
Now we use the last inequality to bound the second term in the right hand side of~\eqref{E:aux Friedrichs}, 
\begin{align*}
 \frac{\int_{\omega_\beta} \beta}{\int_{\Gamma_{\beta}} \beta} \int_{\Gamma_{\beta}} |v-c|\beta &\le \int_{\omega_\beta} \beta\sum_{S\subset\Gamma_\beta}\frac{\int_S |v-c|\beta}{\int_S \beta}\\
&\le   \sqrt{d} \left(\int_{\omega_\beta}\beta\right)\left(\sum_{S\subset\Gamma_\beta} \frac{1}{\int_{Q_S} \beta}\right)^{\frac12}\left( \|v-c\|_{L^2(\omega_\beta,\beta)}+ \frac{1}{p+1}\diam(\omega_\beta)\|\nabla v\|_{L^2(\omega_\beta,\beta)}\right).
\end{align*}
where we have used the Cauchy-Schwarz inequality. Finally,~\eqref{E:Friedrichs inequality} follows taking into account~\eqref{E:aux Friedrichs}, Theorem~\ref{T:Poincare inequality} and the last inequality.
\end{proof}

\section{A computable upper bound for the error}\label{S:estimators}

In this section we use the Poincar\'e type inequalities proved in the previous section in order to get an a posteriori upper bound for the energy error when computing the discrete solution of problem~\eqref{E:weak form} using the hierarchical schemes proposed in Section~\ref{S:discretization}. Once such key inequalities are available, the procedure to derive reliable residual type a posteriori error estimators follows the standard steps used when considering classical finite elements. 

Let $u\in H^1_0(\Omega)$ be the solution of problem~\eqref{E:weak form} and $U\in\VV_0$ be the Galerkin approximation of $u$ satisfying~\eqref{E:disc prob}. Specifically, the main goal of this section is to define some computable quantities $\EE_\HH(U,\beta)$, for $\beta\in\HH$, so that
$$\|\nabla(u-U)\|_{L^2(\Omega)}\le C \left(\sum_{\beta\in\HH} \EE_\HH(U,\beta)^2\right)^\frac12,$$
for some constant $C>0$. 

Notice that the coercivity~\eqref{E:coercivity} of the bilinear form implies that
\begin{equation*}
 \|\nabla(u-U)\|_{L^2(\Omega)}^2\le \frac{1}{\gamma_1}B[u-U,u-U] = \frac{1}{\gamma_1} \langle\Res(U),u-U\rangle,
\end{equation*}
where the \emph{residual} $\Res$ of a function $V\in\VV_0$ is given by
\begin{equation*}
 \langle \Res(V), v\rangle := F(v)-B[V,v]=B[u-V,v], \qquad \text{for all $v\in 
H^1_0(\Omega)$}.
\end{equation*}
Thus, we have that
\begin{equation}\label{E:bound error for residual}
 \|\nabla(u-U)\|_{L^2(\Omega)}\le  \frac{1}{\gamma_1} \|\Res(U)\|_{H^{-1}(\Omega)},
\end{equation}
and therefore, we have to bound $\|\Res(U)\|_{H^{-1}(\Omega)}$.
Assuming that $\VV_0\subset C^1(\Omega)$, since $U\in\VV_0$, integration by parts yields
\begin{equation*}
 \langle \Res(U), v\rangle = \int_\Omega \underbrace{(f+\Div(\AAA\nabla U)- \bm{b} \cdot \nabla U - cU)}_{=:r(U)}v,\qquad\forall v\in H^1_0(\Omega).
\end{equation*}
Now, for each $v\in H^1_0(\Omega)$ we associate a discrete function $v_\QQ\in \VV_0$ given by
\begin{equation*}
 v_\QQ := \sum_{\beta\in\HH} c_\beta a_{\beta}\beta, \quad\text{where }\quad c_\beta := \begin{cases} \frac{\int_{\omega_\beta}v\beta}{\int_{\omega_\beta}\beta},&  \text{if }\beta_{|_{\partial \Omega}}\equiv 0,\\ 0,& \text{otherwise.}\end{cases}
\end{equation*}                                                                                                                                                         
Here, the coefficients $a_\beta$ are those given by~\eqref{E:partition of the unity in the hierarchical space} and $\omega_\beta=\supp\beta$ (cf.~\eqref{E:omega beta}). Taking into account that $\sum_{\beta\in\HH} a_\beta \beta\equiv 1$ on $\Omega$ and using that $U$ satisfies~\eqref{E:disc prob}, we have that
$$
 \langle \Res(U),v\rangle=\langle \Res(U),v-v_\QQ\rangle = \sum_{\beta\in\HH}a_\beta \langle \Res(U),(v-c_\beta)\beta\rangle = \sum_{\beta\in\HH}a_\beta\int_{\omega_\beta} r(U)(v-c_\beta)\beta.$$
 By H\"older's inequality and the weighted Poincar\'e type inequalities given in Theorems~\ref{T:Poincare inequality} and~\ref{T:Friedrichs}, it follows that
 \begin{align*}
 \langle \Res(U),v\rangle
 &\le \sum_{\beta\in\HH} a_\beta \|r(U)\|_{L^2(\omega_\beta,\beta)}\|v-c_\beta\|_{L^2(\omega_\beta,\beta)}\\
 &\le C_1\sum_{\beta\in\HH} a_\beta \|r(U)\|_{L^2(\omega_\beta,\beta)}\diam(\omega_\beta)\|\nabla v\|_{L^2(\omega_\beta,\beta)}\\
 &\le C_1\left(\sum_{\beta\in\HH}\|r(U)\|_{L^2(\omega_\beta,\beta)}^2\diam(\omega_\beta)^2 a_\beta\right)^\frac12 
\left(\sum_{\beta\in\HH}\|\nabla v\|_{L^2(\omega_\beta,\beta)}^2 a_\beta\right)^\frac12\\
&=C_1\left(\sum_{\beta\in\HH} \int_{\omega_\beta} |r(U)|^2\diam(\omega_\beta)^2 a_\beta \beta\right)^\frac12\|\nabla v\|_{L^2(\Omega)}. 
\end{align*}
where $C_1:=\max(C_F,\frac{1}{\pi})$ (cf.~\eqref{E:Friedrichs inequality}). In consequence,
$$\|\Res(U)\|_{H^{-1}(\Omega)}\le C_1 \left(\sum_{\beta\in\HH}a_\beta \diam(\omega_\beta)^2 \int_{\omega_\beta} |r(U)|^2   \beta \right)^\frac12.$$
Regarding~\eqref{E:bound error for residual} we finally obtain that
\begin{equation*}
 \|\nabla(u-U)\|_{L^2(\Omega)}\le \frac{C_1}{\gamma_1} \left(\sum_{\beta\in\HH}a_\beta \diam(\omega_\beta)^2 \int_{\omega_\beta} |r(U)|^2   \beta \right)^\frac12.
\end{equation*}
Let $\{\QQ_\ell\}_{\ell\in\NN_0}$ be the sequence of underlying Cartesian meshes associated to the different levels as explained in Section~\ref{S:discretization}. Without losing generality, we assume that $$\D\max_{\ell\in\NN_0}\frac{\D\max_{Q\in\QQ_\ell} \diam(Q)}{\D\min_{Q\in\QQ_\ell} \diam(Q)}<\infty,$$
where $\diam(Q)$ denotes the diameter of the cell $Q$. Now, we define the \emph{meshsize $h_\ell$ at level $\ell$} (corresponding to the Cartesian grid $\QQ_\ell$) by
\begin{equation}\label{E:meshsize}
 h_\ell:=\max_{Q\in\QQ_\ell} \diam(Q),\qquad\ell\in\NN_0.
\end{equation}
We also define
\begin{equation}\label{E:meshsize beta}
 h_\beta := h_\ell,
\end{equation}
where $\ell$ is such that $\beta\in\BB_\ell$; and notice that $h_\beta$ is equivalent to $\diam(\omega_\beta)$. 

Finally, for $V\in\Span\HH$ and $\beta\in\HH$, we define the \emph{local error indicator} $\EE_\HH(V,\beta)$ by
\begin{equation}\label{E:a posteriori error estimators}
 \EE_\HH(V,\beta) :=\sqrt{a_\beta}h_\beta\left(\int_{\omega_\beta} |r(V)|^2 \beta\right)^\frac12,
\end{equation}
and summarize we have just proved in the following result.

\begin{theorem}
 Let $u\in H^1_0(\Omega)$ be the solution of problem~\eqref{E:weak form} and $U\in\VV_0$ be the Galerkin approximation of $u$ satisfying~\eqref{E:disc prob}. Then, there exists a constant $C>0$ such that
$$\|\nabla(u-U)\|_{L^2(\Omega)}\le C \left(\sum_{\beta\in\HH} \EE_\HH(U,\beta)^2\right)^\frac12.$$
\end{theorem}

\section{Refinement for hierarchical spaces and estimator reduction}\label{S:refinement}

In this section we explain precisely how to perform the refinement of a hierachical mesh and we show that the error estimator defined in the previous section is reduced after mesh refinement in the sense of~\cite[Corollary 3.4]{CKNS08-optimality}, which is an important property in order to get results about the optimality of adaptive algorithms.

\subsection{Refinement of hierachical meshes}

We start with the following basic definition.

\begin{definition}[Enlargement]
 Let ${\bf\Omega}_n := \{\Omega_0,\Omega_1,\dots,\Omega_n\}$ and ${\bf\Omega}_{n+1}^*:= \{\Omega_0^*,\Omega_1^*,\dots,\Omega_n^*,\Omega_{n+1}^*\}$ be hierarchies of subdomains of $\Omega$ of depth (at most) $n$ and $n+1$, respectively. We say that ${\bf\Omega}_{n+1}^*$ is an \emph{enlargement} of ${\bf\Omega}_n$ if
 $$\Omega_\ell\subset\Omega_\ell^*,\qquad \ell=1,2,\dots,n.$$ 
 \end{definition}

In order to enlarge the current subdomains we have to select the regions in $\Omega$ where more ability of approximation is required. Such a choice can be done by selecting to \emph{refine} some active functions, as we explain in Section~\ref{S:aigm} below where we will consider a precise way of enlarging the hierarchy ${\bf\Omega}_n$.

If ${\bf\Omega}_{n+1}^*$ is an enlargement of ${\bf\Omega}_n$, we denote by $\HH^*$ the hierarchical basis (associated to the hierarchy ${\bf\Omega}_{n+1}^*$) as in Definition~\ref{D:hierachical basis}.  
From~\cite[Theorem 5.4]{BG15}) it follows that
 $$\Span\HH\subset \Span\HH^*,$$
 and thus, we say that $\HH^*$ is a \emph{refinement} of $\HH$. 

Finally, we denote by $\QQ^*$ the \emph{refined} 
hierarchical mesh given by 
$$\QQ^*:=\bigcup_{\ell = 0}^{n} \{Q\in\QQ_\ell \mid Q\subset \Omega_\ell^* 
\wedge Q\not\subset \Omega_{\ell+1}^*\}.$$ 

\subsection{Error estimator reduction}

Let $\{\QQ_\ell\}_{\ell\in\NN_0}$ be the sequence of underlying Cartesian meshes associated to the different levels as explained in Section~\ref{S:discretization} and let $\{h_\ell\}_{\ell\in\NN_0}$ be the sequence of meshsizes defined by~\eqref{E:meshsize}. In order to analyse the behaviour of the error estimator under refinement, we assume that the successive levels are obtained performing $q$-adic refinement in the tensor product meshes. More precisely, we state the following assumption.

\begin{assumption}\label{A:q adic refinement}
 There exists $q\in\NN$, $q\ge2$ such that
 $$h_{\ell+1}\le \frac1q h_\ell,\qquad \forall\,\ell\in\NN_0.$$
\end{assumption}

Let $\HH$ be the hierarchical basis associated to a hierarchy of 
subdomains of depth $n$,
${\bf\Omega}_n := \{\Omega_0,\Omega_1,\dots,\Omega_n\}$. Let $\{a_\beta\}_{\beta\in\HH}$ be the sequence of positive numbers as in~\eqref{E:partition of the unity in the hierarchical space} such that $\sum_{\beta\in\HH}a_\beta\beta(x) =1$, for $x\in\Omega$. For $V\in\Span\HH$ and $\beta\in\HH$, the local error indicator $\EE_\HH(V,\beta)$ defined in~\eqref{E:a posteriori error estimators} is given by
\begin{equation*}
 \EE_\HH(V,\beta) :=\sqrt{a_\beta}{h}_\beta\left(\int_{\omega_\beta} |r(V)|^2 \beta\right)^\frac12,
\end{equation*}
where $h_\beta$ is defined by~\eqref{E:meshsize beta}. Additionally, for $\NNN\subset \HH$ we define $\EE_\HH(V,\NNN)$ by
$$\EE_\HH(V,\NNN):=\left(\sum_{\beta\in\NNN}\EE_\HH^2(V,\beta)\right)^\frac12.$$
We remark that the \emph{global indicator} $\EE_{\HH}(V,\HH)$ can be computed as
$$\EE_{\HH}(V,\HH)=\left(\int_\Omega |r(V)|^2 H_\HH^2\right)^\frac12=\|r(V)\|_{L^2(\Omega,H_\HH^2)},$$
where the function $H_{\HH}^2\in\Span\HH$ is given by
\begin{equation}\label{E:H2 function}
 H_{\HH}^2:=\sum_{\beta\in\HH}a_\beta{h}_\beta^2\beta.
\end{equation}

The main result of this section is the following proposition.

\begin{proposition}[Estimator reduction]\label{P:estimator reduction 1}
Let $\{h_\ell\}_{\ell\in\NN_0}$ be the meshsizes defined in~\eqref{E:meshsize} and let Assumption~\ref{A:q adic refinement} be valid. Let $\HH$ be a hierarchical basis and let $\HH^*$ be a refinement of $\HH$. If $\RRR:=\HH\setminus\HH^*$ denotes the set of \emph{refined basis functions}, then
 \begin{equation*}
  \EE_{\HH^*}^2(V,\HH^*)\le \EE_\HH^2(V,\HH)-\lambda \EE_{\HH}^2(V,\RRR),\qquad\forall\,V\in\Span \HH.
 \end{equation*}
where $\lambda:=(1-\frac{1}{q^2})$.
\end{proposition}

This result is a consequence of Lemma~\ref{L:aux estimator reduction} stated below. In the proof of that lemma we will use the following result which is a consequence of the fact that deactivated functions of level $\ell$ (in $\HH^*$) can be written in terms of active functions of higher levels (in $\HH^*$). More precisely, from~\cite[Lemma 5.4]{BG15} we have that
\begin{equation}\label{E:decomposition}
 \{\beta_\ell\in\BB_\ell\,|\,\supp\beta_\ell\subset\Omega_{\ell+1}^*\}\subset \Span\left(\HH^*\cap\bigcup_{k=\ell+1}^{n}\BB_k\right),\qquad\ell=0,1,\dots,n-1.
\end{equation}

\begin{corollary}\label{C:decomposition}
Let $\HH$ be a hierarchical basis and let $\HH^*$ be a refinement of $\HH$. If $\RRR:=\HH\setminus\HH^*$ denotes the set of refined functions, then
$$\RRR_\ell:=\RRR\cap\BB_\ell\subset \Span\left(\HH^*\cap\bigcup_{k=\ell+1}^{n}\BB_k\right),\qquad\ell = 0,1,\dots,n-1.$$
\end{corollary}

\begin{proof}
 Let $\beta_\ell\in\RRR_\ell$ for some $\ell = 0,1,\dots,n-1$. Since $\beta_\ell\in\HH\cap\BB_\ell$, $\supp\beta_\ell\subset\Omega_\ell\subset\Omega_\ell^*$. Thus, we have that $\supp\beta_\ell\subset\Omega_{\ell+1}^*$ due to $\beta_\ell\not\in\HH^*$. Finally,~\eqref{E:decomposition} implies that $\beta_\ell\in\Span\left(\HH^*\cap\bigcup_{k=\ell+1}^{n}\BB_k\right)$.
\end{proof}

\begin{lemma}\label{L:aux estimator reduction}
Let $\{h_\ell\}_{\ell\in\NN_0}$ be the meshsizes defined in~\eqref{E:meshsize} and let Assumption~\ref{A:q adic refinement} be valid. Let $\HH$ be a hierarchical basis and let $\HH^*$ be a refinement of $\HH$. If $\RRR:=\HH\setminus\HH^*$ denotes the set of refined functions, then
 \begin{equation*}
  H_{\HH^*}^2(x)\le H_\HH^2(x) -\left(1-\frac{1}{q^2}\right)\sum_{\beta\in\RRR} a_\beta {h}_\beta^2 \beta(x),\qquad \forall\,x\in\Omega,
 \end{equation*}
where $H_\HH$ and $H_{\HH^*}$ are defined as in~\eqref{E:H2 function}, and $q\ge 2$ is the constant appearing in Assumption~\ref{A:q adic refinement}.
\end{lemma}

\begin{proof}
 Notice that
 $$
  1=\sum_{\beta\in\HH}a_\beta\beta(x)= \sum_{\beta\in\HH\setminus\RRR}a_\beta\beta(x) + \sum_{\ell=0}^{n-1}\sum_{\beta\in\RRR_\ell}a_\beta\beta(x),\qquad\forall\,x\in\Omega.$$
  By Corollary~\ref{C:decomposition}, we have that $\beta_\ell:=\sum_{\beta\in\RRR_\ell}a_\beta\beta\in\Span\left(\HH^*\cap\bigcup_{k=\ell+1}^{n}\BB_k\right)$ and therefore,
  $$\beta_\ell = \sum_{\beta\in\HH^*\cap\bigcup_{k=\ell+1}^n \BB_k}c_{\beta,\ell} \beta,$$
  for some constants $c_{\beta,\ell}$.
  Then, 
  \begin{equation}\label{E:alternative partition of the unity}
  \sum_{\beta\in\HH\setminus\RRR}a_\beta\beta(x) + \sum_{\ell=0}^{n-1}\sum_{\beta\in\HH^*\cap\bigcup_{k=\ell+1}^n \BB_k}c_{\beta,\ell} \beta(x)=1,\qquad\forall\,x\in\Omega. 
  \end{equation}
  Notice that the last equation gives a partition of unity with functions $\HH^*$. Let $\{a_\beta^*\}_{\beta\in{\HH}^*}$ be the sequence of positive numbers such that 
  \begin{equation}\label{E:partition of unity in Hstar}
   \sum_{\beta\in{\HH}^*} a_\beta^*\beta(x)= 1,\qquad\forall\,x\in\Omega.
  \end{equation}
  Notice that the difference between~\eqref{E:alternative partition of the unity} and~\eqref{E:partition of unity in Hstar} is that in the former some basis functions appear more than once. Finally, taking into account Assumption~\ref{A:q adic refinement} we have that
  \begin{align*}
   H_{\HH^*}^2&=\sum_{\beta\in\HH^*}a_\beta^*{h}_\beta^2 \beta=\sum_{\beta\in\HH\setminus\RRR}a_\beta{h}_\beta^2 \beta + \sum_{\ell=0}^{n-1}\sum_{\beta\in\HH^*\cap\bigcup_{k=\ell+1}^n \BB_k}{h}_\beta^2 c_{\beta,\ell} \beta\le \sum_{\beta\in\HH\setminus\RRR}a_\beta{h}_\beta^2 \beta + \sum_{\ell=0}^{n-1}h_{\ell+1}^2\beta_\ell\\
   &\le \sum_{\beta\in\HH\setminus\RRR}a_\beta{h}_\beta^2 \beta + \frac{1}{q^2}\sum_{\ell=0}^{n-1}h_\ell^2\beta_\ell=\sum_{\beta\in\HH\setminus\RRR}a_\beta{h}_\beta^2 \beta + \frac{1}{q^2}\sum_{\beta\in\RRR}a_\beta{h}_\beta^2 \beta = H_\HH^2 -\left(1-\frac{1}{q^2}\right)\sum_{\beta\in\RRR} a_\beta {h}_\beta^2 \beta,
  \end{align*}
which concludes the proof.
\end{proof}

We finish this section with the proof of the estimator reduction property.

\begin{proof}[Proof of Proposition~\ref{P:estimator reduction 1}]
Let $V\in\Span\HH$ and let $\lambda:=(1-\frac{1}{q^2})$. By Lemma~\ref{L:aux estimator reduction}, we have that
 \begin{align*}
   \EE_{\HH^*}^2(V,\HH^*)&=\int_\Omega |r(V)|^2H_{\HH^*}^2\le \int_\Omega |r(V)|^2\left(H_\HH^2-\lambda\sum_{\beta\in\RRR} a_\beta{h}_\beta^2 \beta\right)\\
   & = \int_\Omega |r(V)|^2H_\HH^2 -\lambda\sum_{\beta\in\RRR} \int_\Omega |r(V)|^2 a_\beta{h}_\beta^2 \beta\\
   &=\EE_\HH^2(V,\HH)-\lambda\sum_{\beta\in\RRR} \EE_\HH^2(V,\beta)\\
   &=\EE_\HH^2(V,\HH)-\lambda\EE_{\HH}^2(V,\RRR).
 \end{align*}
\end{proof}

\section{Adaptive loop and numerical examples} \label{S:aigm}

In this section we propose an adaptive algorithm guided by the a posteriori error estimators defined in Section~\ref{S:estimators}. Additionally, we show the performance of such adaptive procedure in practice through several numerical examples. 

The adaptive loop that we consider is quite standard and consists of the following modules:

\begin{equation}\label{E:steps}
 \textsc{SOLVE} \quad \longrightarrow \quad
\textsc{ESTIMATE} \quad \longrightarrow \quad
\textsc{MARK} \quad \longrightarrow \quad
\textsc{REFINE}.
\end{equation}

We start with a tensor product mesh and the corresponding spline space, regarded as the current hierarchical mesh $\QQ$ and hierarchical space $\VV(\QQ)=\Span\HH$, respectively. We perform the steps in~\eqref{E:steps} in order to get an adaptively refined mesh $\QQ^*$ and its corresponding hierachical space $\VV(\QQ^*)=\Span\HH^*$. Next, we consider $\QQ^*$ and $\HH^*$ as the current hierachical mesh and basis, respectively, and perform the steps in~\eqref{E:steps}, and so on. We now briefly describe the modules of the adaptive loop.
\begin{itemize}
 \item \textsc{SOLVE}: Compute the solution $U$ of the discrete problem~\eqref{E:disc prob} in the current hierachical space $\VV(\QQ)=\Span \HH$.
\item \textsc{ESTIMATE}: Use the current discrete solution $U$ to compute the \emph{a posteriori error
estimators} $\EE_{\beta}:=\EE_\HH(U,\beta)$ defined in~\eqref{E:a posteriori error estimators}, for each $\beta\in\HH$.
\item \textsc{MARK}: Use the a posteriori error estimators $\{\EE_\beta\}_{\beta\in\HH}$ to compute the set of \emph{marked functions} $\MM\subset \HH$ using the \emph{maximum strategy} with parameter $\theta = 0.5$, that is, $\MM\subset\HH$ consists of the basis functions $\beta\in \HH$ such that
\begin{equation*}
 \EE_\beta\ge \theta \max_{\beta'\in\HH}\EE_{\beta'}.
\end{equation*}

\item \textsc{REFINE}: Use the set of marked functions $\MM$ to enlarge the current hierarchy of subdomains ${\bf\Omega}_{n} := 
\{\Omega_0,\Omega_1,\dots,\Omega_n\}$ as follows: 

Let $\MM_\ell:=\MM\cap\BB_\ell$, for $\ell = 0,1,\dots,n-1$. Now, we define the hierarchy of subdomains ${\bf\Omega}^*_{n+1} := 
\{\Omega_0^*,\Omega_1^*,\dots,\Omega_n^*,\Omega_{n+1}^*\}$ of depth at most $n+1$, by
\begin{align}\label{E:marking basis functions}
\begin{cases}\Omega_0^* &:=\Omega_0,\\
\Omega_{\ell}^* &:=\Omega_{\ell}\cup \D\bigcup_{\beta\in\MM_{\ell-1}}\supp 
\beta,\qquad \ell = 1,2,\dots,n,\\
\Omega_{n+1}^*&:= \emptyset.
\end{cases}
\end{align}
If $\HH^*$ is the hierarchical basis associated to ${\bf\Omega}^*_{n+1}$, we notice that $\MM\subset\HH\setminus\HH^*$; in other words, at least the functions in $\MM$ have been refined and removed from the hierarchical basis $\HH$. 
\end{itemize}

Now, we present some numerical tests to show the performance of the proposed algorithm. In particular, we study the decay of the energy error in terms of degrees of freedom (DOFs) in each example, and analyse the rates of convergence. The implementation of the adaptive procedure was done using the data structure and algorithms introduced in~\cite{GV16}.

We consider the problem 
\begin{equation}\label{E:poisson}
\left\{
\begin{aligned}
-\Delta u &= f\qquad & &
\text{in }\Omega\\
u&= g\qquad & &\text{on }\partial \Omega,
\end{aligned}
\right.
\end{equation}
giving in each particular example the definition of the domain $\Omega$ and the problem data $f$ and $g$.


\begin{example}[Regular solution in the unit square]\label{Ex:regular solution}
We consider $\Omega=[0,1]\times [0,1]$ and the problem data $f$ and $g$ in~\eqref{E:poisson} are chosen such that the exact solution $u$ is given by $u(x,y) = e^{-100((x-\frac12)^2+(y-\frac12)^2)}$. In Figure~\ref{F:example regular solution meshes} we plot the exact solution, some hierarchical meshes and the decay of the energy error vs. degrees of freddom for different spline degrees. As expected, both tensor product meshes and hierarchical meshes reach optimal orders of convergence, but notice that in all cases, the curves corresponding to the adaptive strategy are meaningfully by below. For example, for attaining an energy error of around $3.10^{-4}$ using biquadratics, the global refinement requires $66564$ DOFs whereas the adaptive strategy only needs $14548$ DOFs.

\begin{figure}[H!tbp]
\begin{center}
\includegraphics[width=.3\textwidth]{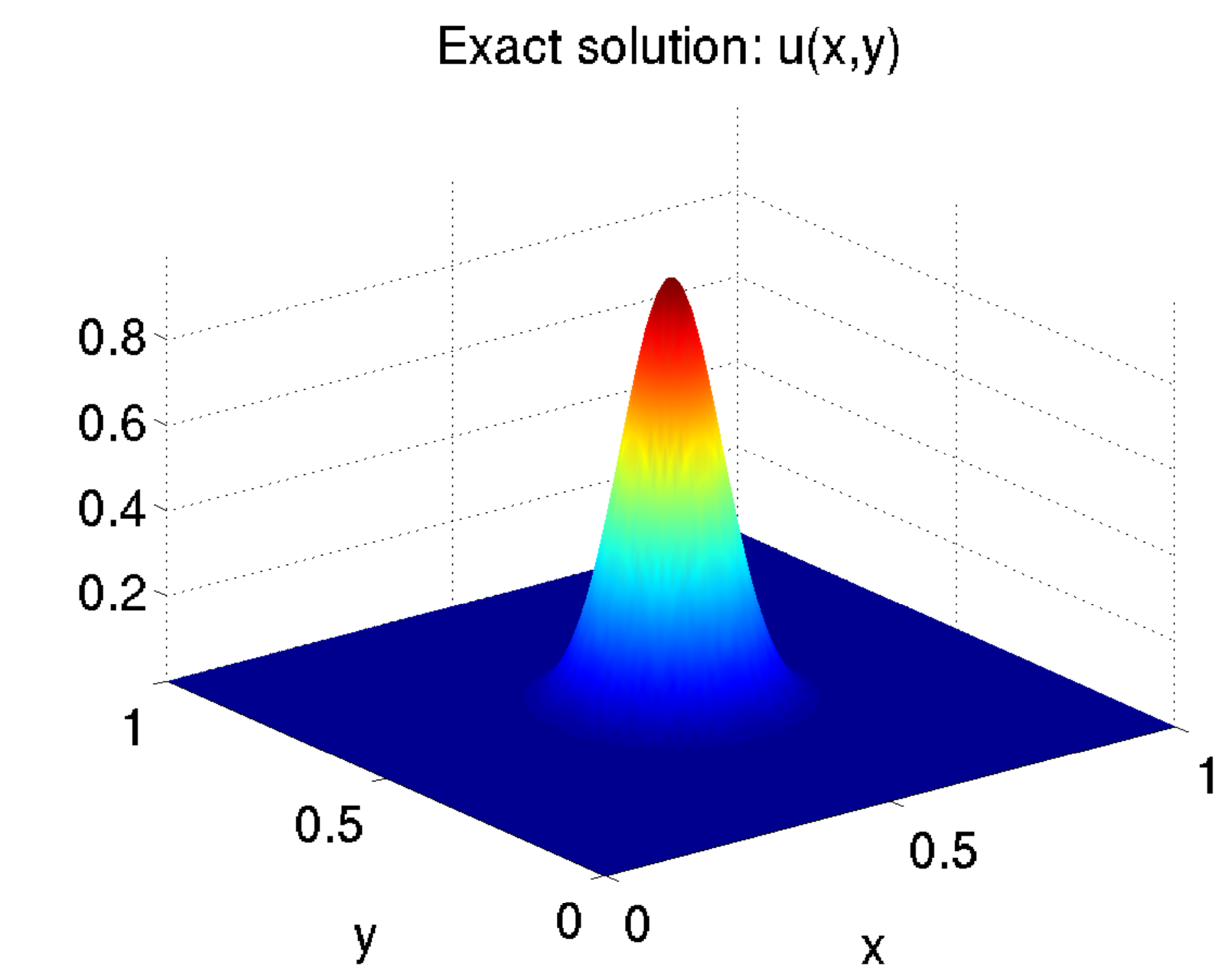}
\includegraphics[width=.2\textwidth]{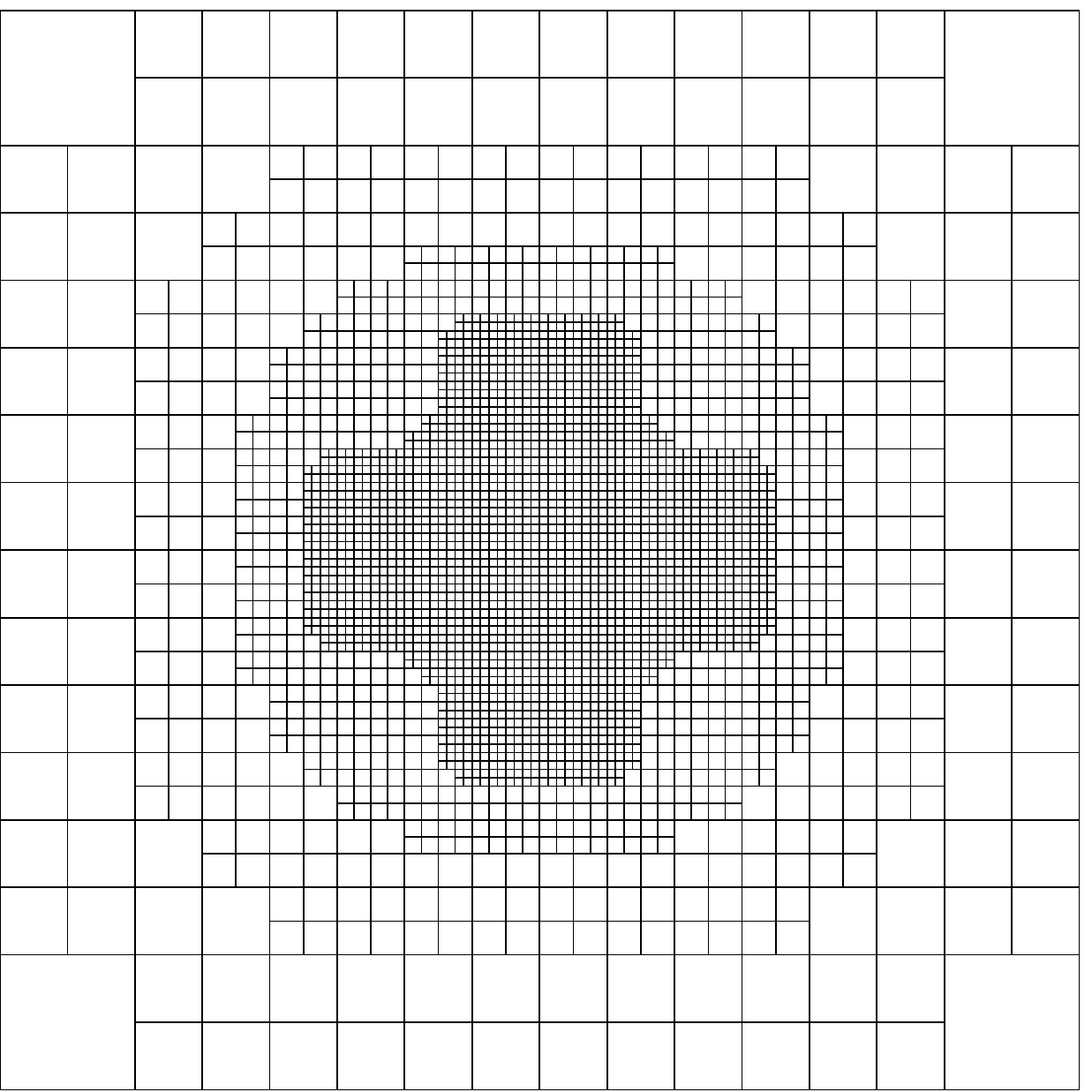}
\includegraphics[width=.2\textwidth]{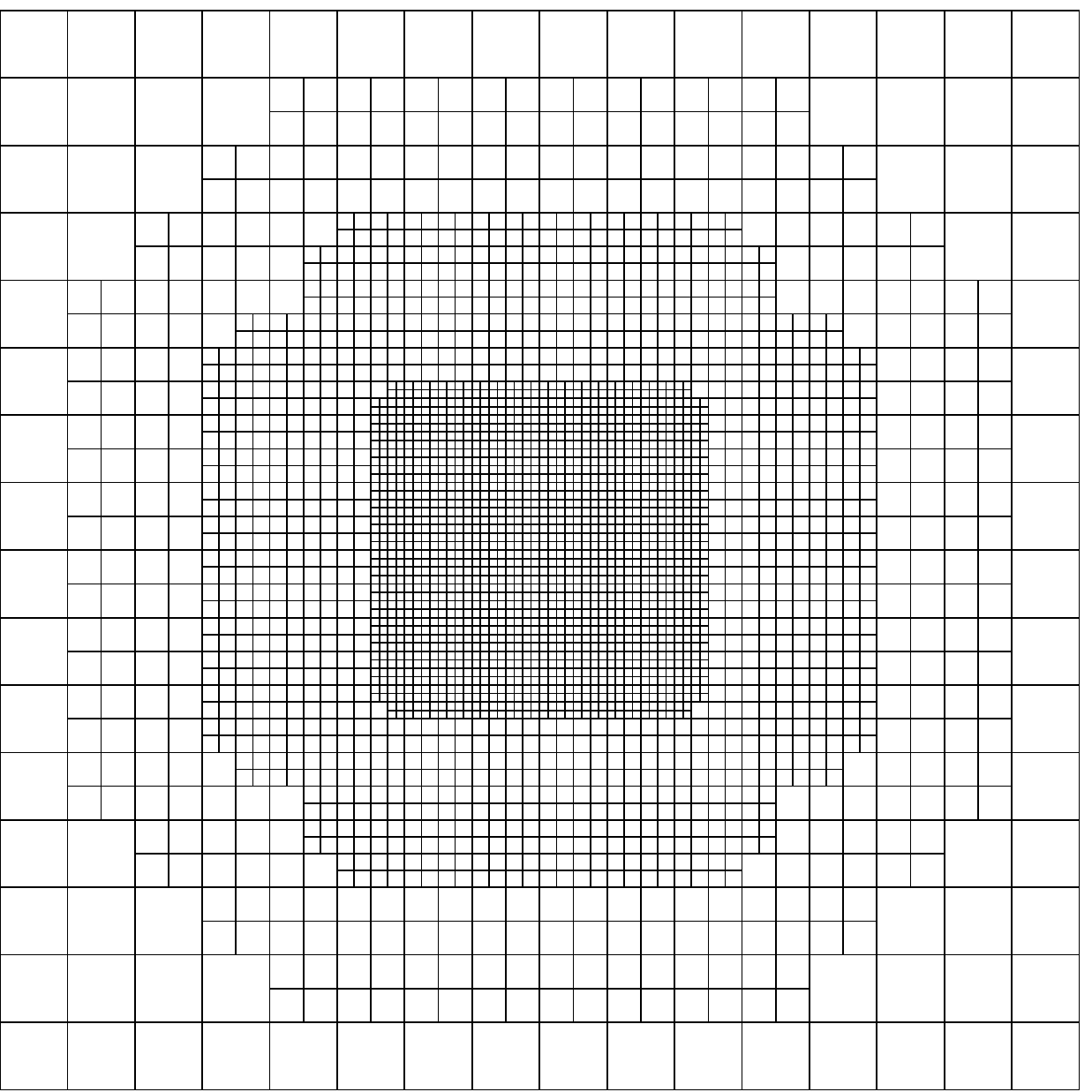} 
\includegraphics[width=.2\textwidth]{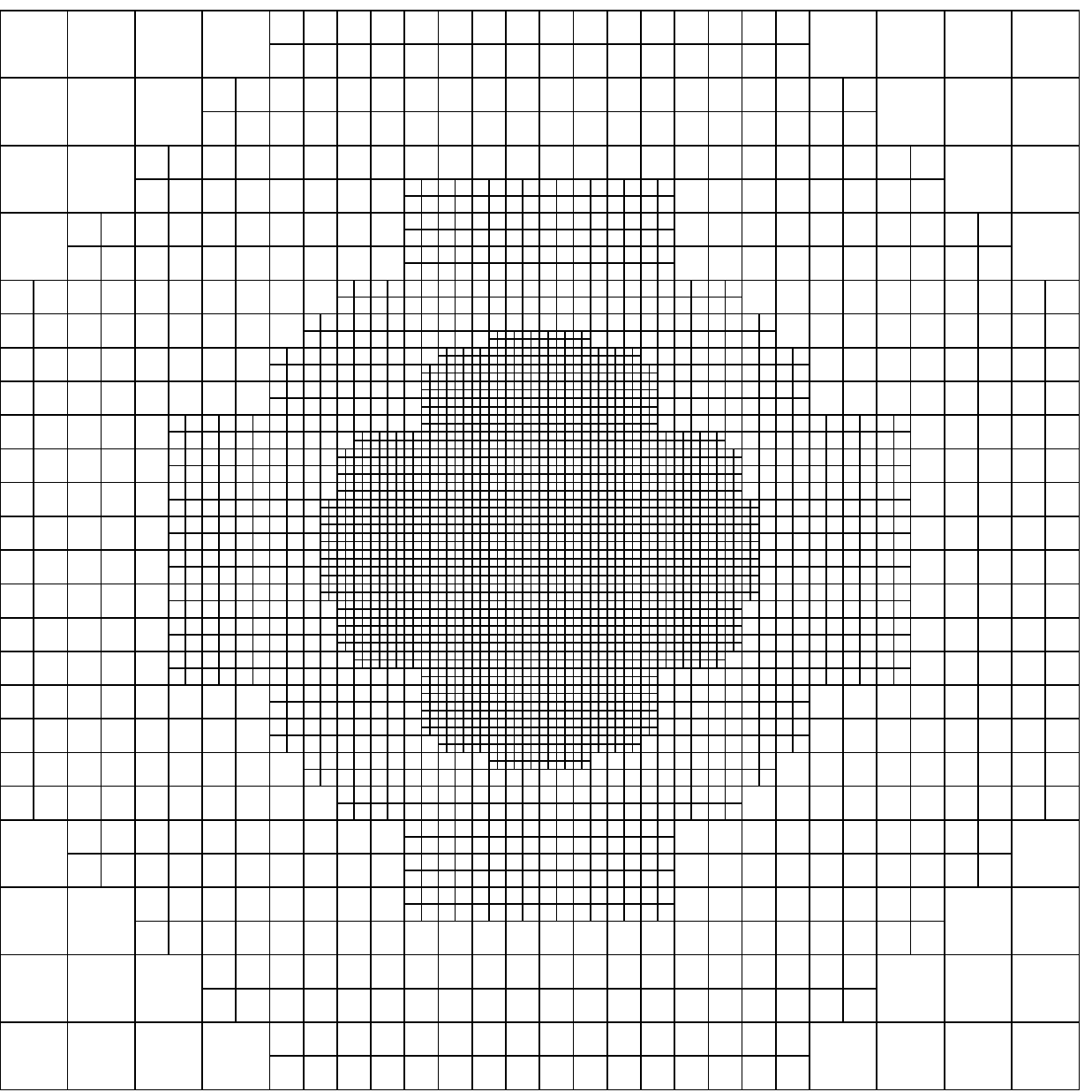}
\medskip

\includegraphics[width=.3\textwidth]{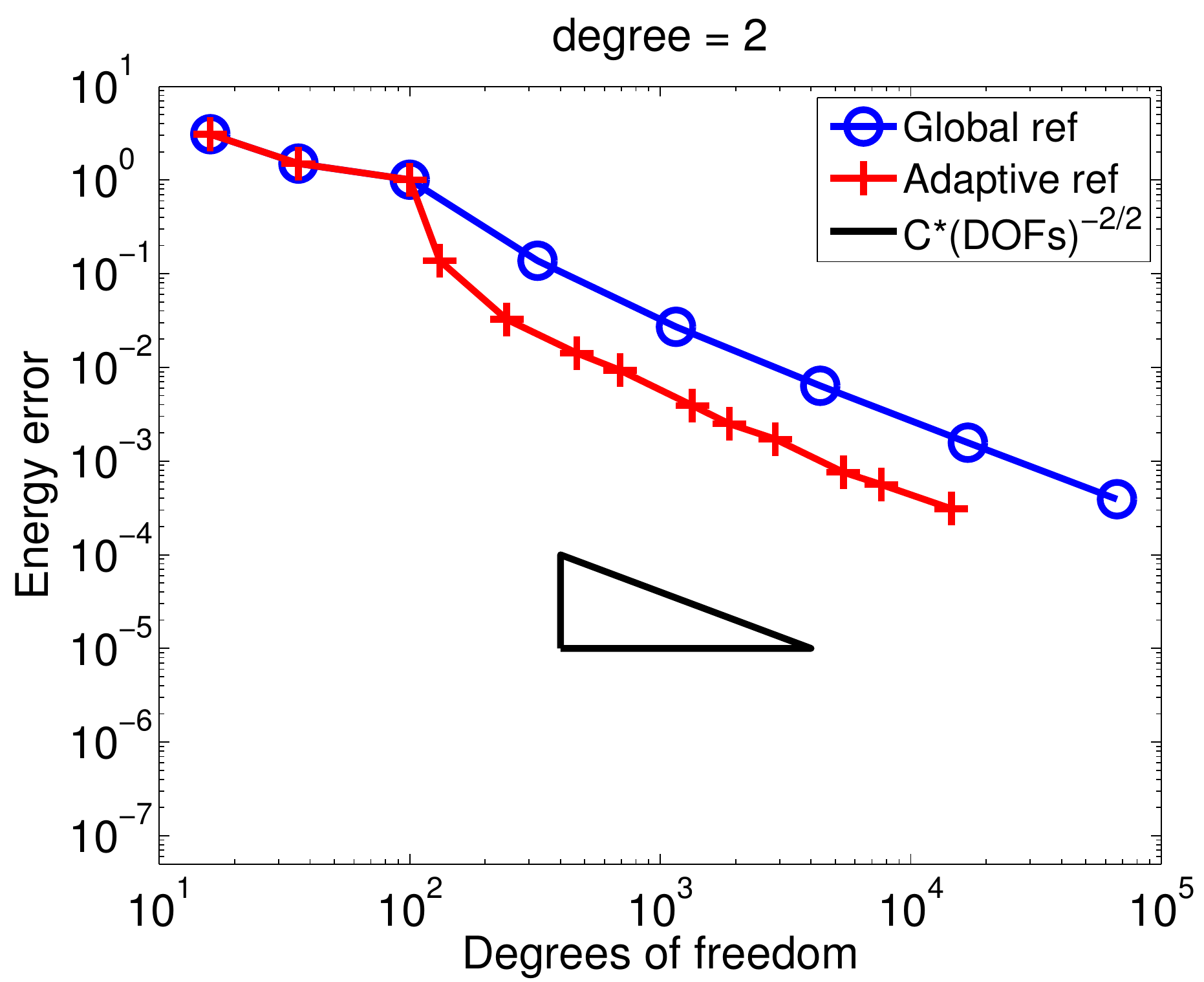}
\includegraphics[width=.3\textwidth]{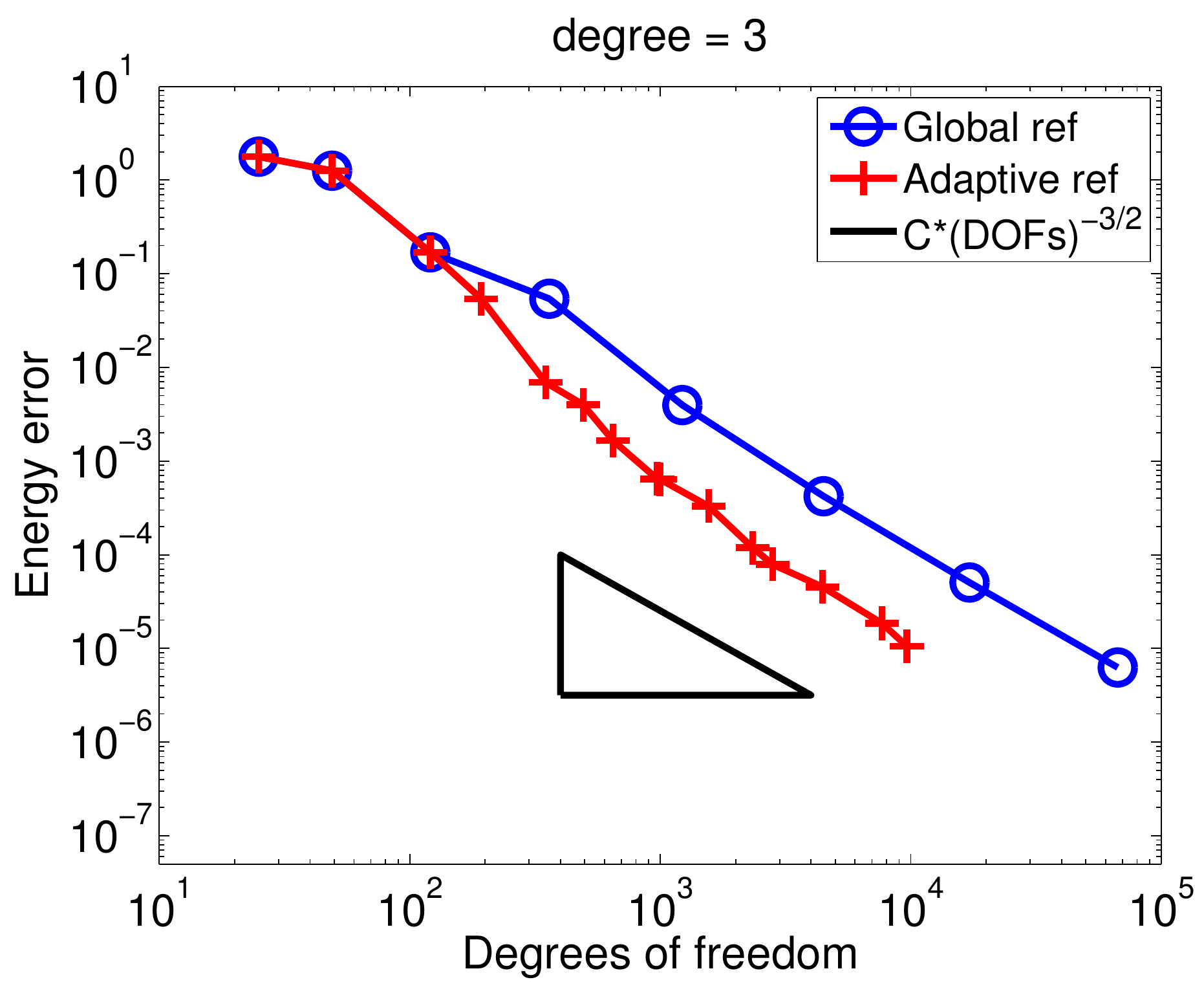}
\includegraphics[width=.3\textwidth]{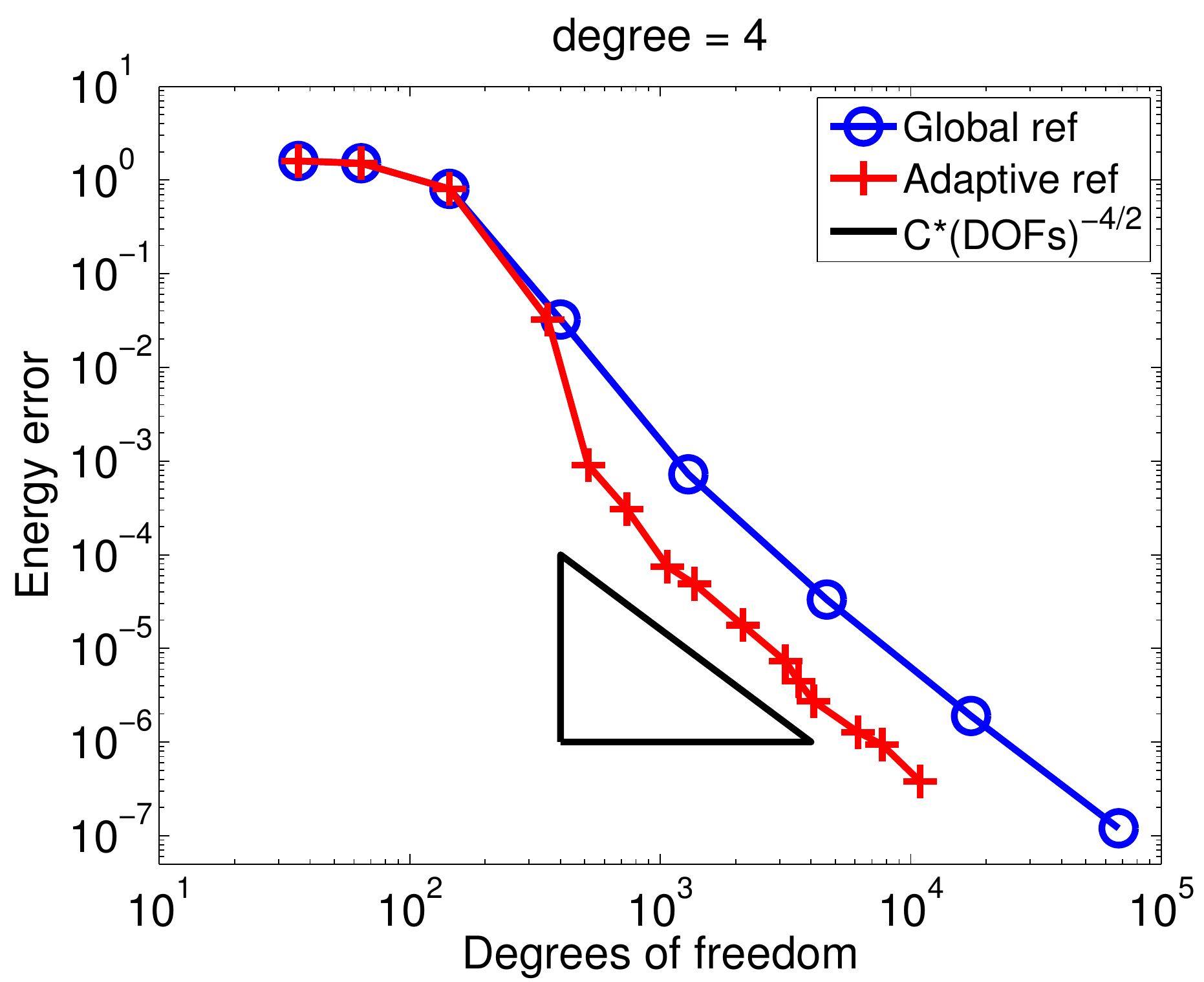}
\end{center}
\caption{\label{F:example regular solution meshes} \small Some hierarchical meshes for the solution of Example~\ref{Ex:regular solution}; for biquadratics with $3064$ elements and $2884$ DOFs (top left), bicubics with $3028$ elements and $2809$ DOFs (top middle) and biquartics with $3400$ elements and $3160$ DOFs (top right). Notice that although all meshes have nearly the same amount of elements, the refinement is more spread for high order splines due to the sizes of the basis function supports. We plot the energy error $|u-U|_{H^1(\Omega)}$ vs. degrees of freedom; for biquadratics (bottom left), bicubics (bottom middle) and biquartics (bottom right).}
\end{figure}

Additionally, in Figure~\ref{F:example regular solution meshes 2} we present the adaptive meshes obtained for different polynomial degrees, starting with an initial tensor product mesh of $4$ elements, reaching in all cases an enegy error $\approx 2.10^{-3}$. It is interesting to remark that although it is well known~\cite{ABCMS15} that the element-by-element assembly is very costly for higher degree, the use of adaptivity changes this picture. Indeed, due to the reduced number of elements for bicubics and biquartics, 
the time-to-solution for bicubics is $35\%$ and for biquartics is $17\%$ of the time-to-solution for biquadratics.

\begin{figure}[H!tbp]
\begin{center}
\includegraphics[width=.3\textwidth]{mesh-problem3-degree2-est3-iter010--eps-converted-to.pdf}
\includegraphics[width=.3\textwidth]{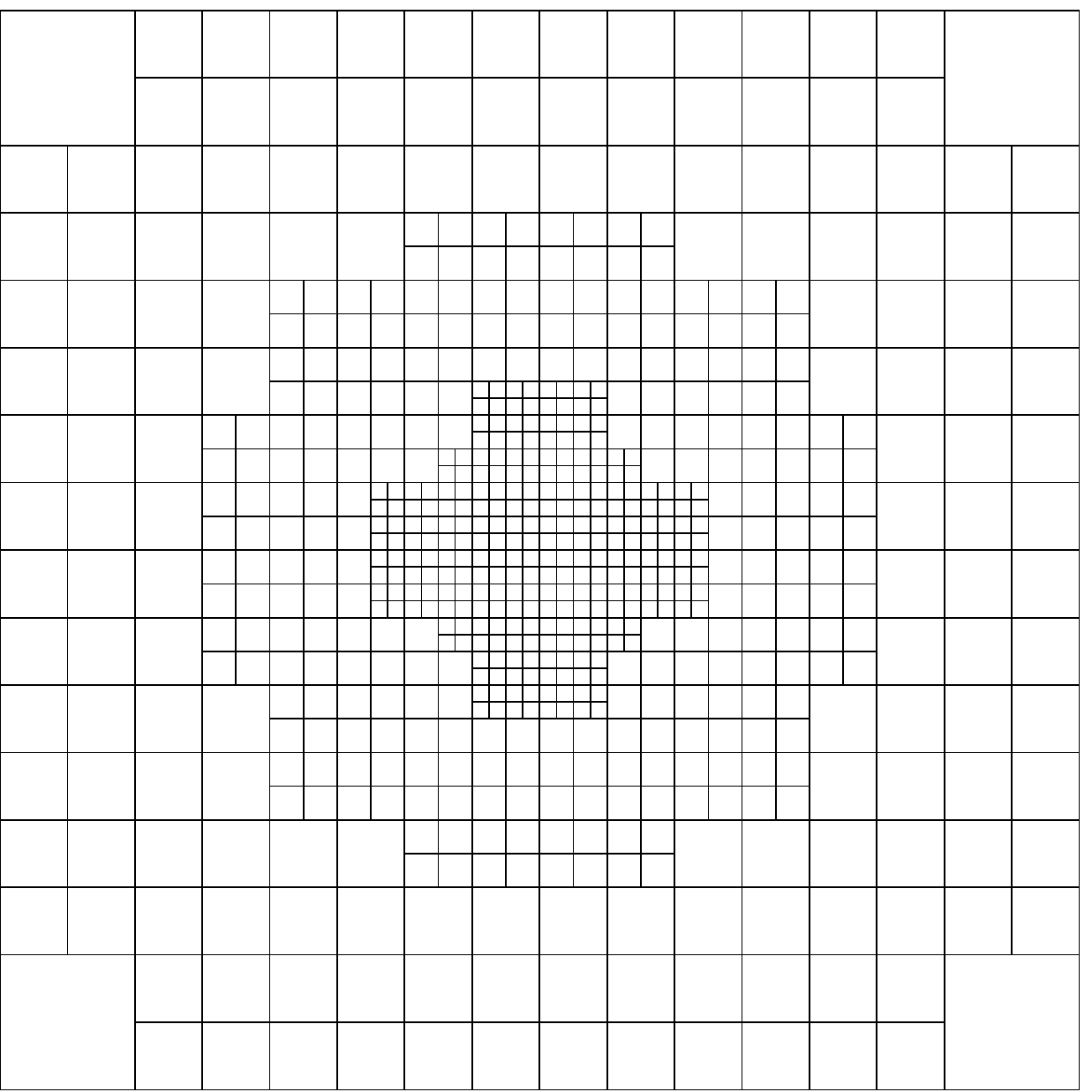} 
\includegraphics[width=.3\textwidth]{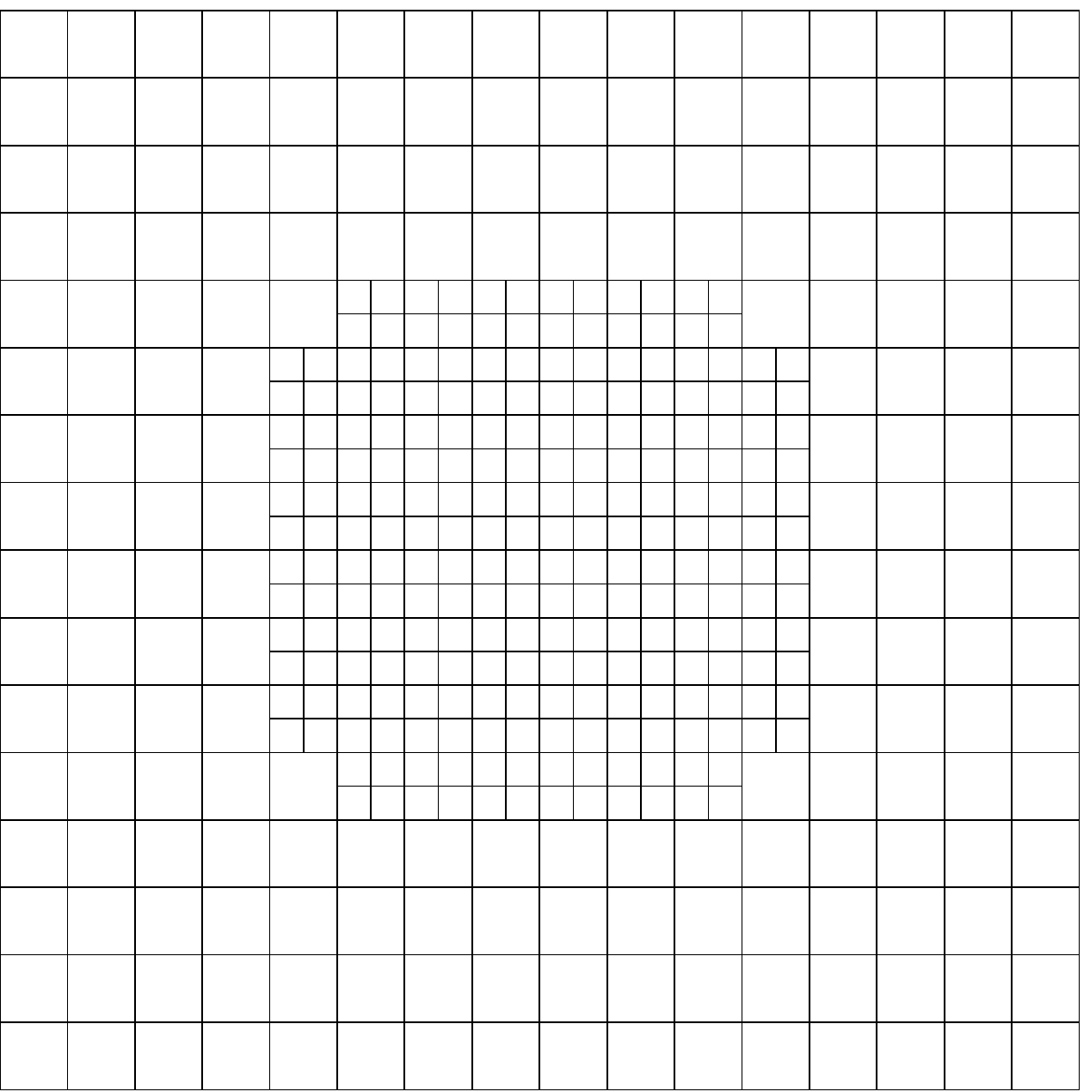}
\end{center}
\caption{\label{F:example regular solution meshes 2} \small Comparison of adaptive meshes obtained for Example~\ref{Ex:regular solution} using different polynomial degrees, obtaining in all cases that $|u-U|_{H^1(\Omega)}\approx 2.10^{-3}$. The mesh for biquadratics has $3064$ elements and $2884$ DOFs (left), the mesh bicubics has $688$ elements and $649$ DOFs (middle) and the mesh for biquartics has $436$ elements and $516$ DOFs (right).}
\end{figure}
\end{example}


\begin{example}[Diagonal refinement in the unit square]\label{Ex:diagonal}
We take $\Omega=[0,1]\times[0,1]$ and choose $f$ and $g$ such that the exact solution $u$ of~\eqref{E:poisson} is given by $u(x,y) = \tan^{-1}(25(x-y))$. In Figure~\ref{F:example diagonal} we plot the exact solution, some hierarchical meshes and the decay of the energy error vs. degrees of freddom for different spline degrees. As in the previous example, both tensor product meshes and hierarchical meshes reach optimal orders of convergence, but notice that in all cases, the curves corresponding to the adaptive strategy are again meaningfully by below. For example, for attaining an energy error of around $10^{-5}$ using biquadratics, the global refinement requires $67600$ DOFs whereas the adaptive strategy only needs $10186$ DOFs.

\begin{figure}[H!tbp]
\begin{center}
\includegraphics[width=.3\textwidth]{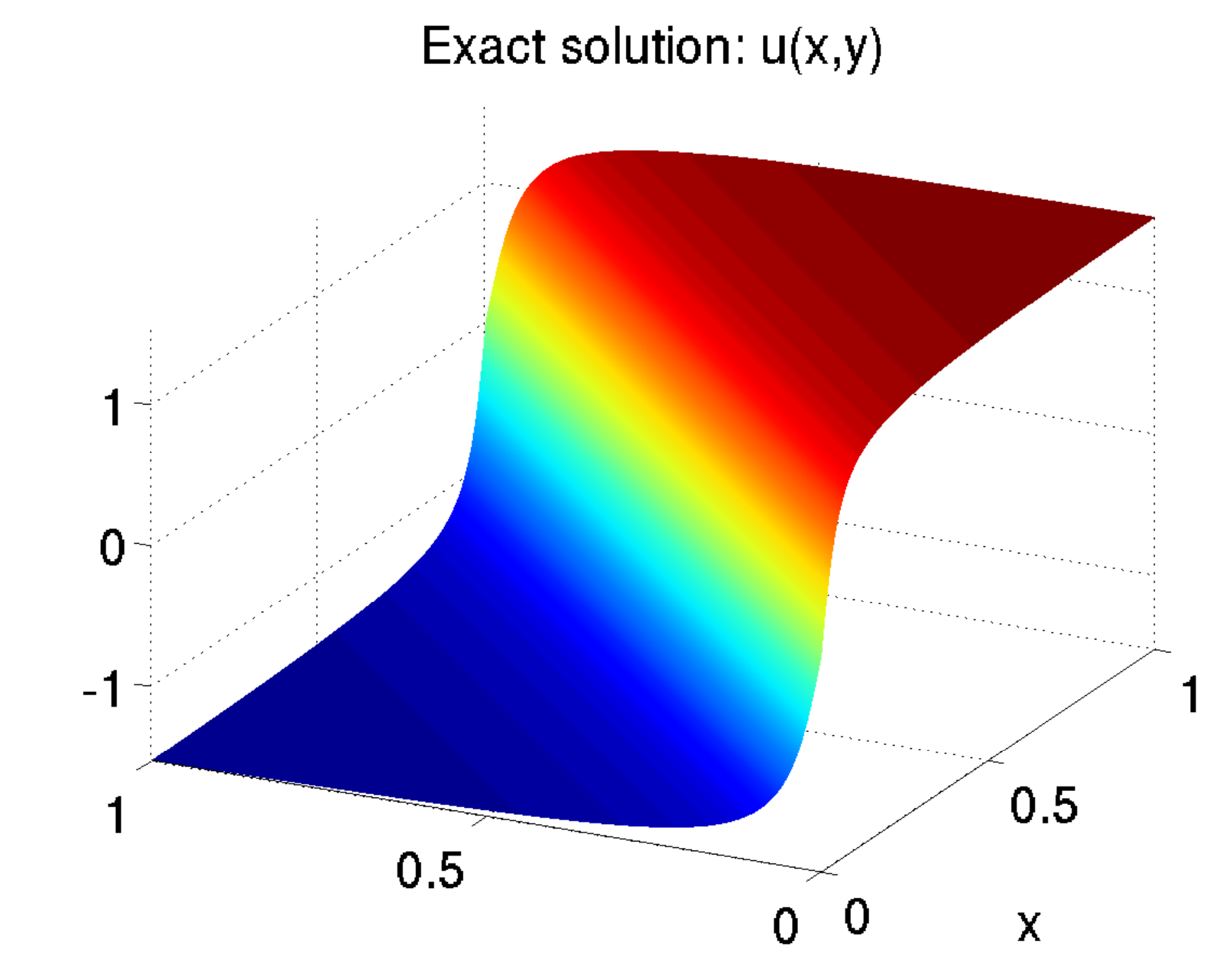}
\includegraphics[width=.2\textwidth]{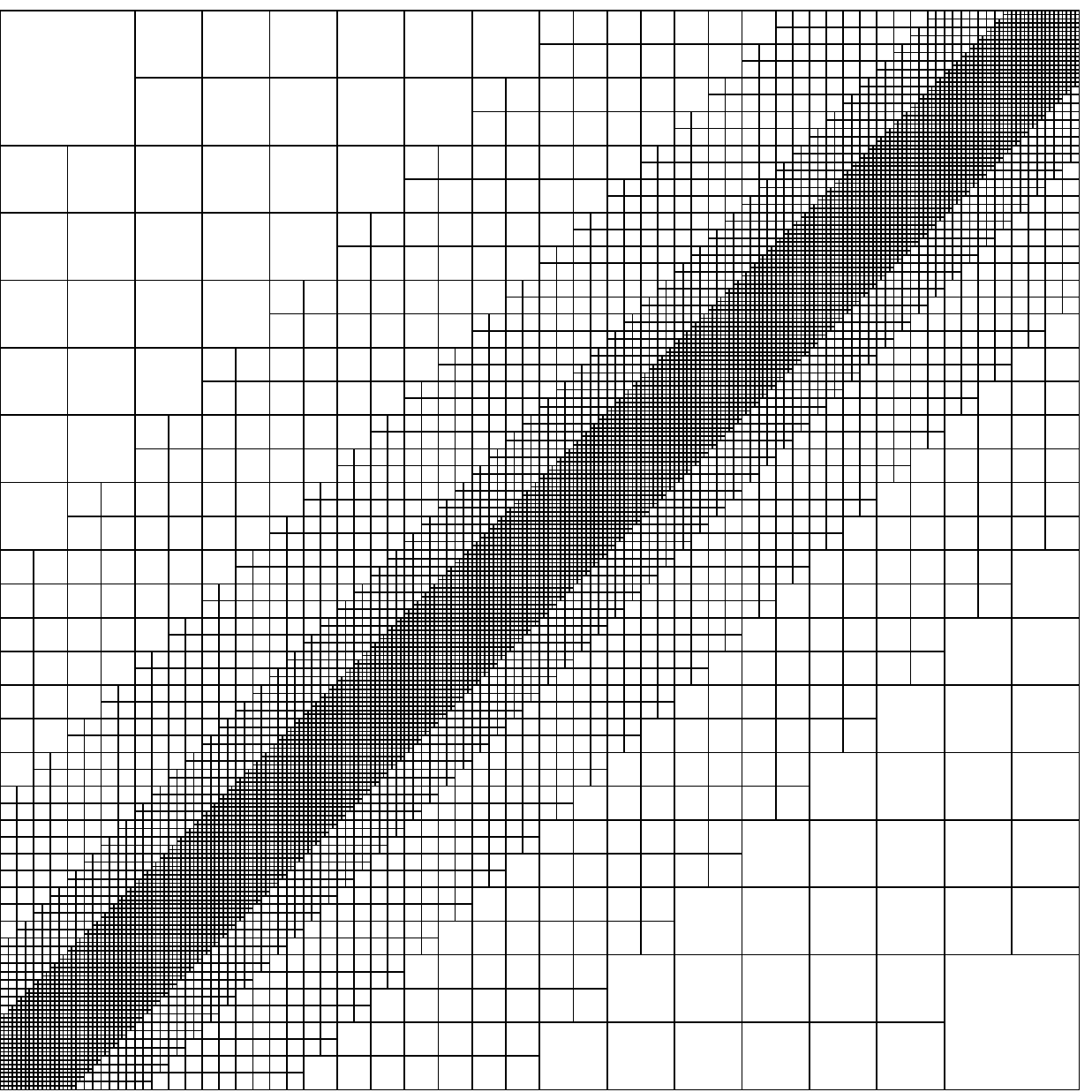}
\includegraphics[width=.2\textwidth]{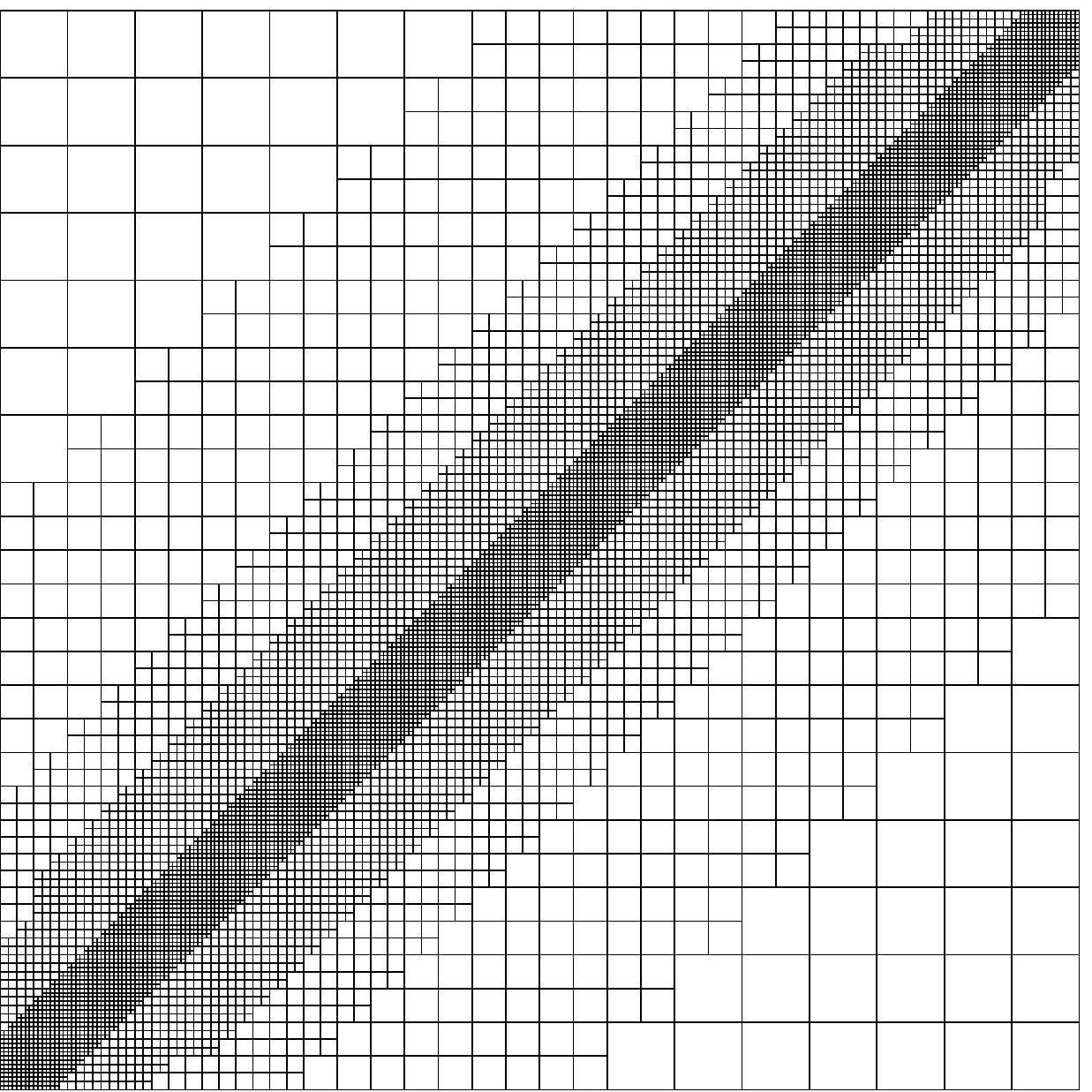} 
\includegraphics[width=.2\textwidth]{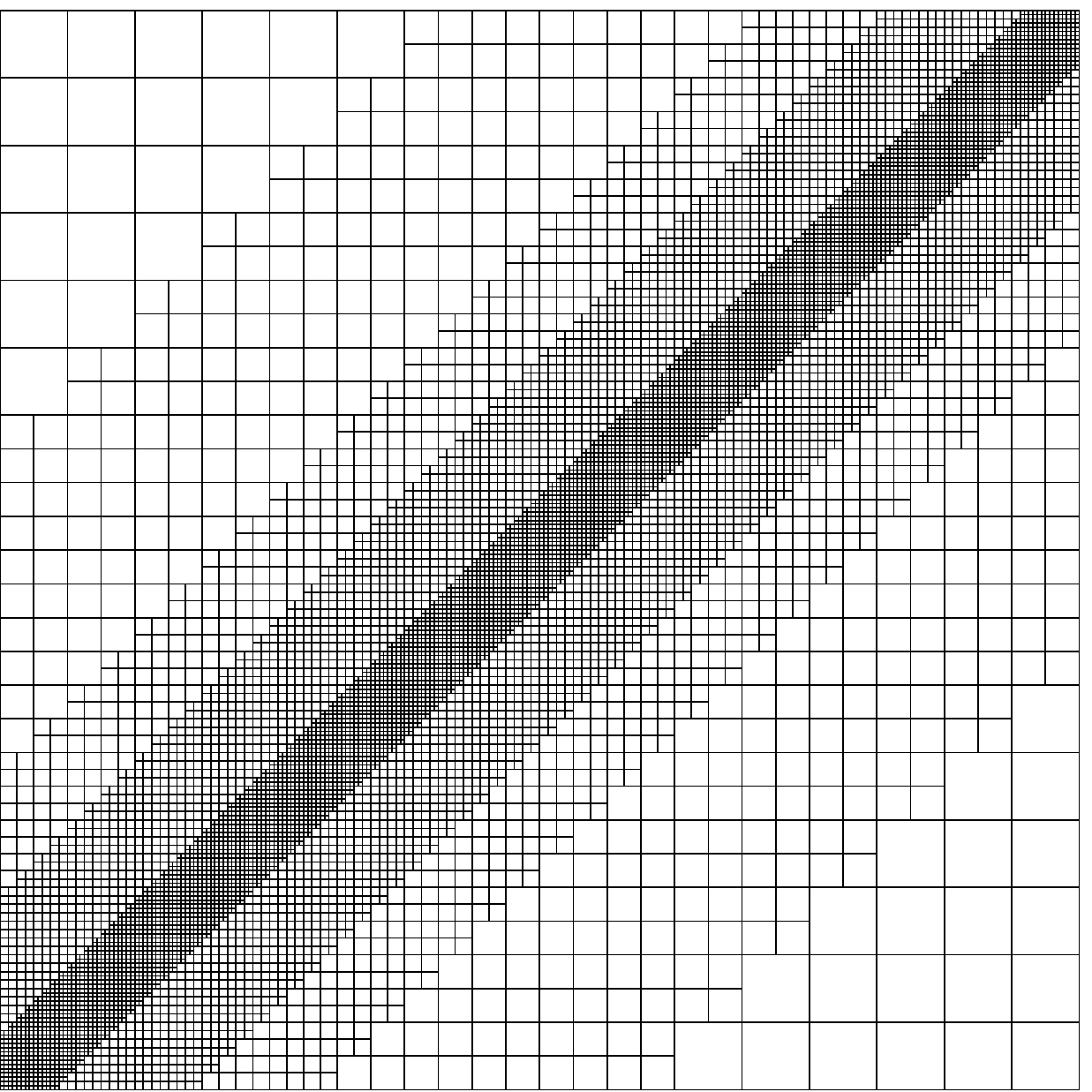}
\medskip 

\includegraphics[width=.3\textwidth]{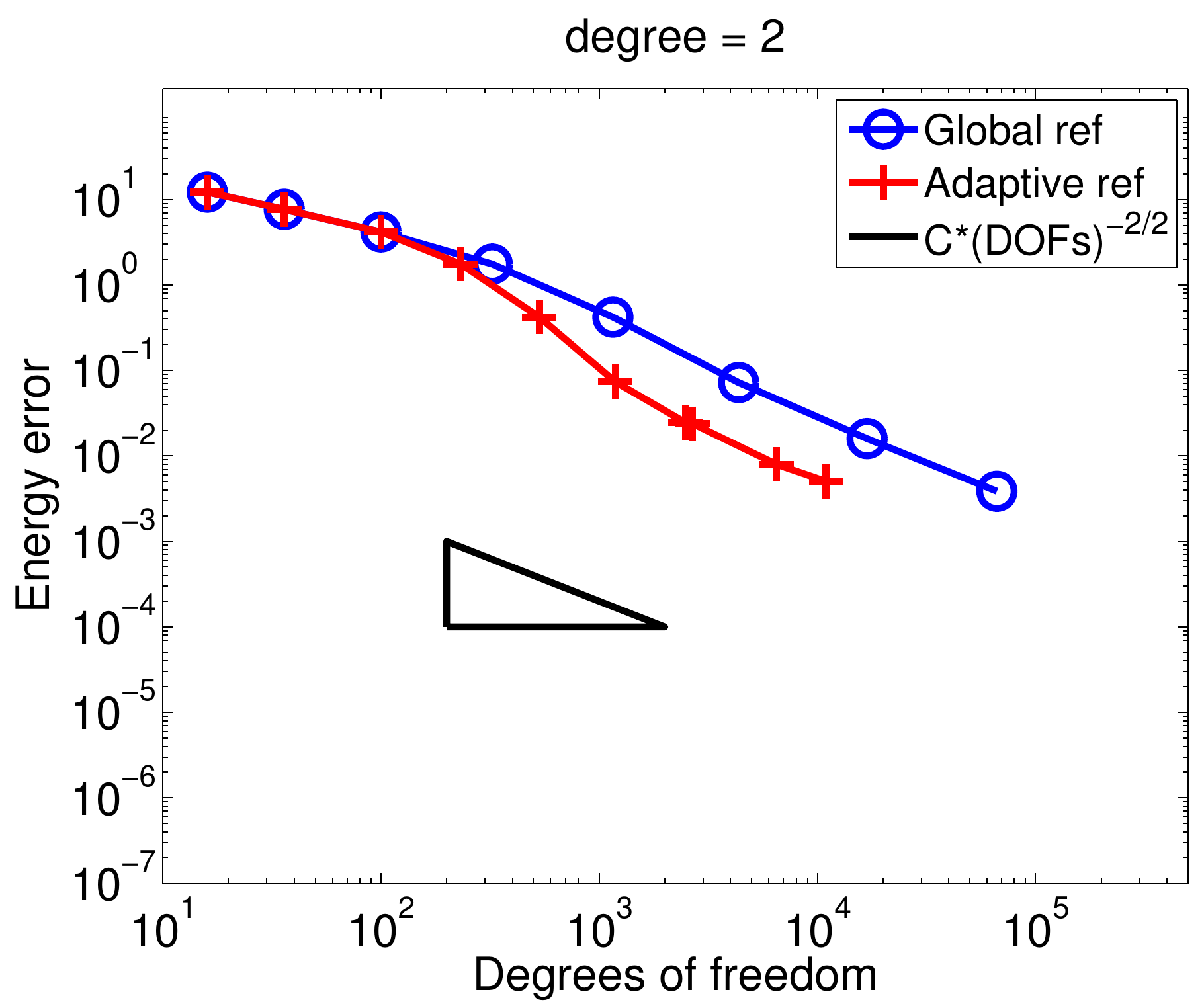}
\includegraphics[width=.3\textwidth]{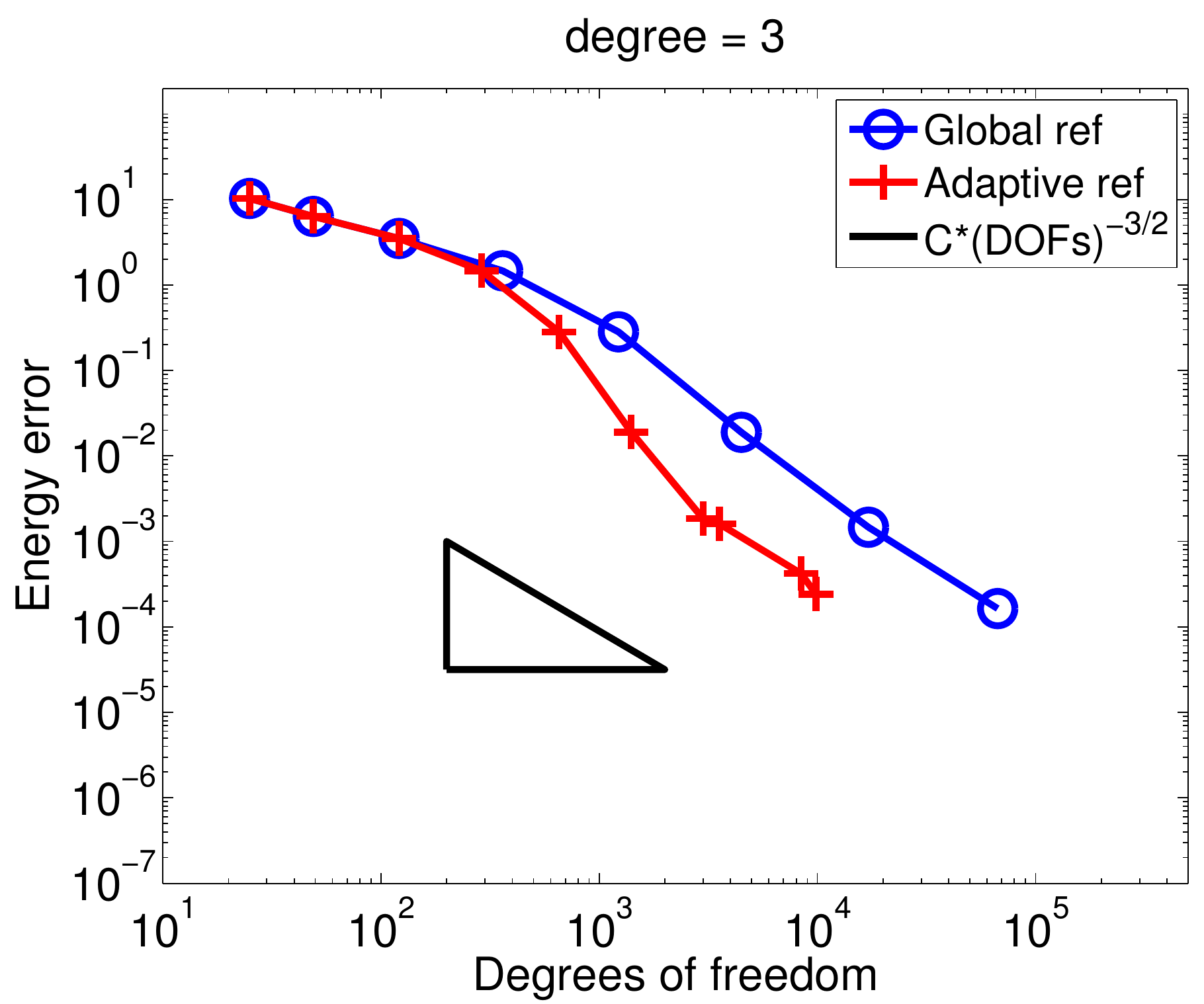}
\includegraphics[width=.3\textwidth]{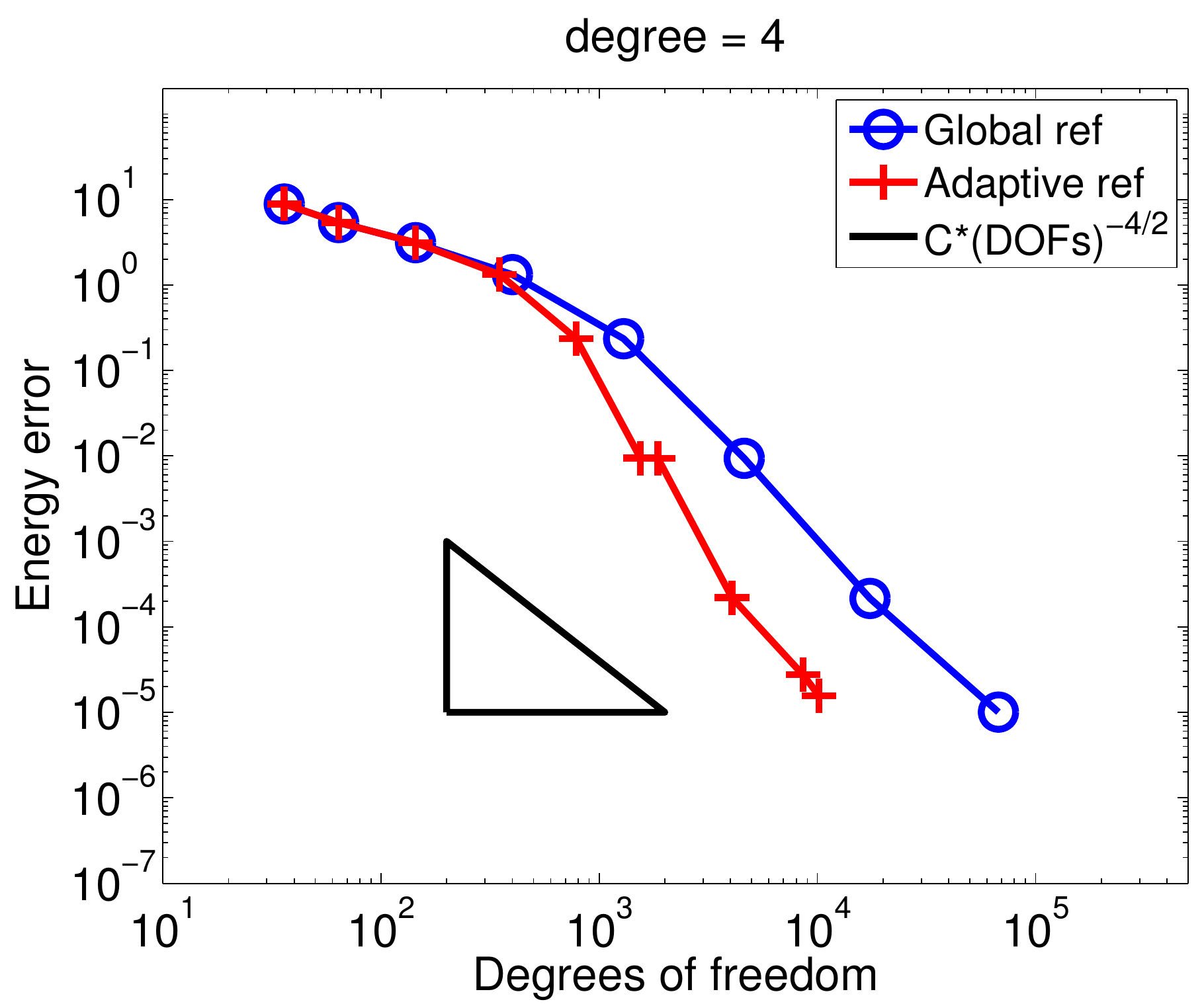}
\end{center}
\caption{\label{F:example diagonal} \small Some hierarchical meshes for the solution of Example~\ref{Ex:diagonal}; for biquadratics with $11578$ elements and $10958$ DOFs (top left), bicubics with $10792$ elements and $9865$ DOFs (top middle) and biquartics with $11338$ elements and $10186$ DOFs (top right). Notice that although all meshes have nearly the same amount of elements, the refinement is more spread for high order splines due to the sizes of the basis function supports. We plot the energy error $|u-U|_{H^1(\Omega)}$ vs. degrees of freedom; for biquadratics (bottom left), bicubics (bottom middle) and biquartics (bottom right).}
\end{figure}
\end{example}


\begin{example}[Singular domain: an L-shaped domain]\label{Ex:Lshaped}
We consider the L-shaped domain $\Omega=[-1,1]^2\setminus((0,1)\times(-1,0))$ and choose $f$ and $g$ such that the exact solution $u$ of~\eqref{E:poisson} is given in polar coordinates by $u(\rho,\varphi) = \rho^{2/3} \sin(2\varphi/3)$. In Figure~\ref{F:example L-shaped} we plot the exact solution, some hierarchical meshes and the decay of the energy error vs. degrees of freddom for different spline degrees. We notice that the global refinement associated to tensor product spaces does not reach the optimal order of convergence due to the singularity. On the other hand, the adaptive strategy recovers the optimal decay for the energy error given by $\OO((\#\text{DOFs})^{\frac{p}{2}})$.

\begin{figure}[H!tbp]
\begin{center}
\includegraphics[width=.3\textwidth]{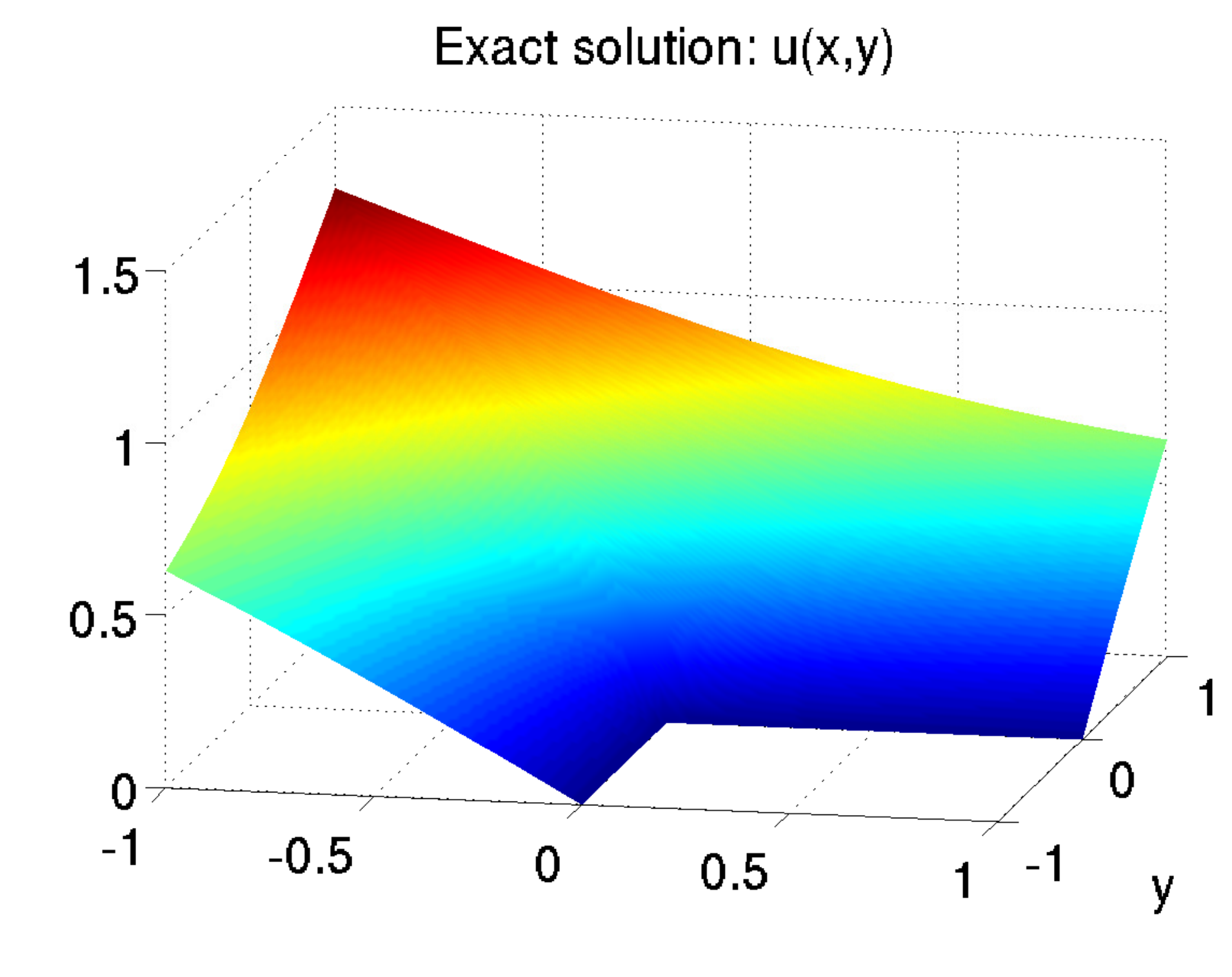}
\includegraphics[width=.2\textwidth]{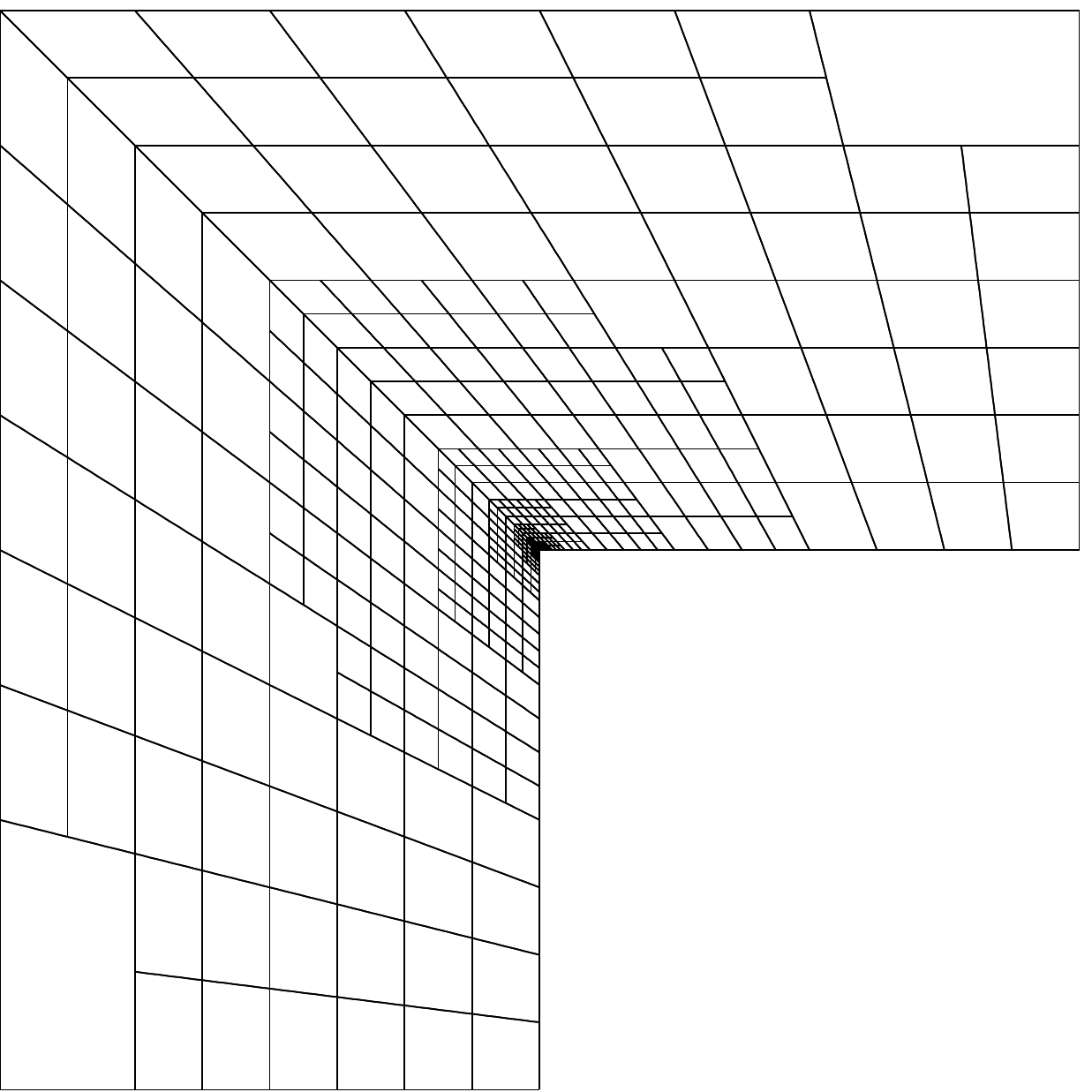}
\includegraphics[width=.2\textwidth]{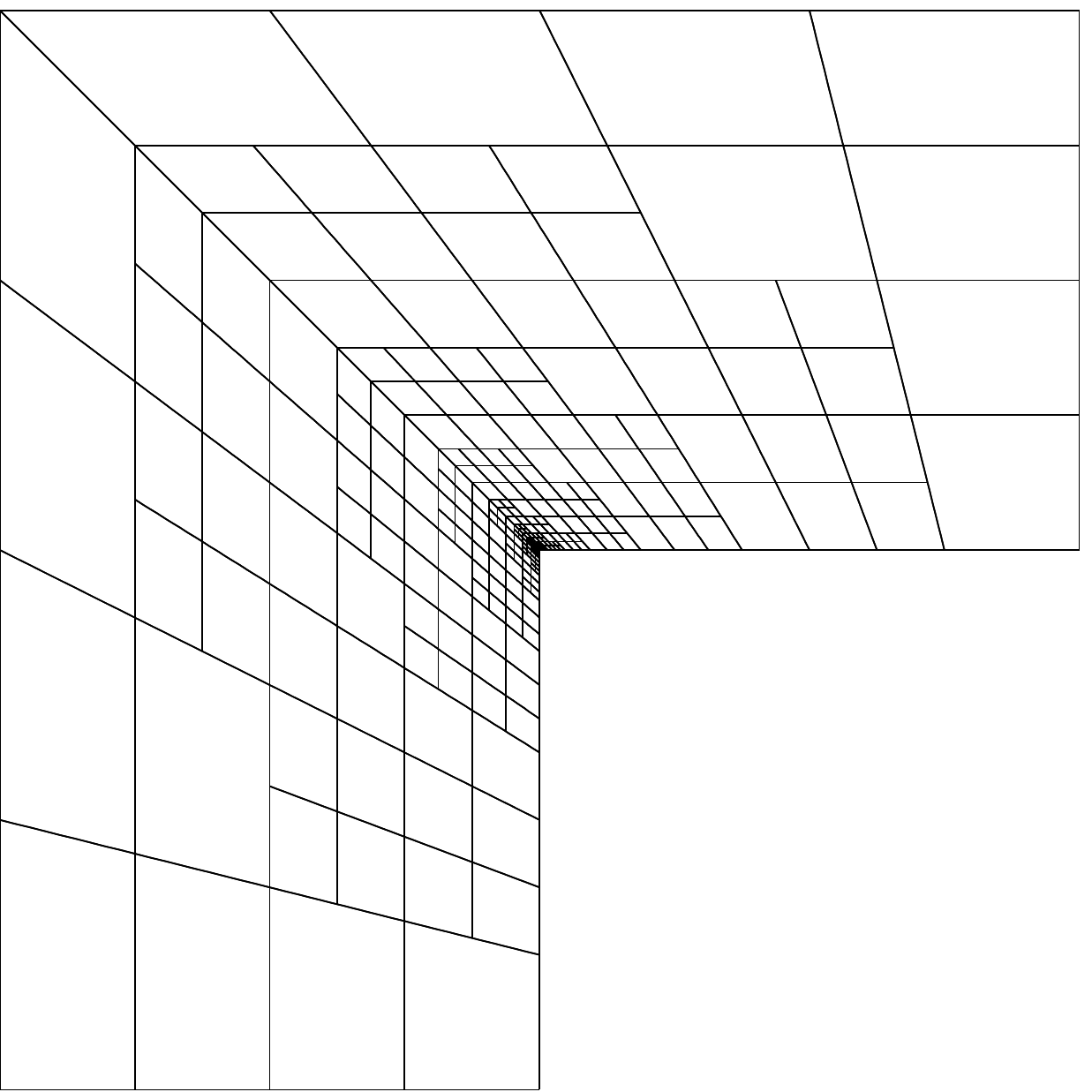} 
\includegraphics[width=.2\textwidth]{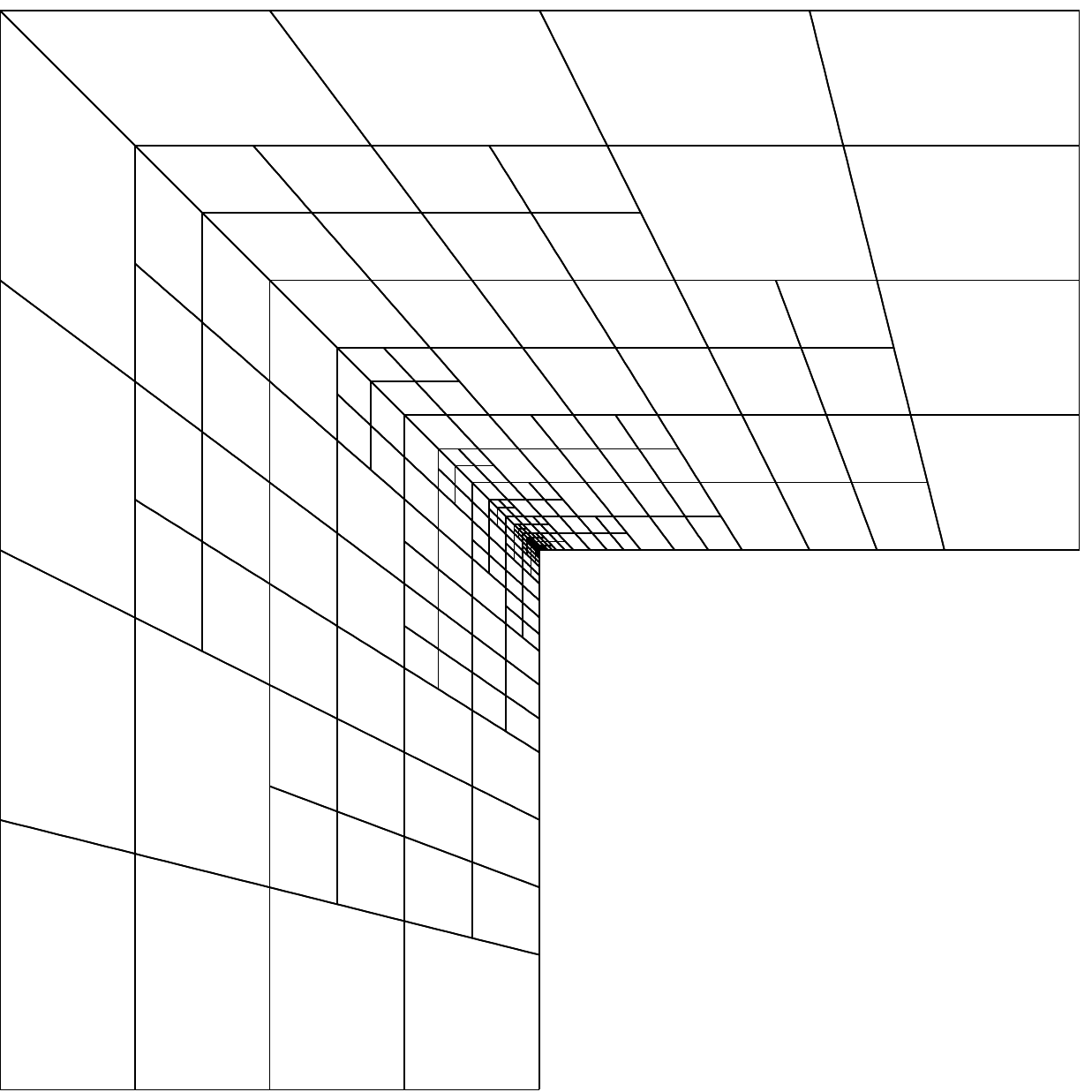}
\medskip

\includegraphics[width=.3\textwidth]{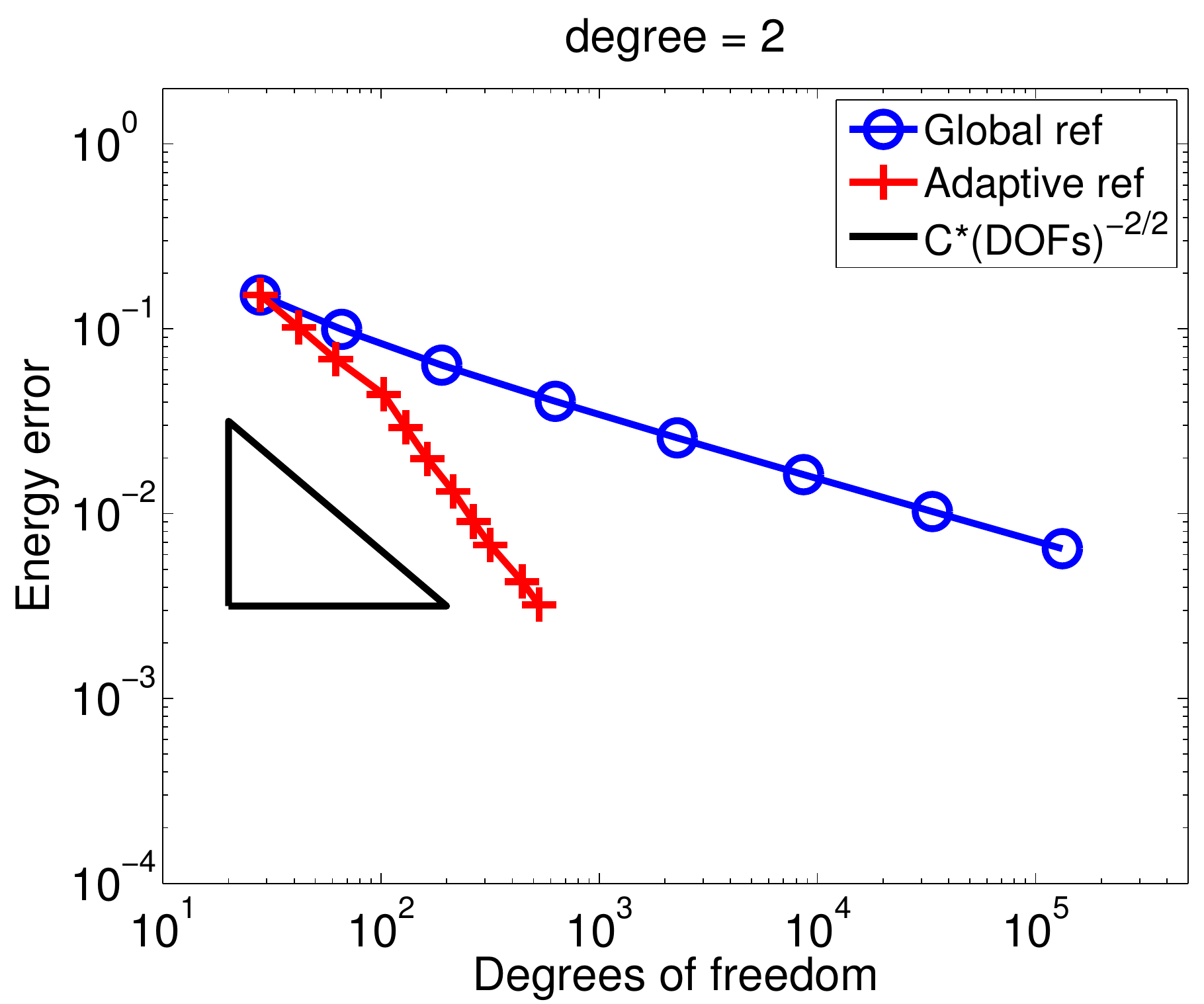}
\includegraphics[width=.3\textwidth]{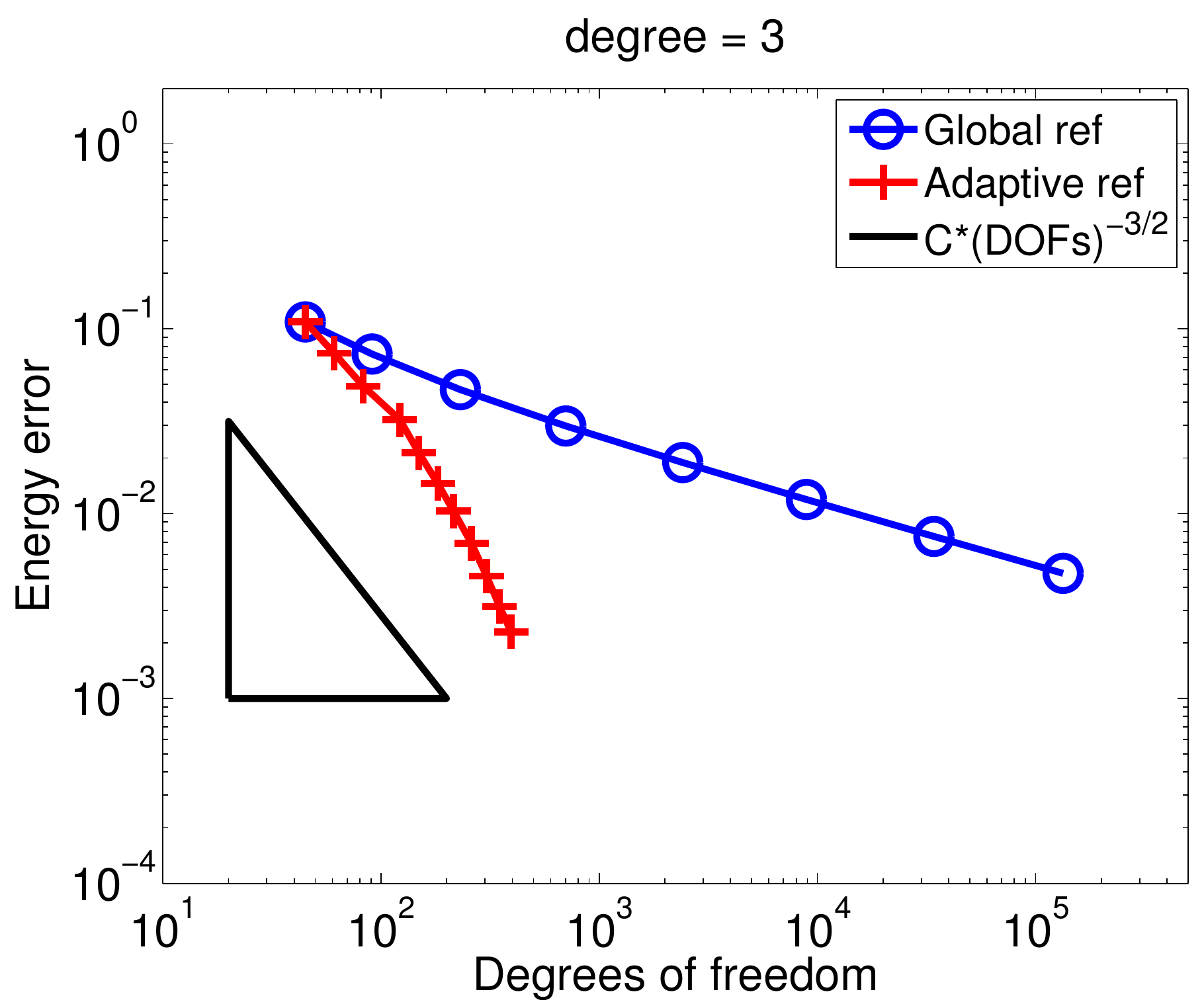}
\includegraphics[width=.3\textwidth]{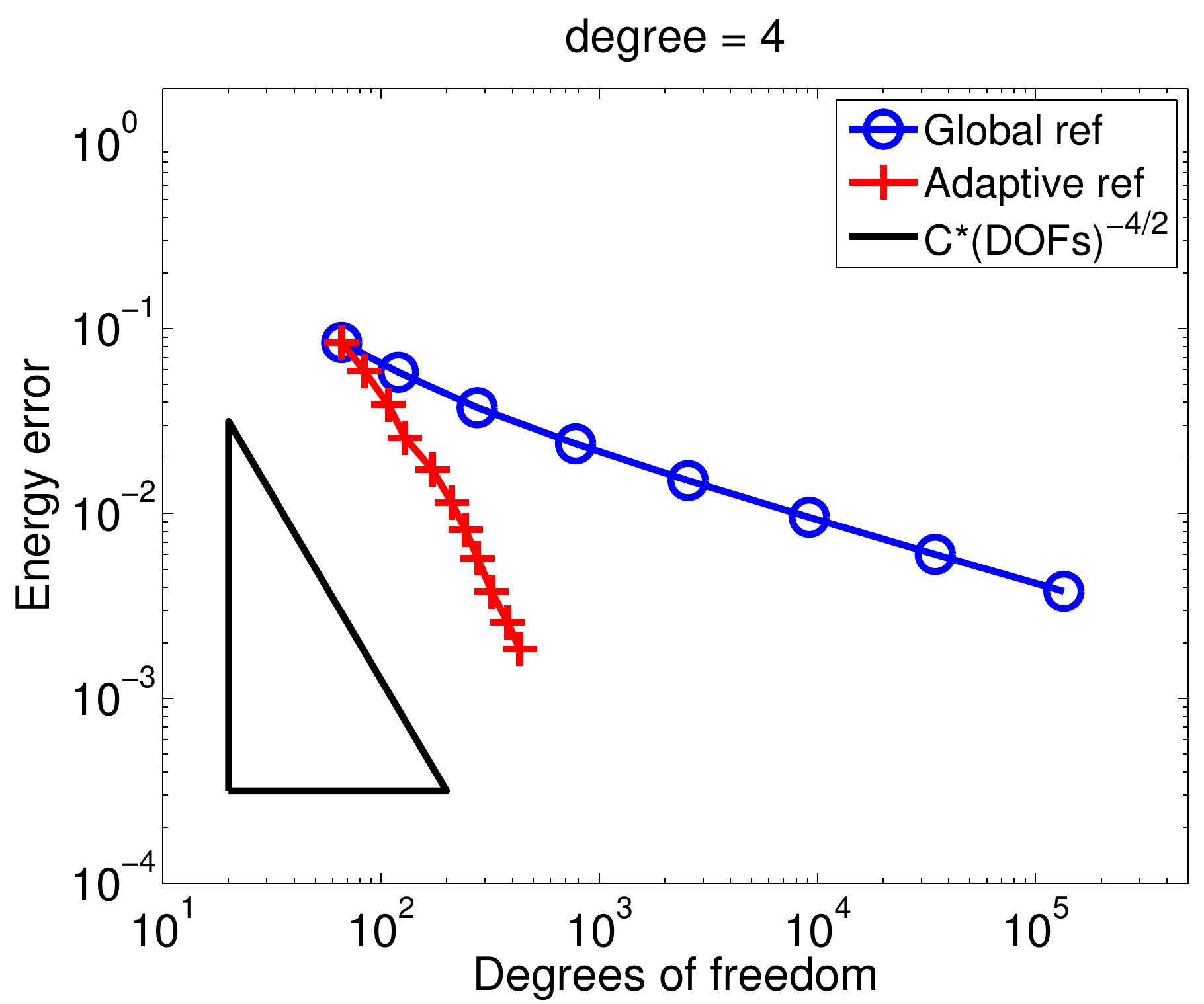}
\end{center}
\caption{\label{F:example L-shaped} \small  Some hierarchical meshes for the solution of Example~\ref{Ex:Lshaped} obtaining in all cases $|u-U|_{H^1(\Omega)}\approx 3.10^{-3}$; for biquadratics with $500$ elements and $530$ DOFs (top left), bicubics with $314$ elements and $350$ DOFs (top middle) and biquartics with $254$ elements and $322$ DOFs (top right). In addition, we plot the energy error $|u-U|_{H^1(\Omega)}$ vs. degrees of freedom; for biquadratics (bottom left), bicubics (bottom middle) and biquartics (bottom right).}
\end{figure}
\end{example}


\begin{example}[Singular solution in the unit square]\label{Ex:singular solution}
We consider a problem whose solution is not too smooth. Specifically, we take $\Omega = [0,1]\times[0,1]$, and choose $g\equiv 0$ and $f$ such that the exact solution $u$ of~\eqref{E:poisson} is given by $u(x,y) = x^{2.3}(1-x)y^{2.9}(1-y)$. Notice that in this case, $u\in H^2(\Omega)\setminus H^3(\Omega)$ and there are singularities along the sides $x=0$ and $y=0$, being a bit stronger the singularity along $x=0$; see Figure~\ref{F:example singular solution}. Some hierarchical meshes and the error decay in terms of degrees of freddom for different polynomial degrees are presented in Figure~\ref{F:example singular solution meshes}. We notice that both global refinement and the adaptive refinement reach the optimal order of convergence when using biquadratics (bottom left), but only the adaptive refinement converges with optimal rates when using bicubics (bottom middle) and biquartics (bottom right), due to the singularity of the solution.

\begin{figure}[H!tbp]
\begin{center}
\includegraphics[width=.3\textwidth]{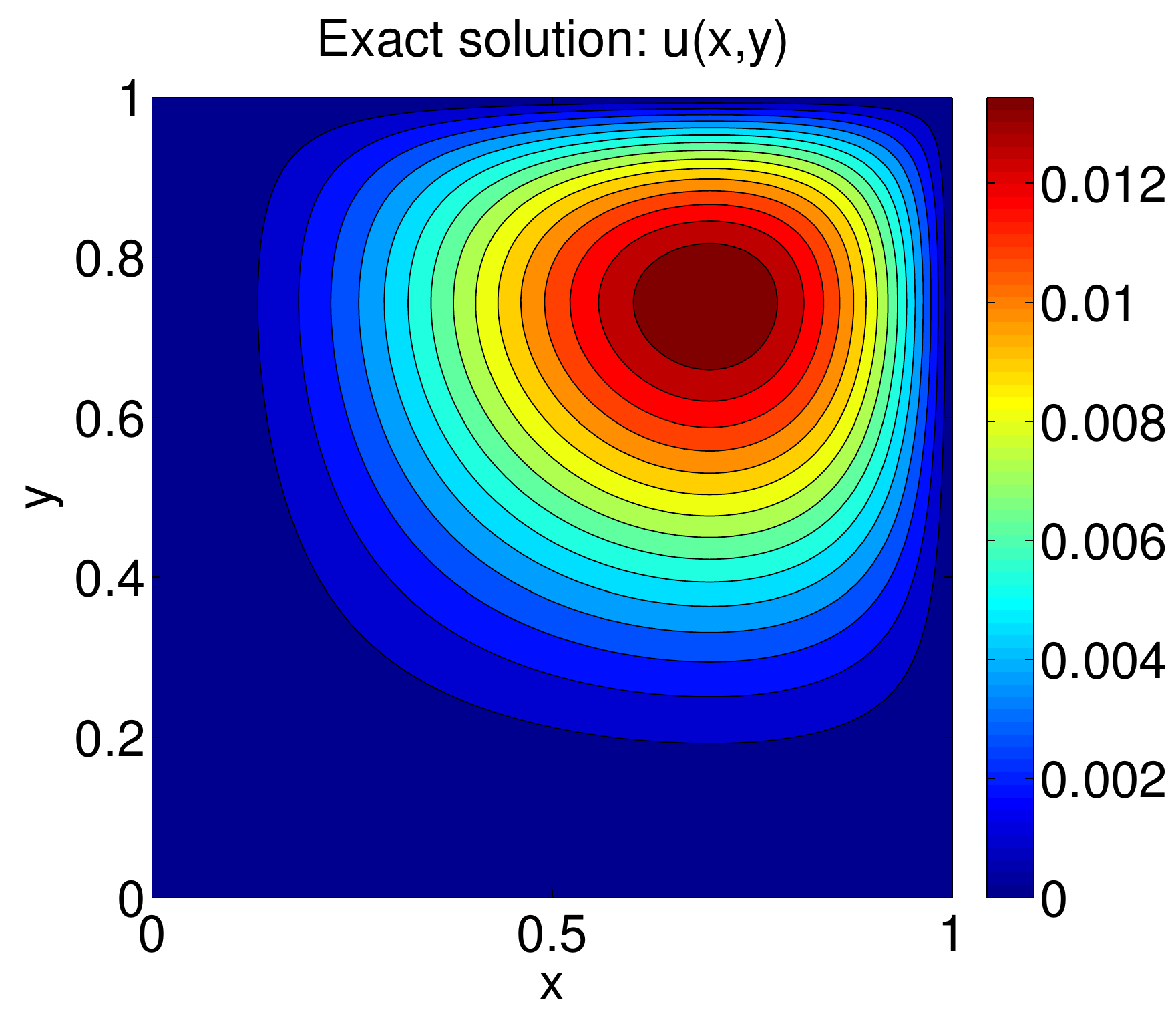} 
\includegraphics[width=.3\textwidth]{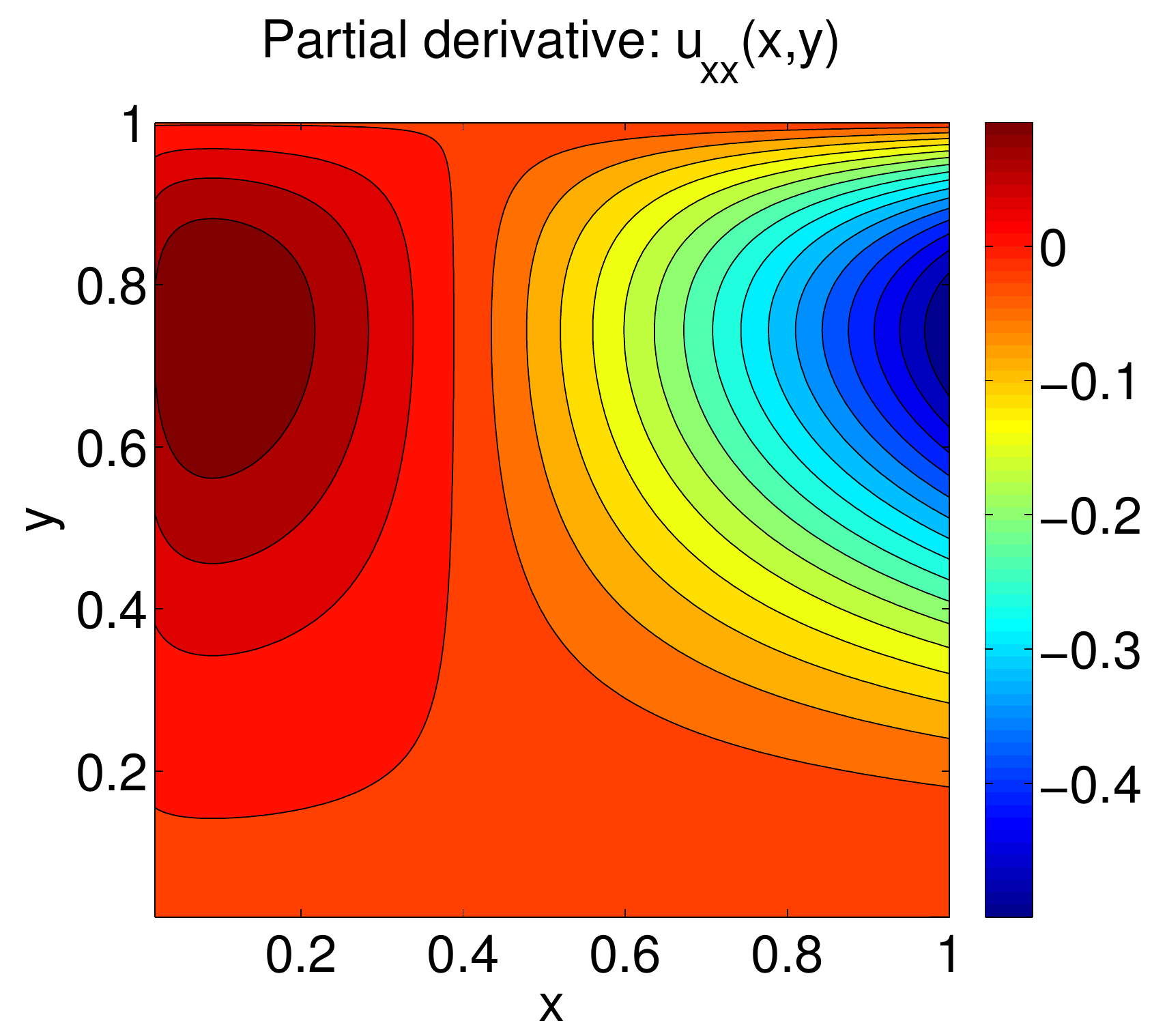} 
\includegraphics[width=.3\textwidth]{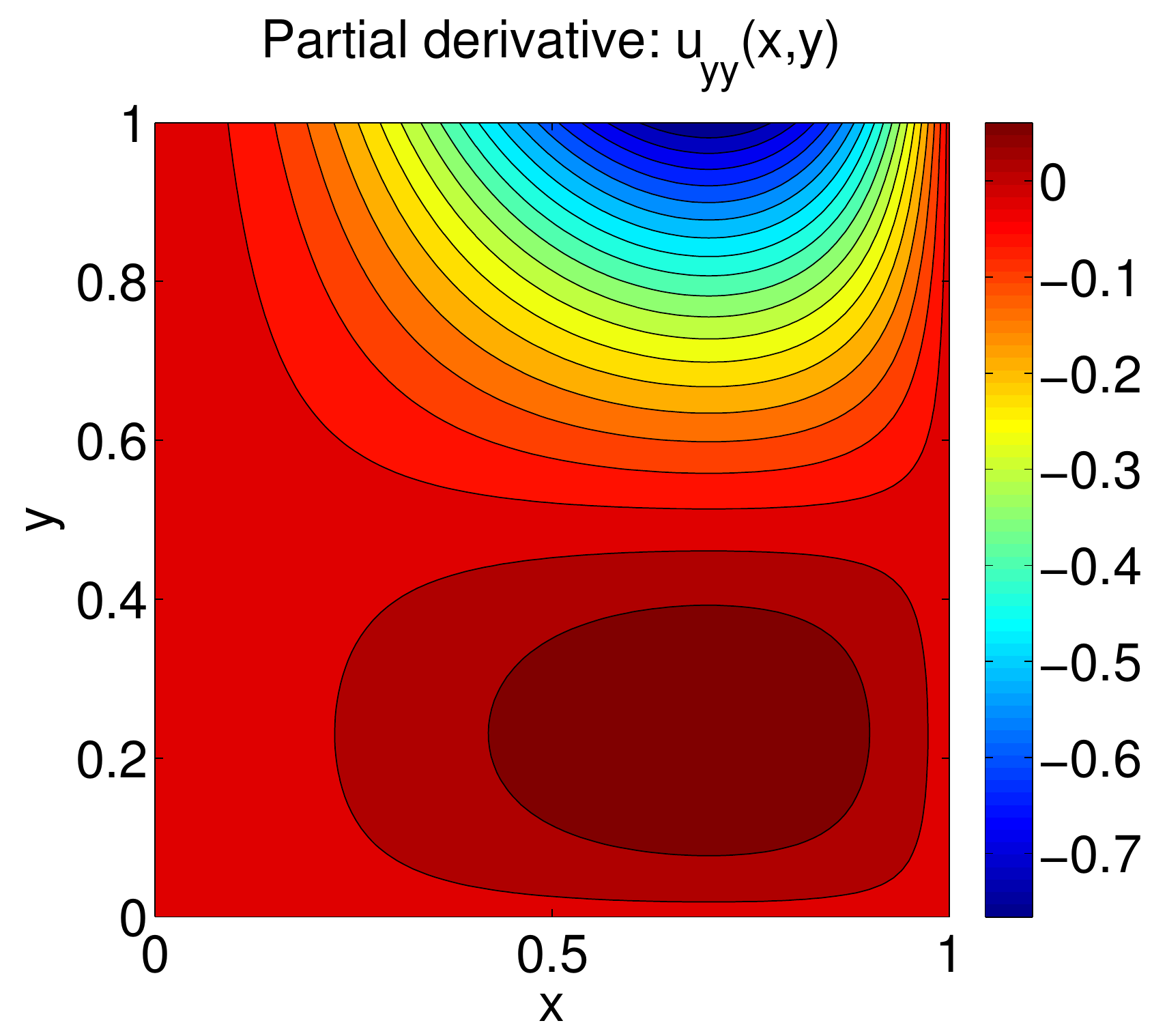} 
\end{center}
\caption{\label{F:example singular solution} \small The exact solution $u(x,y) = x^{2.3}(1-x)y^{2.9}(1-y)$ (left) corresponding to Example~\ref{Ex:singular solution} and its derivatives $u_{xx}$ (middle) and $u_{yy}$ (right).}
\end{figure}

\begin{figure}[H!tbp]
\begin{center}
\includegraphics[width=.3\textwidth]{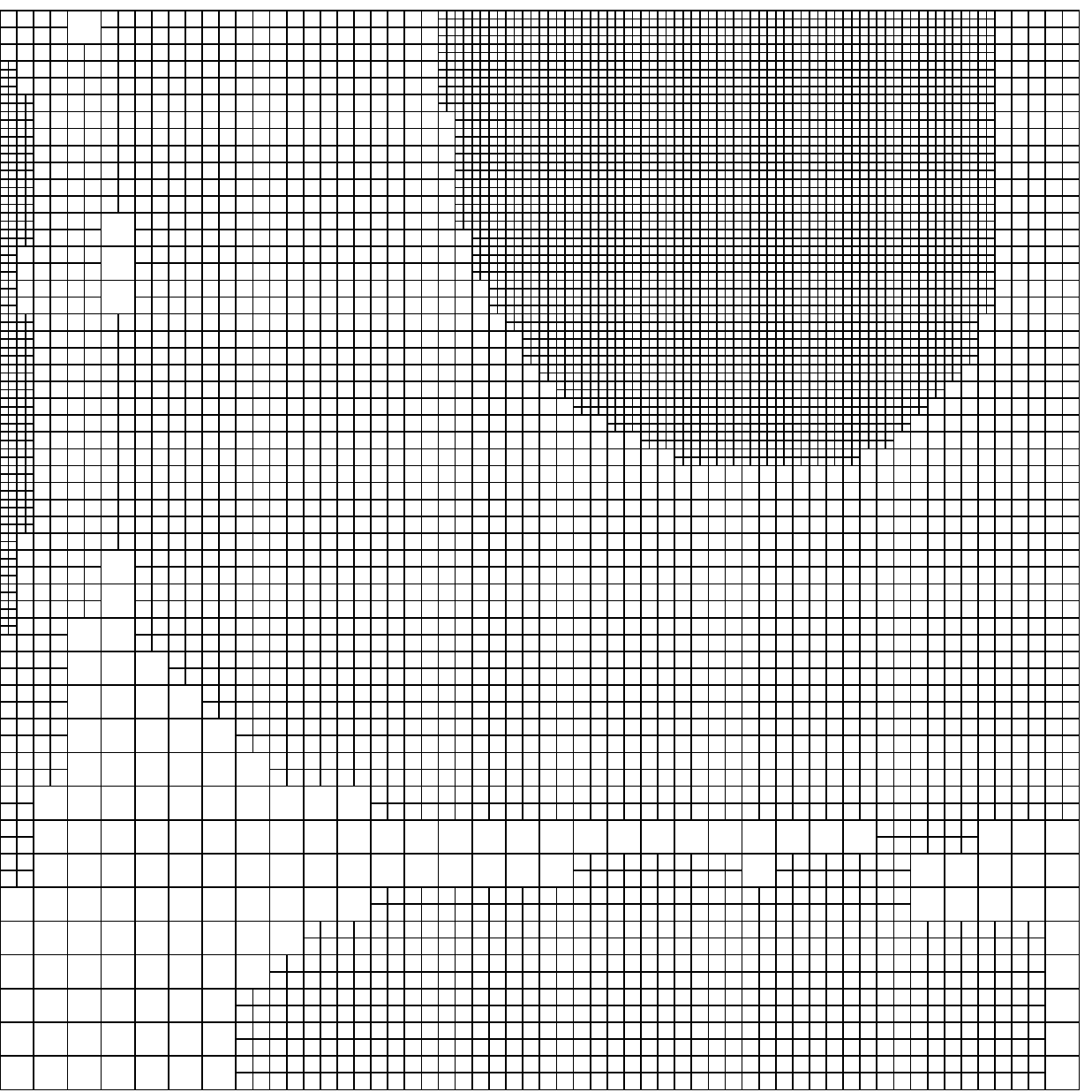}
\includegraphics[width=.3\textwidth]{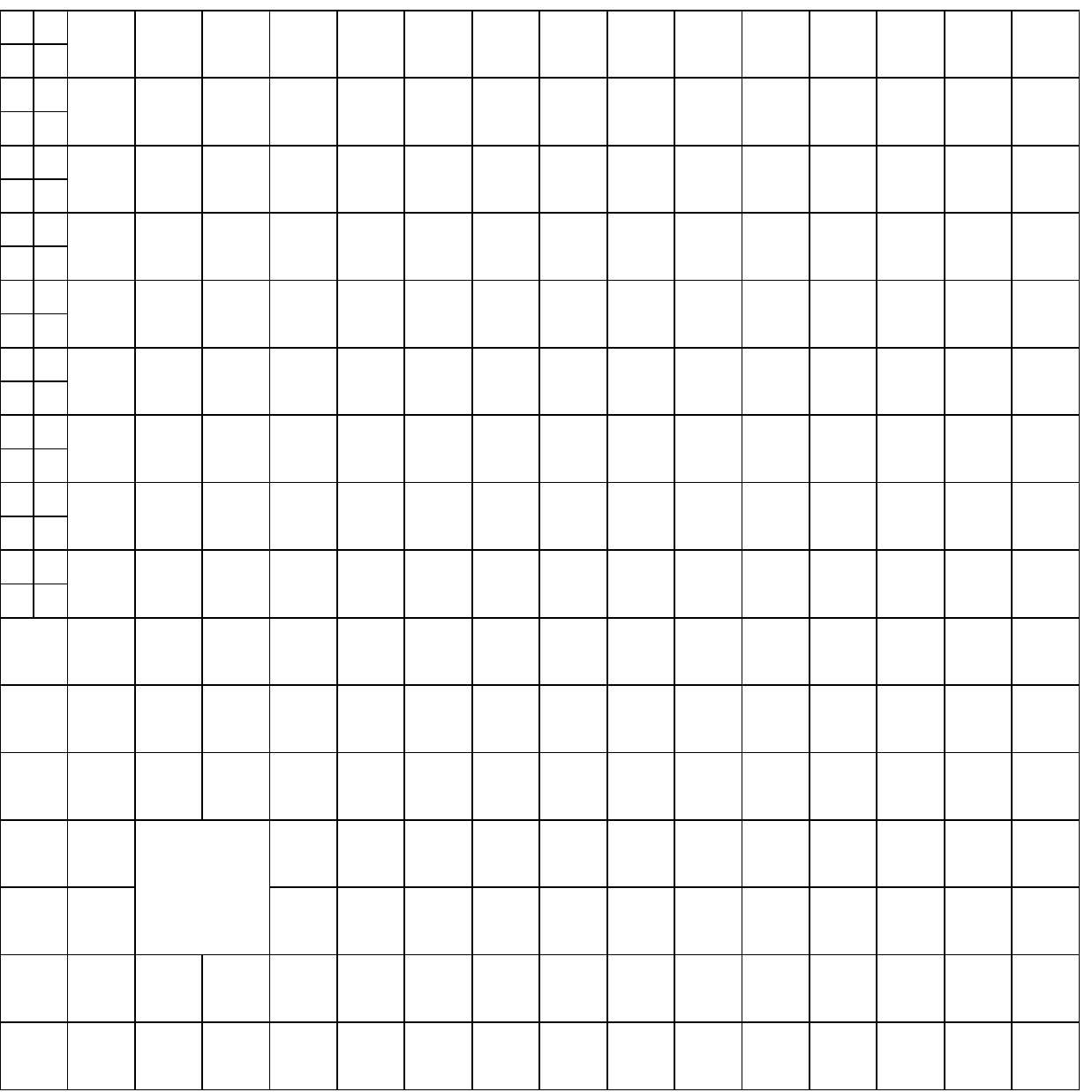} 
\includegraphics[width=.3\textwidth]{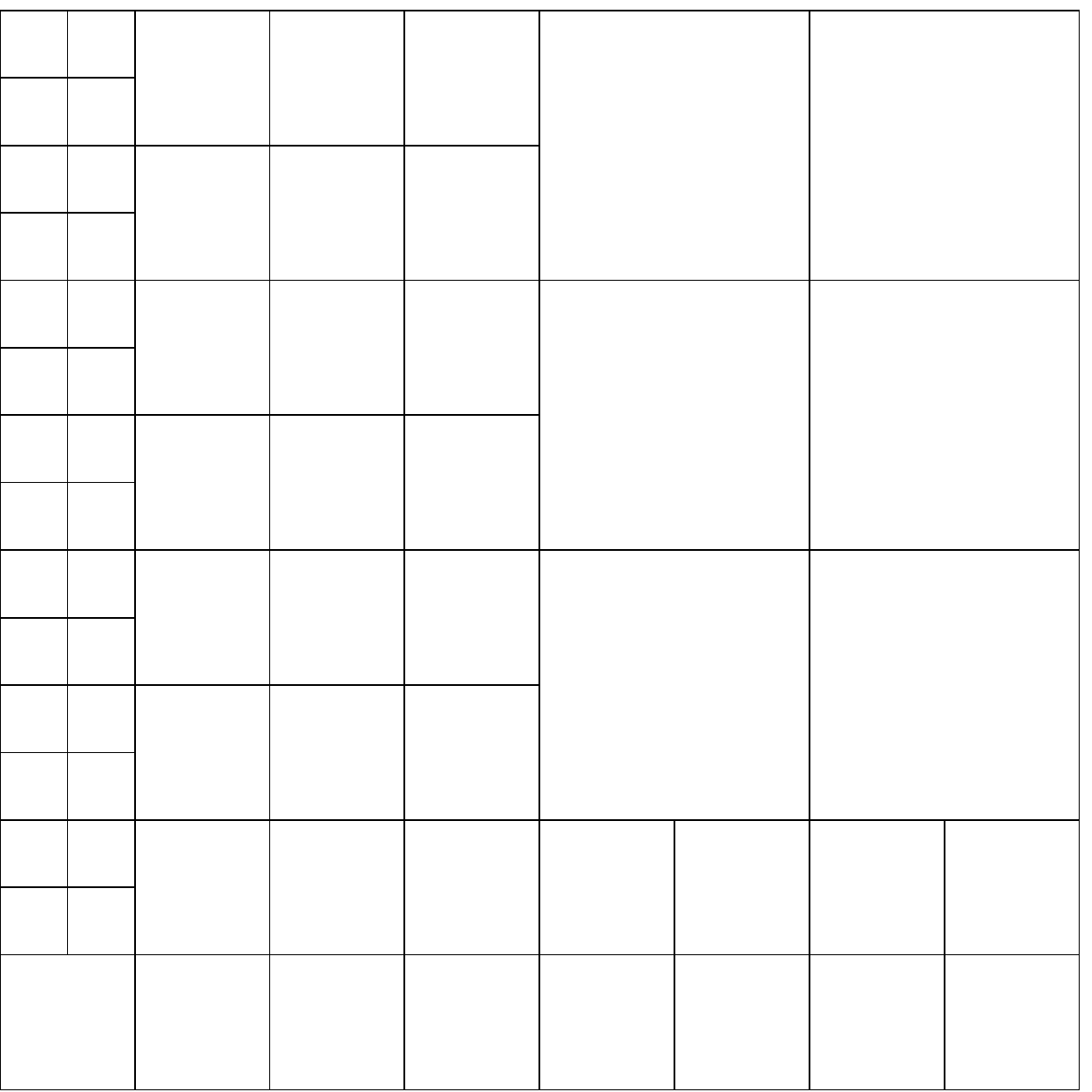}
\medskip

\includegraphics[width=.3\textwidth]{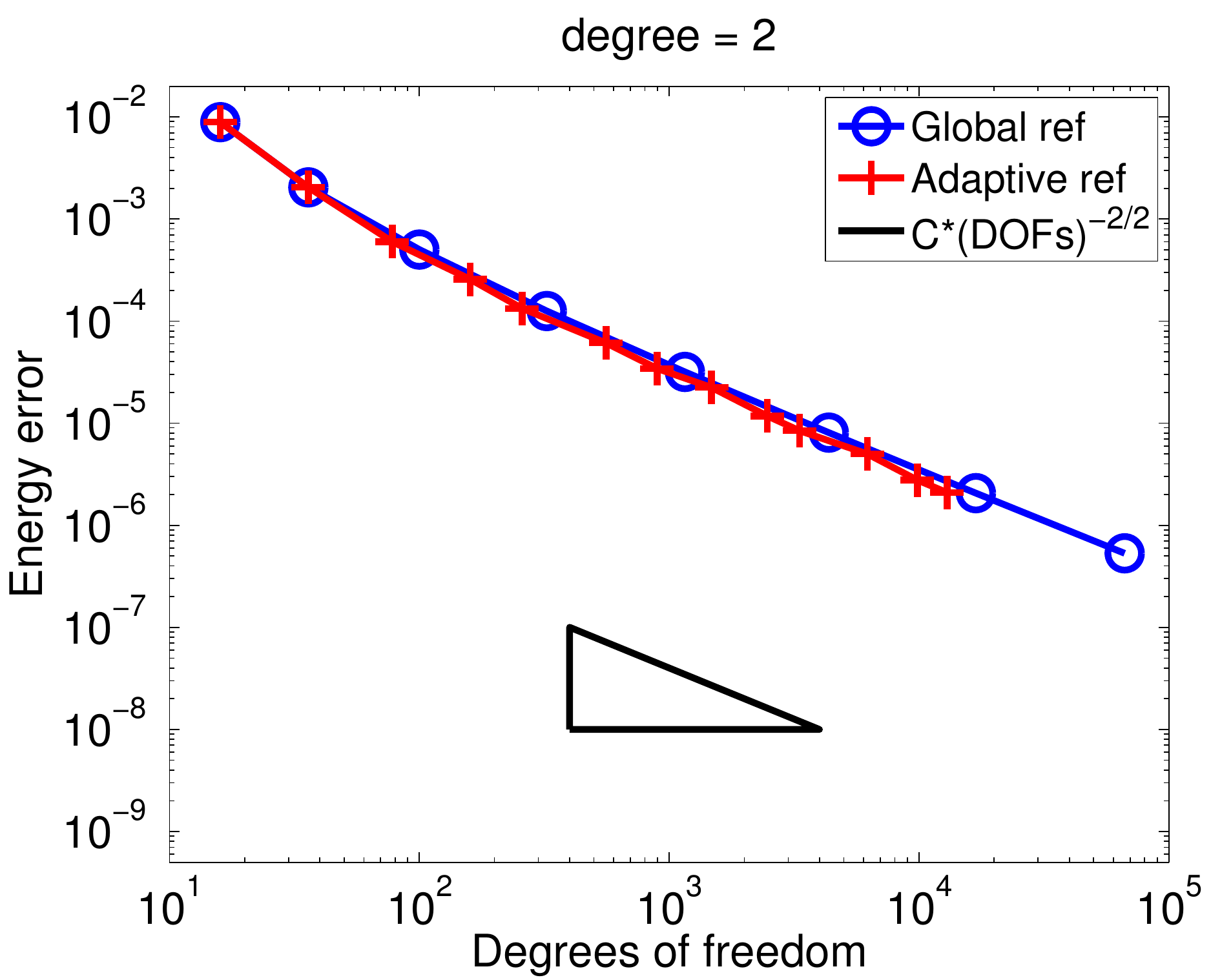}
\includegraphics[width=.3\textwidth]{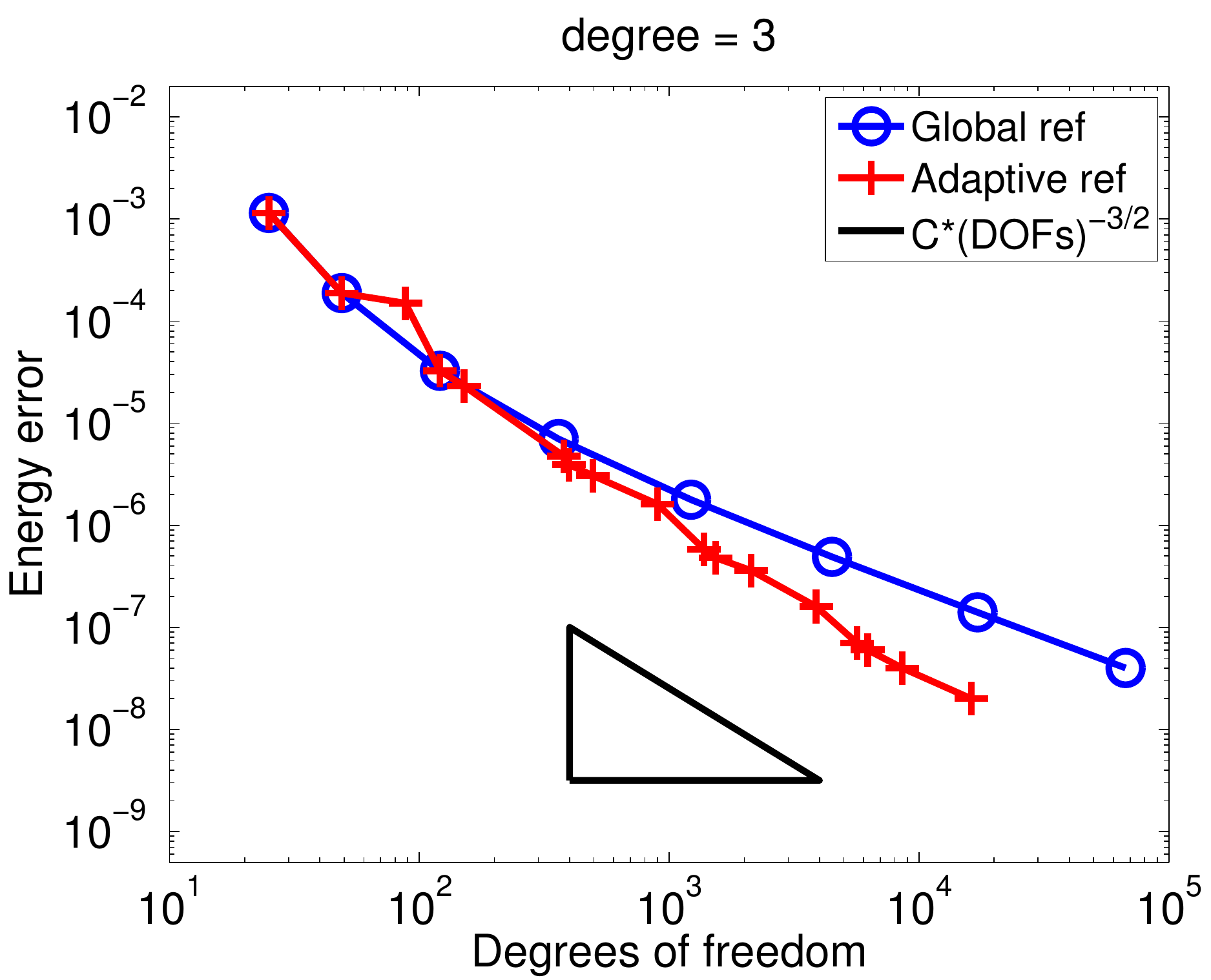}
\includegraphics[width=.3\textwidth]{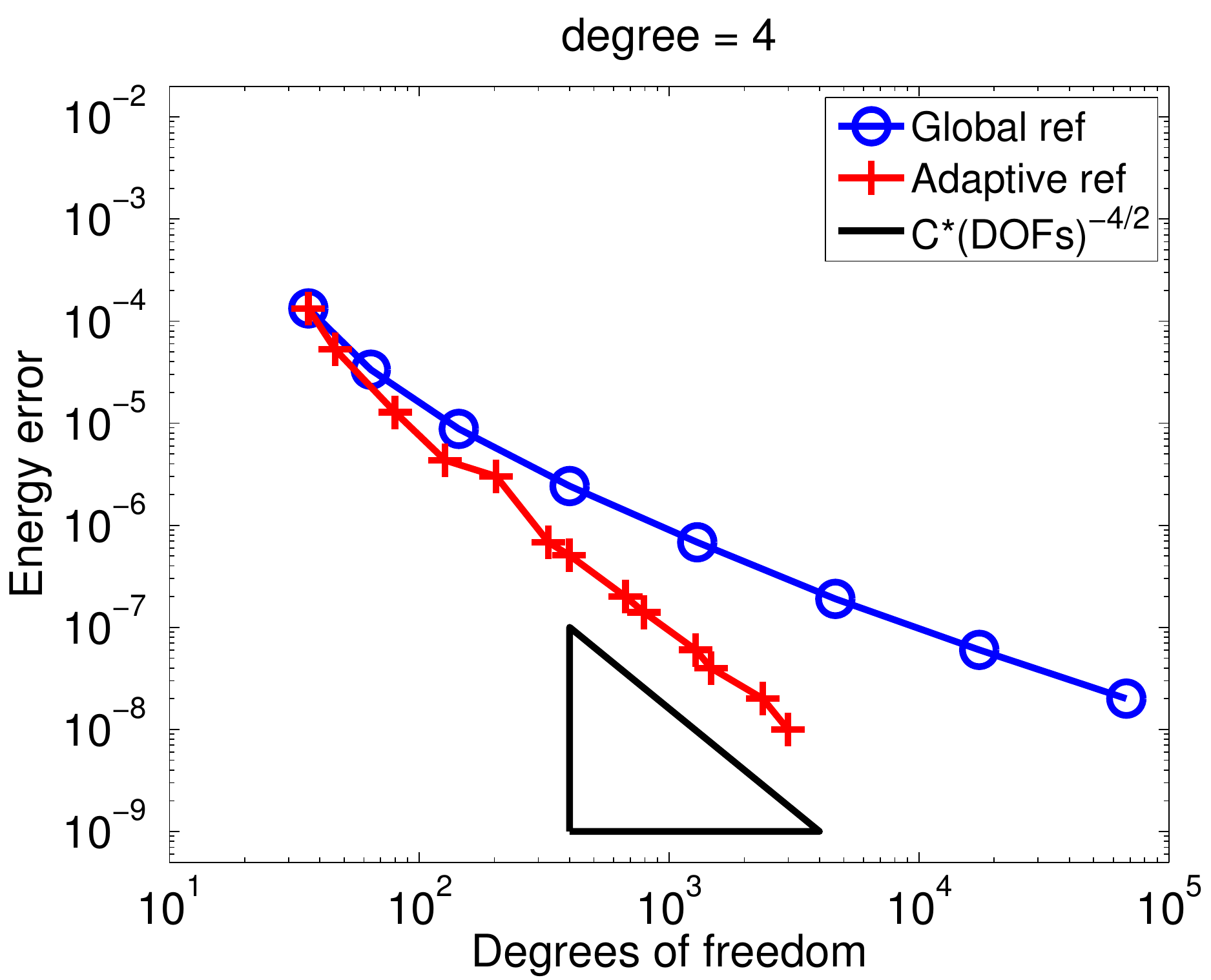}
\end{center}
\caption{\label{F:example singular solution meshes} \small Comparison of adaptive meshes for the solution of Example~\ref{Ex:singular solution} for different polynomial degrees, obtaining in all cases, $|u-U|_{H^1(\Omega)}\approx 5.10^{-6}$. The mesh for biquadratics has $6139$ elements and $6213$ DOFs (top left), the mesh for bicubics has $280$ elements and $379$ DOFs (top middle) and the mesh for biquartics has $67$ elements and $127$ DOFs (top right). On the other hand, we plot the energy error $|u-U|_{H^1(\Omega)}$ vs. degrees of freedom for different polynomials degrees.}
\end{figure}
\end{example}


\begin{example}[A physical domain: a quarter of ring]\label{Ex:ring}
 In this case, we consider the domain $\Omega$ given in polar coordinates by $\Omega=\{(\rho,\varphi)\mid 1\le \rho\le 2\quad\wedge\quad 0\le\varphi\le \frac{\pi}{2}\}$ and we choose the problem data $f$ and $g$ in~\eqref{E:poisson} such that the exact solution $u$ is given by $u(x,y) = e^{-100((x-\frac12)^2+(y-\frac12)^2)}$. Despite optimal rates of convergence are reached using both tensor product meshes and hierarchical meshes (see Figure~\ref{F:example ring}), we emphasize that in this case the adaptive strategy is still convenient. As an example, we notice that to get an energy error of $2.10^{-8}$ using bicubics, the adaptive strategy requieres less than the $3\%$ of the degrees of freedom utilised by the global refinement, because the former procedure requires $1900$ DOFs whereas the latter needs $67081$ DOFs.

\begin{figure}[H!tbp]
\begin{center}
\includegraphics[width=.3\textwidth]{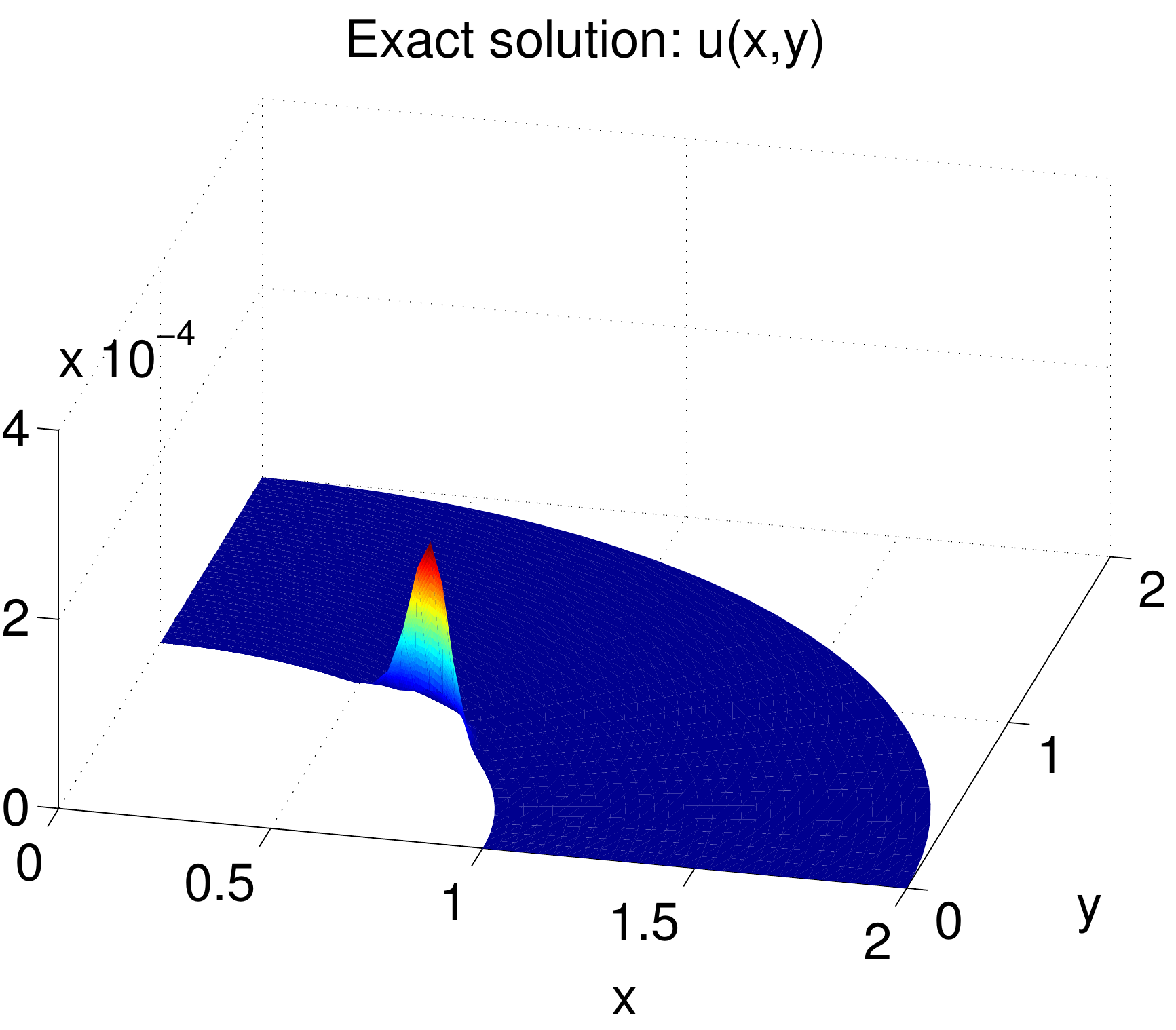}
\includegraphics[width=.2\textwidth]{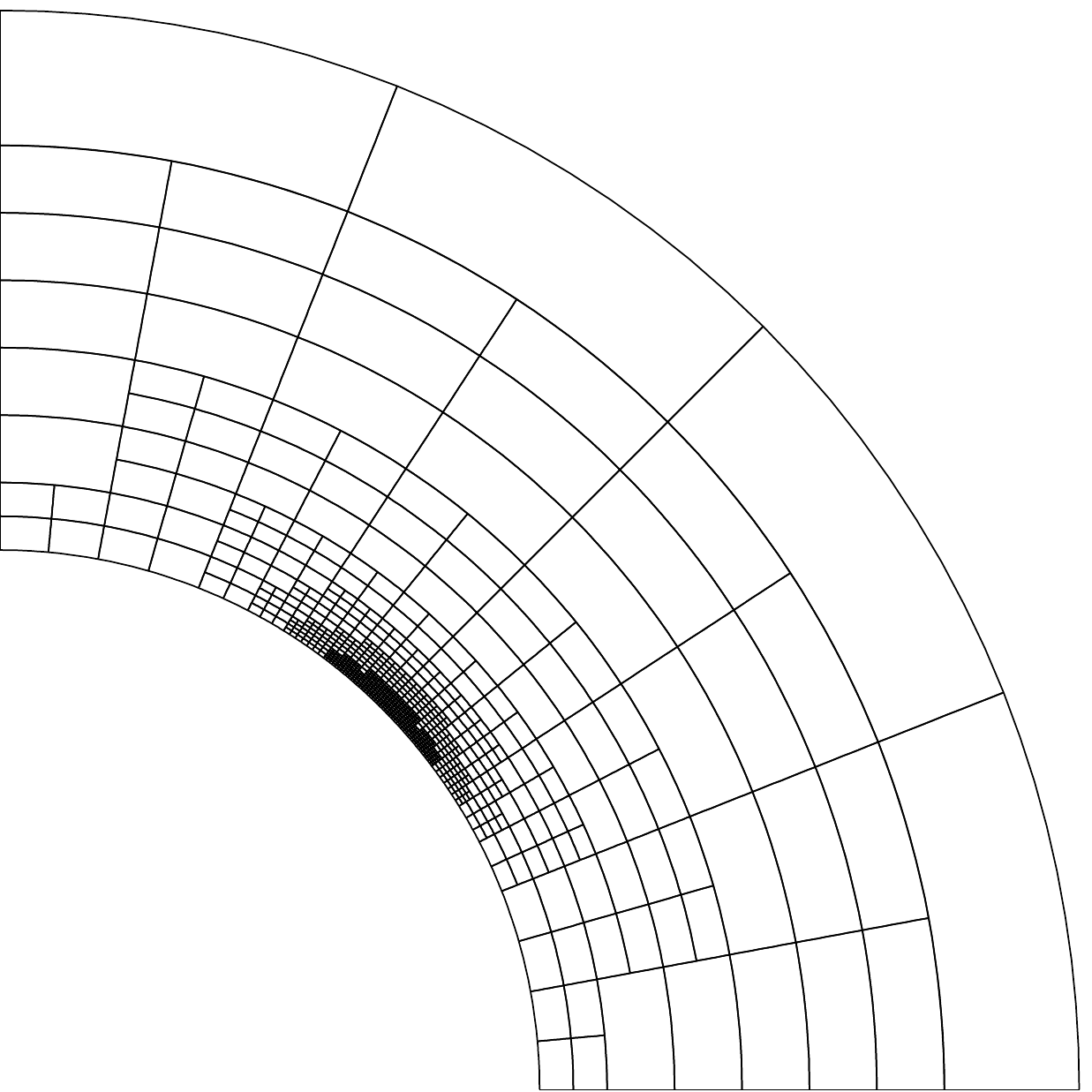}
\includegraphics[width=.2\textwidth]{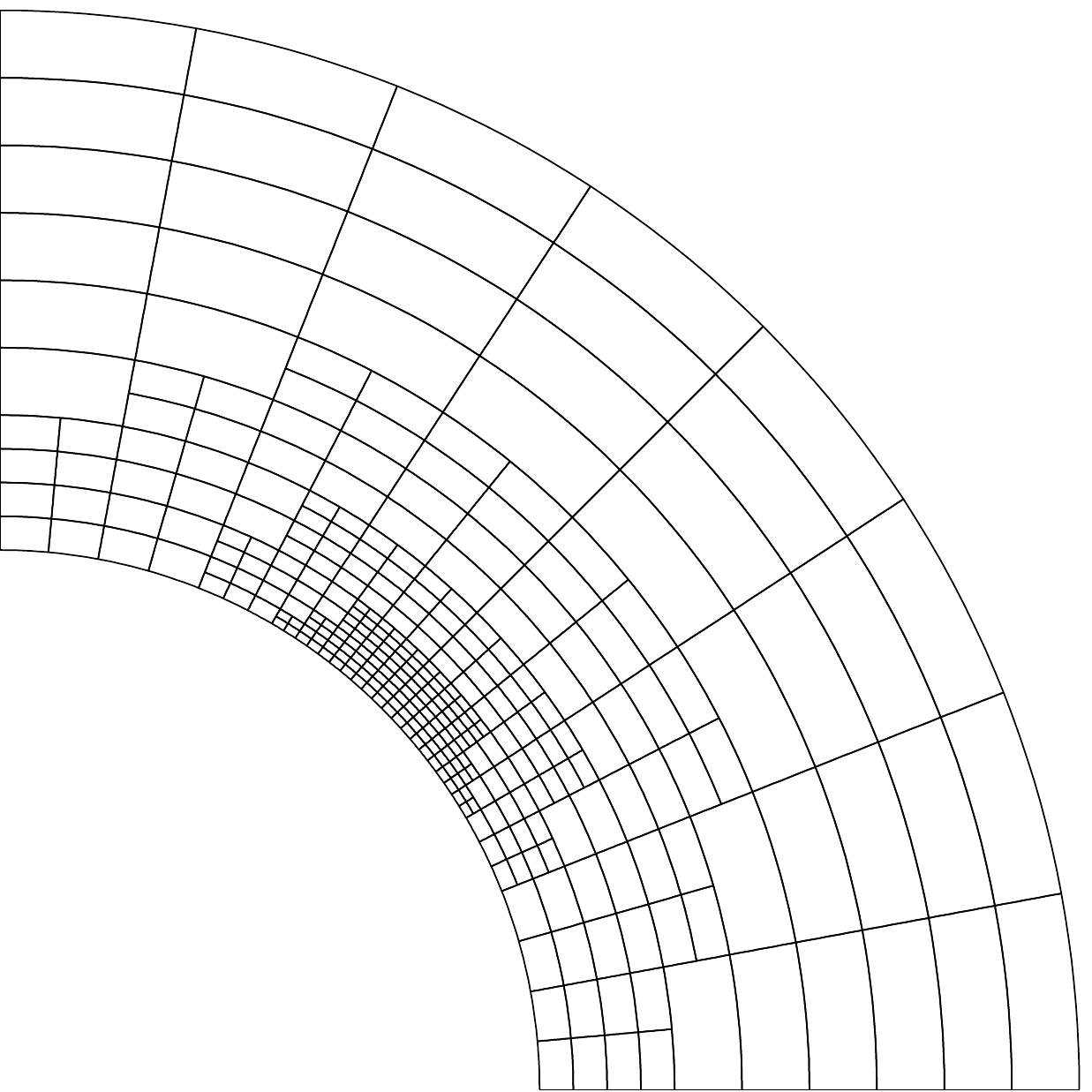} 
\includegraphics[width=.2\textwidth]{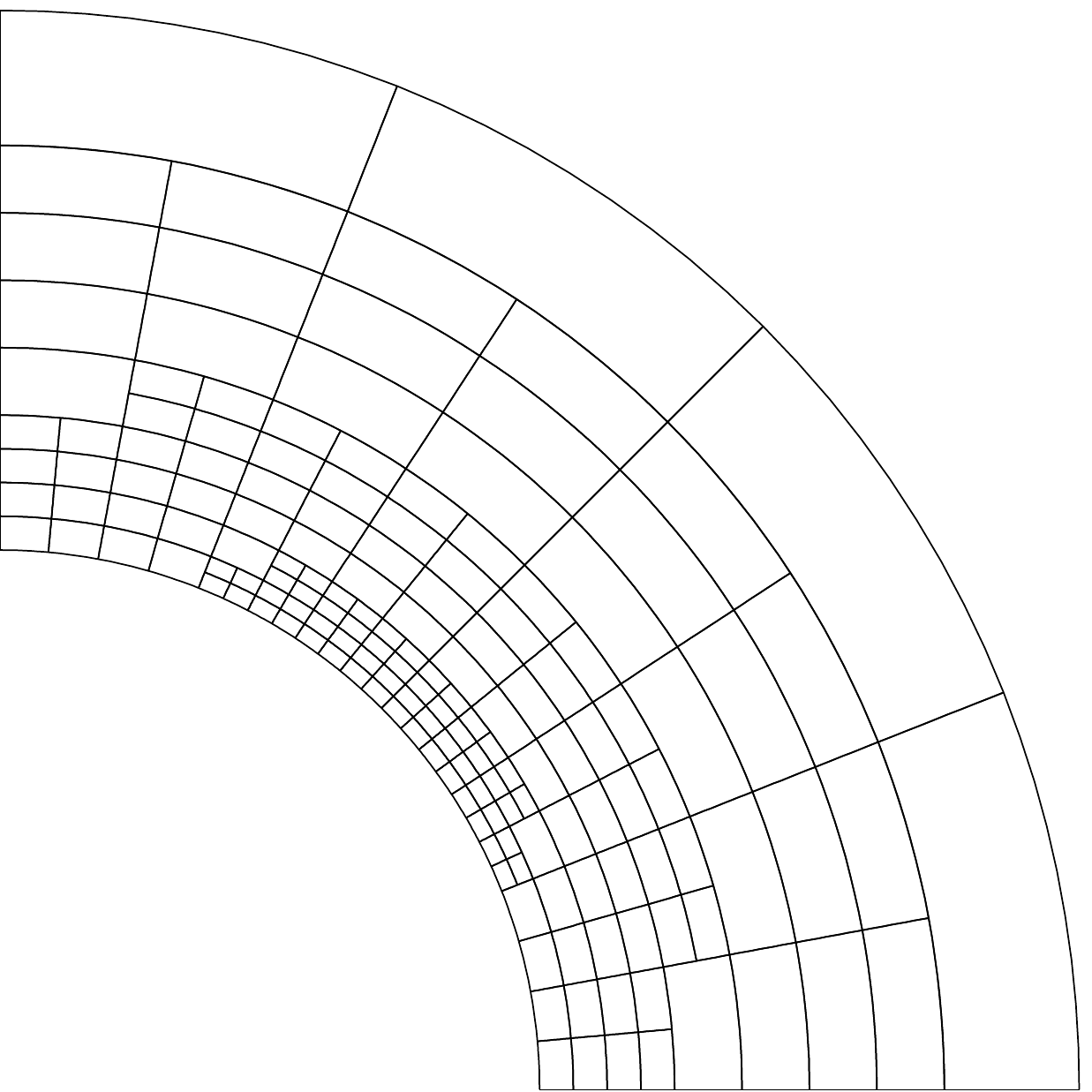}
\medskip 

\includegraphics[width=.3\textwidth]{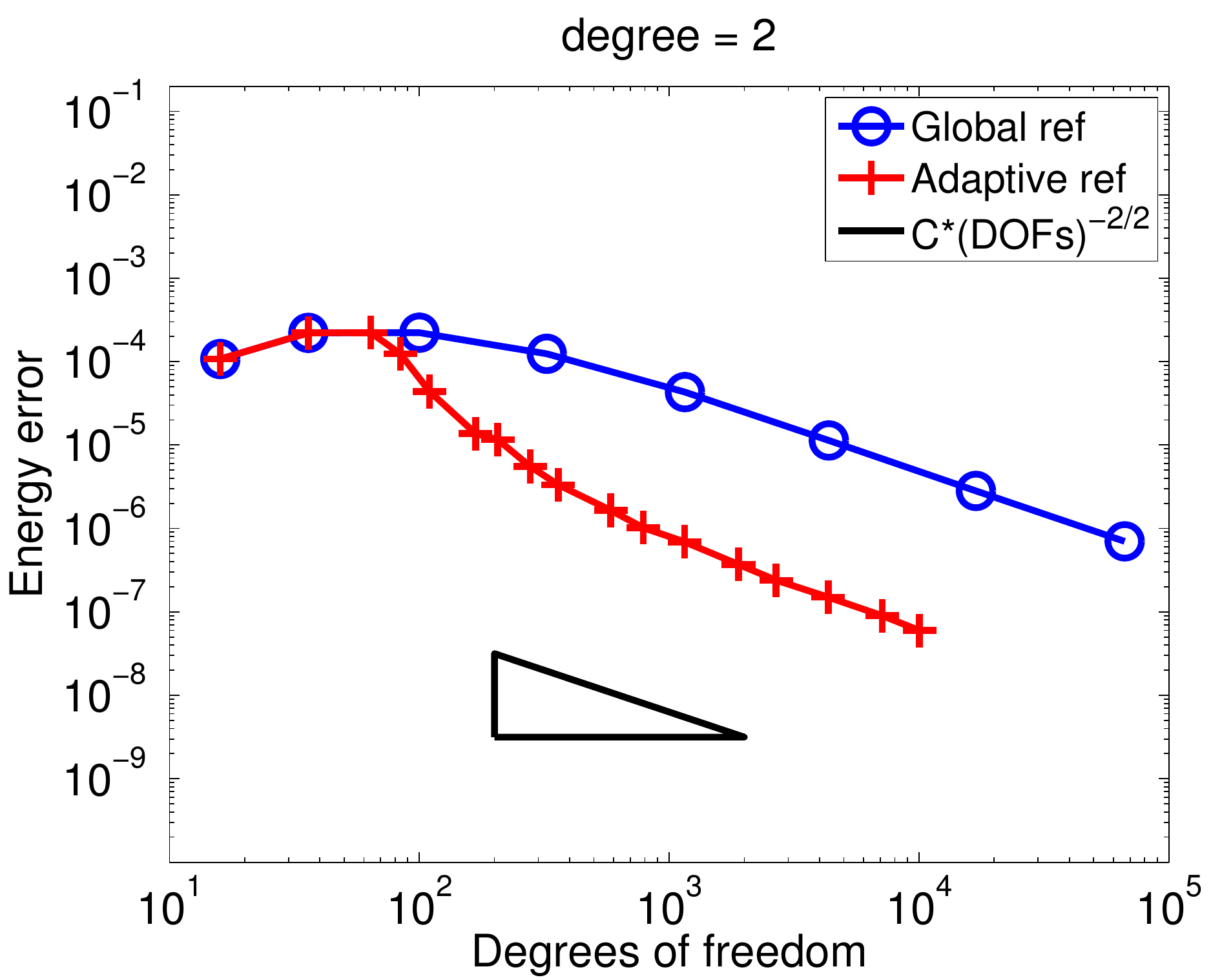}
\includegraphics[width=.3\textwidth]{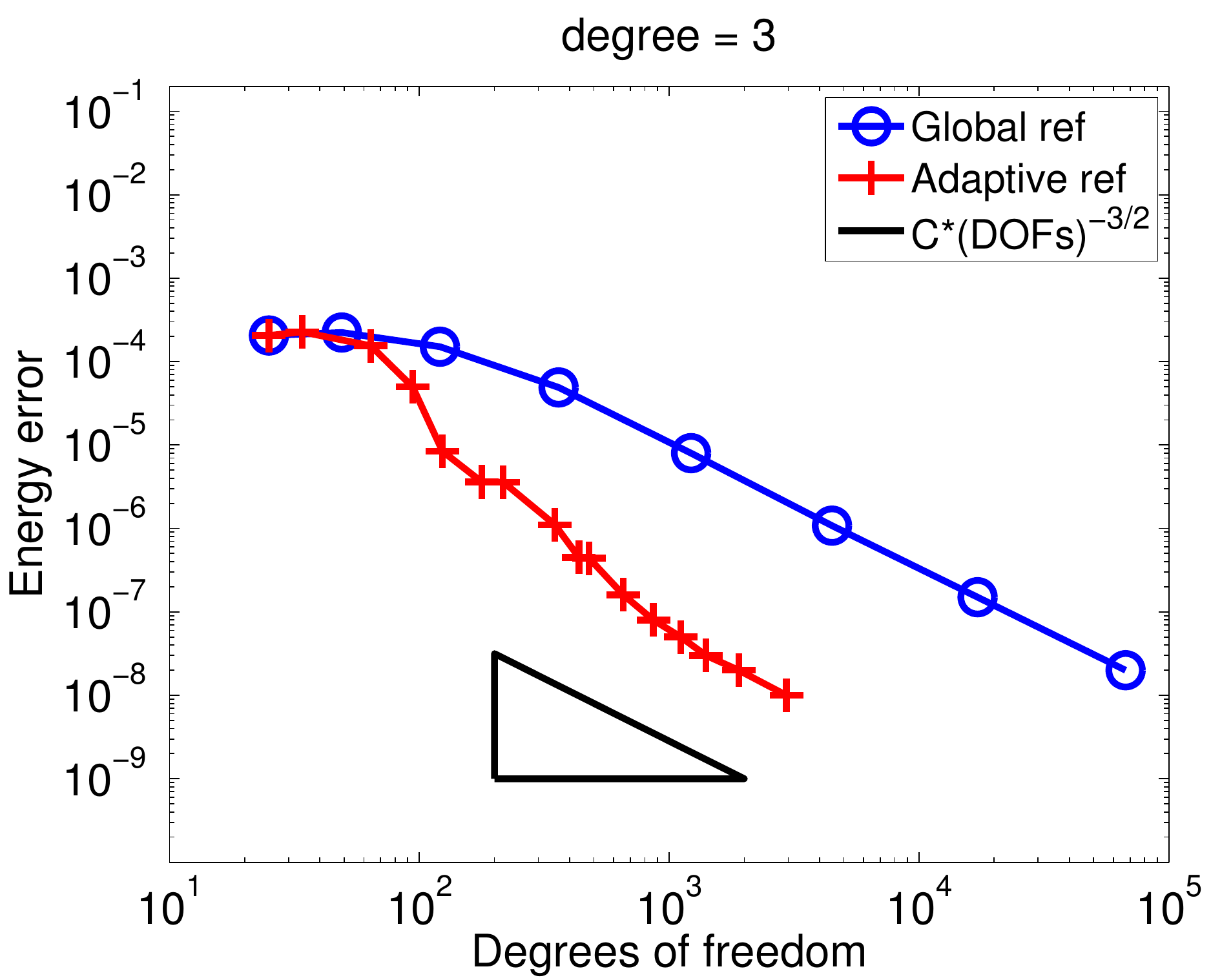}
\includegraphics[width=.3\textwidth]{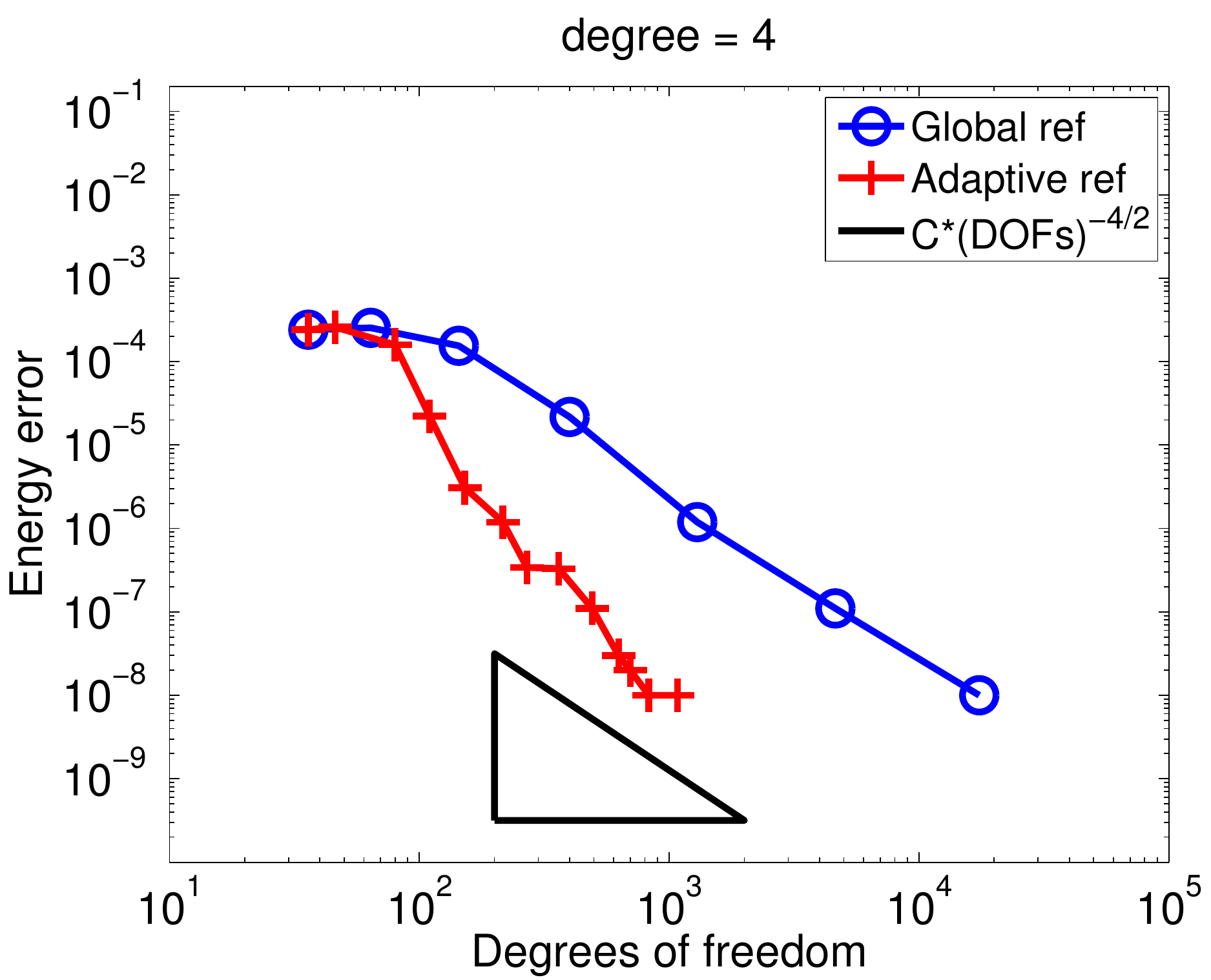}
\end{center}
\caption{\label{F:example ring} \small Some hierarchical meshes for the solution of Example~\ref{Ex:ring} obtaining $|u-U|_{H^1(\Omega)}\approx 1.10^{-6}$ in all cases; for biquadratics with $802$ elements and $788$ DOFs (top left), bicubics with $316$ elements and $349$ DOFs (top middle) and biquartics with $160$ elements and $216$ DOFs (top right). We plot the energy error $|u-U|_{H^1(\Omega)}$ vs. degrees of freedom; for biquadratics (bottom left), bicubics (bottom middle) and biquartics (bottom right).} 
\end{figure}
\end{example}


\begin{example}[A $3d$-domain: the unit cube]\label{Ex:cube}
We consider the cube $\Omega=[0,1]^3$ and choose $f$ and $g$ such that the exact solution $u$ of~\eqref{E:poisson} is given by $u(x,y,z) = e^{-100((x-\frac12)^2+(y-\frac12)^2+(z-\frac12)^2)}$. Since the solution is smooth enough, both strategies reach optimal orders of convergence, as showed in Figure~\ref{F:example cube errors}. However, we notice that in all cases, the curves corresponding to the adaptive strategy are meaningfully by below, which in practice is equivalent to achieve a given accuracy with considerably fewer degrees of freddom.

\begin{figure}[H!tbp]
\begin{center}
\includegraphics[width=.3\textwidth]{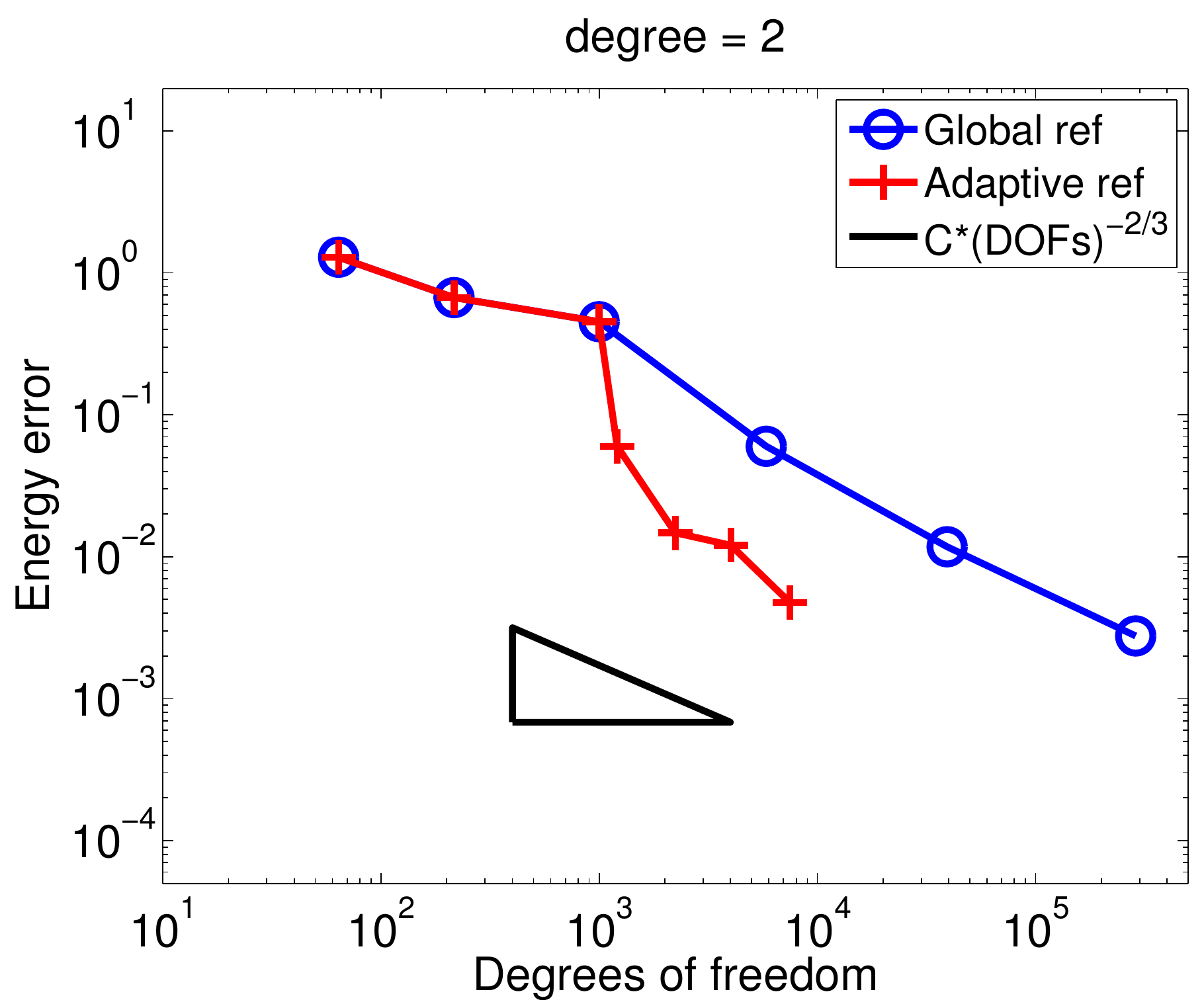}
\includegraphics[width=.3\textwidth]{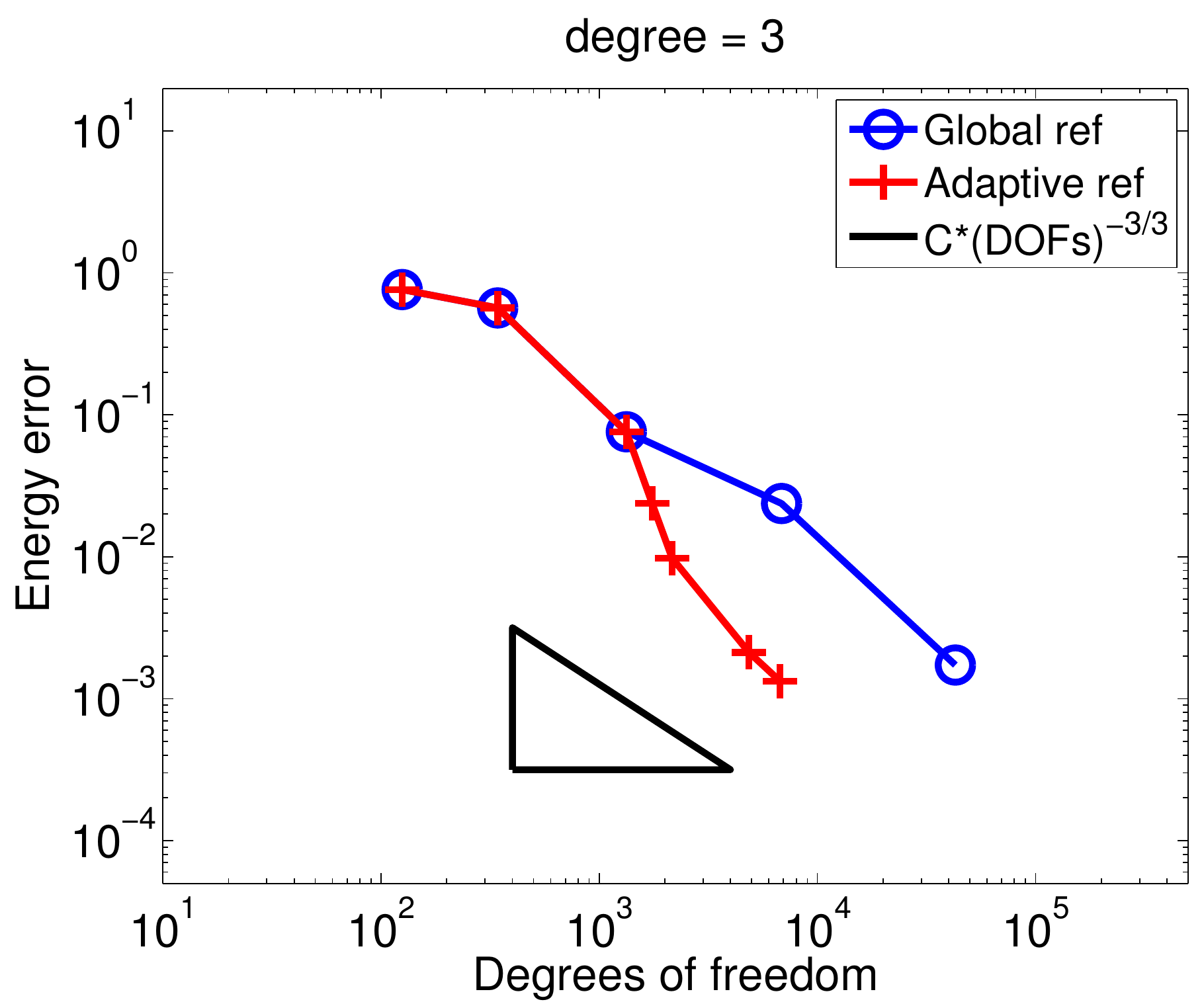}
\includegraphics[width=.3\textwidth]{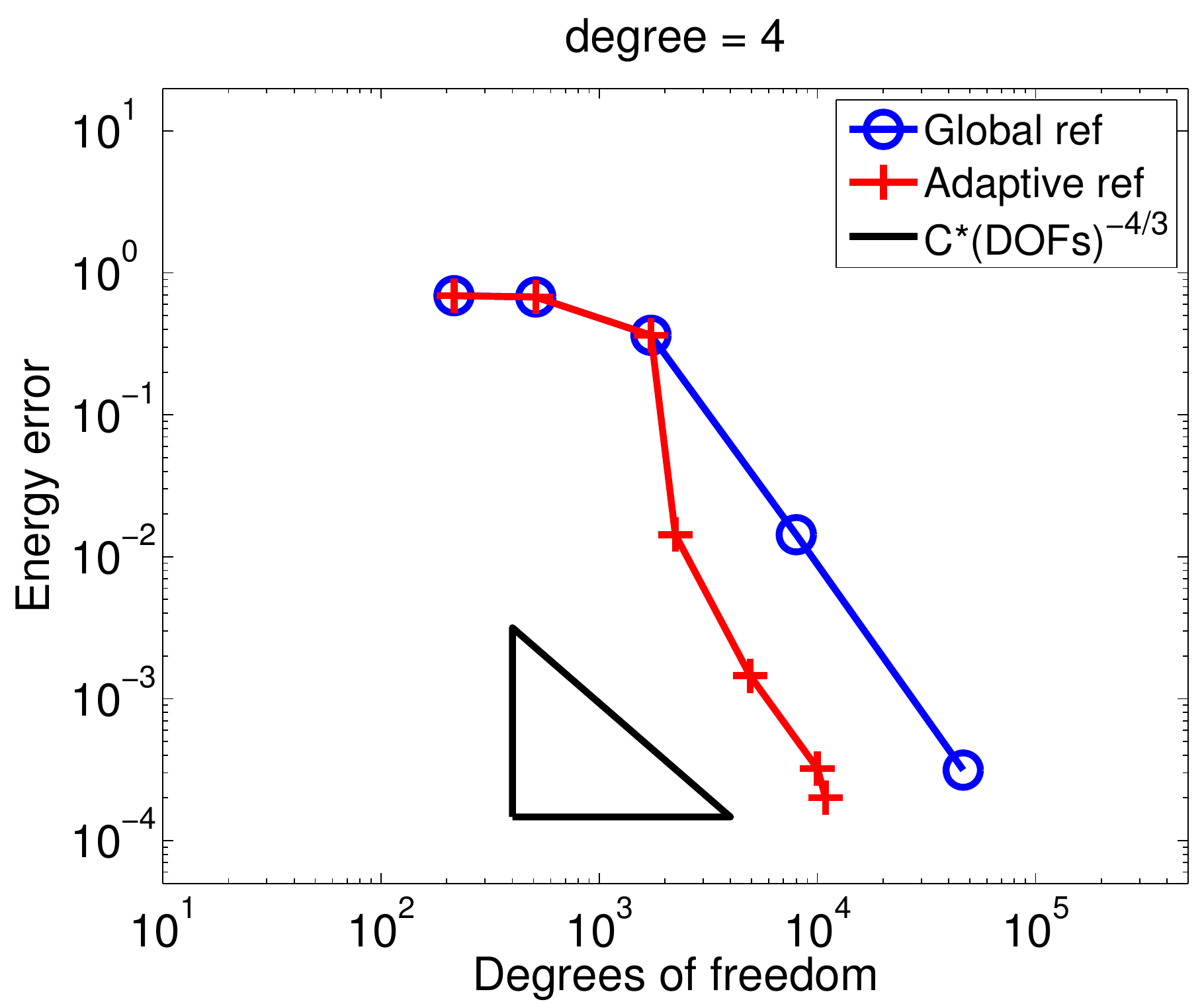}
\end{center}
\caption{\label{F:example cube errors} \small Energy error decay in terms of degrees of freedom for the solution of the Example~\ref{Ex:cube}; using biquadratics (left), bicubics (middle) and biquartics (right).}
\end{figure}
\end{example}

\paragraph{On the efficiency of the error estimators}
Finally, we analyse the behaviour of the efficiency index $\frac{\left(\sum_{\beta\in\HH}\EE_\beta^2\right)^\frac12}{\|\nabla(u-U)\|_{L^2(\Omega)}}$. In Figure~\ref{F:efficiency indices}, we plot this index at each iteration step for all the examples previously presented. We see that the energy error $\|\nabla(u-U)\|_{L^2(\Omega)}$ and the global a posteriori error estimator $\left(\sum_{\beta\in\HH}\EE_\beta^2\right)^\frac12$ are equivalent quantities, that is, there exists constants $c_1$ and $c_2$ such that
$$0<c_1\le\frac{\left(\sum_{\beta\in\HH}\EE_\beta^2\right)^\frac12}{\|\nabla(u-U)\|_{L^2(\Omega)}}\le c_2,$$
at each iteration step. Thus, we conclude that our estimators are not only reliable but also experimentally efficient. 

\begin{figure}[H!tbp]
\begin{center}
 \includegraphics[width=.3\textwidth]{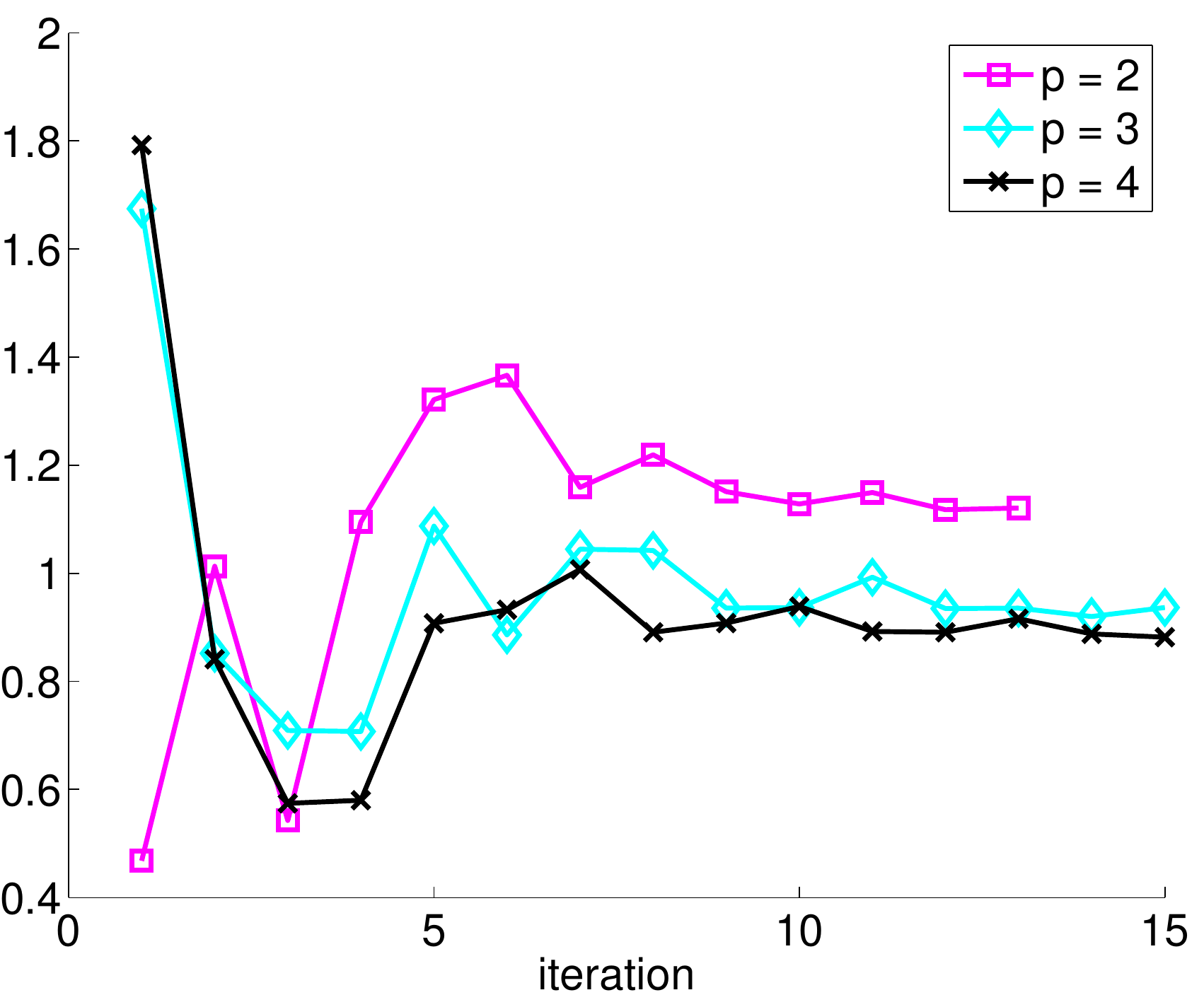}
\includegraphics[width=.3\textwidth]{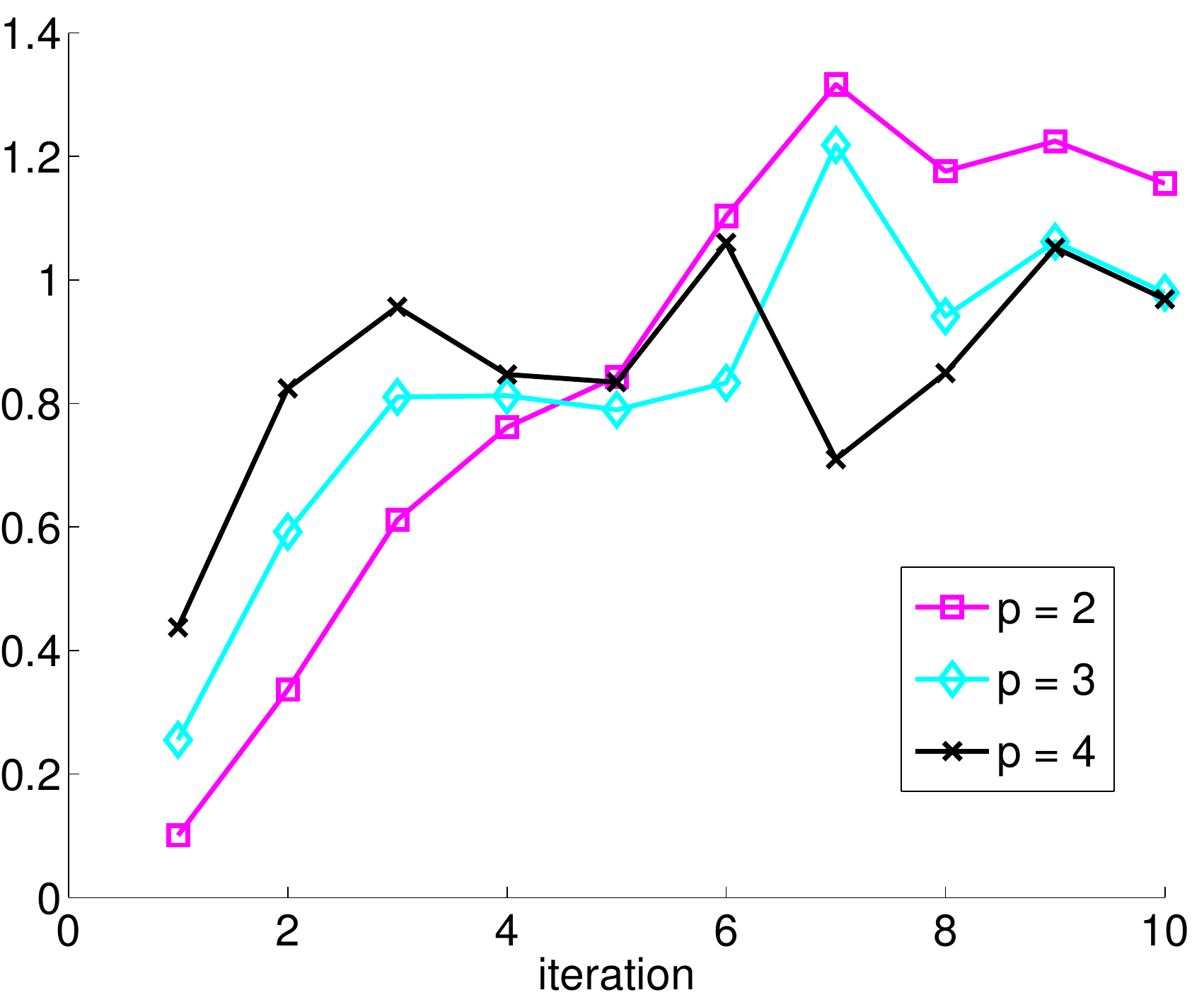}
 \includegraphics[width=.3\textwidth]{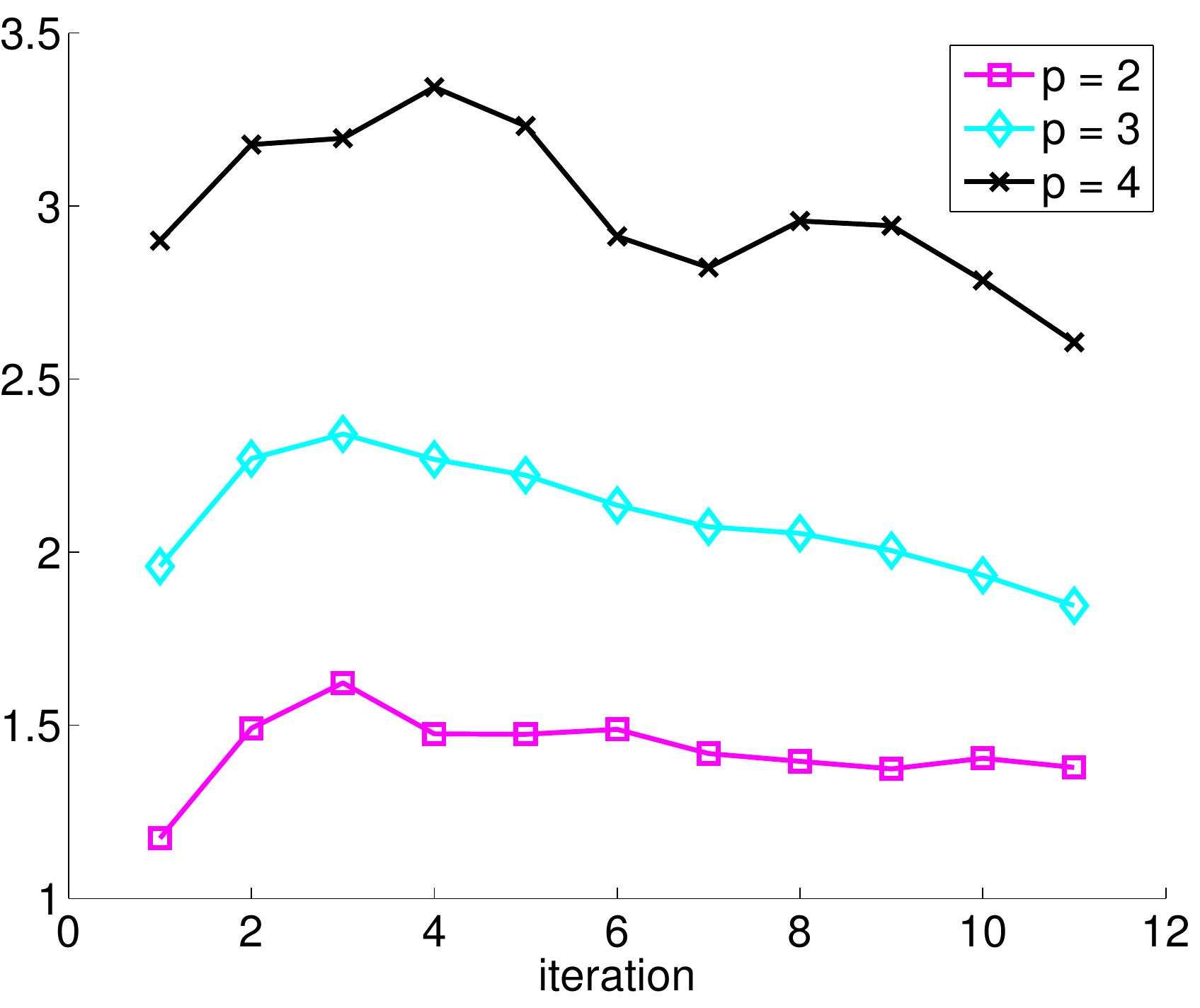}
\includegraphics[width=.3\textwidth]{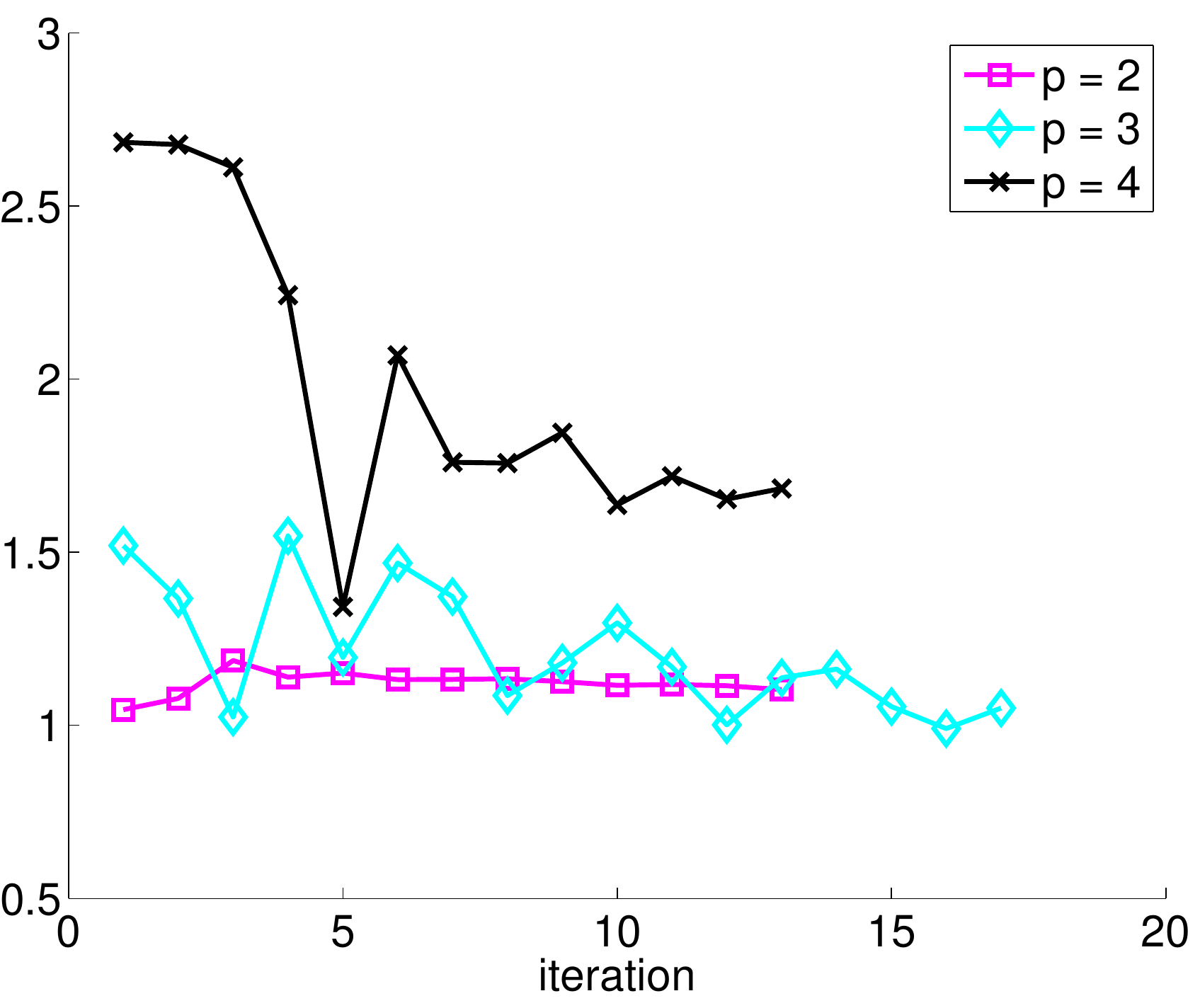}
\includegraphics[width=.3\textwidth]{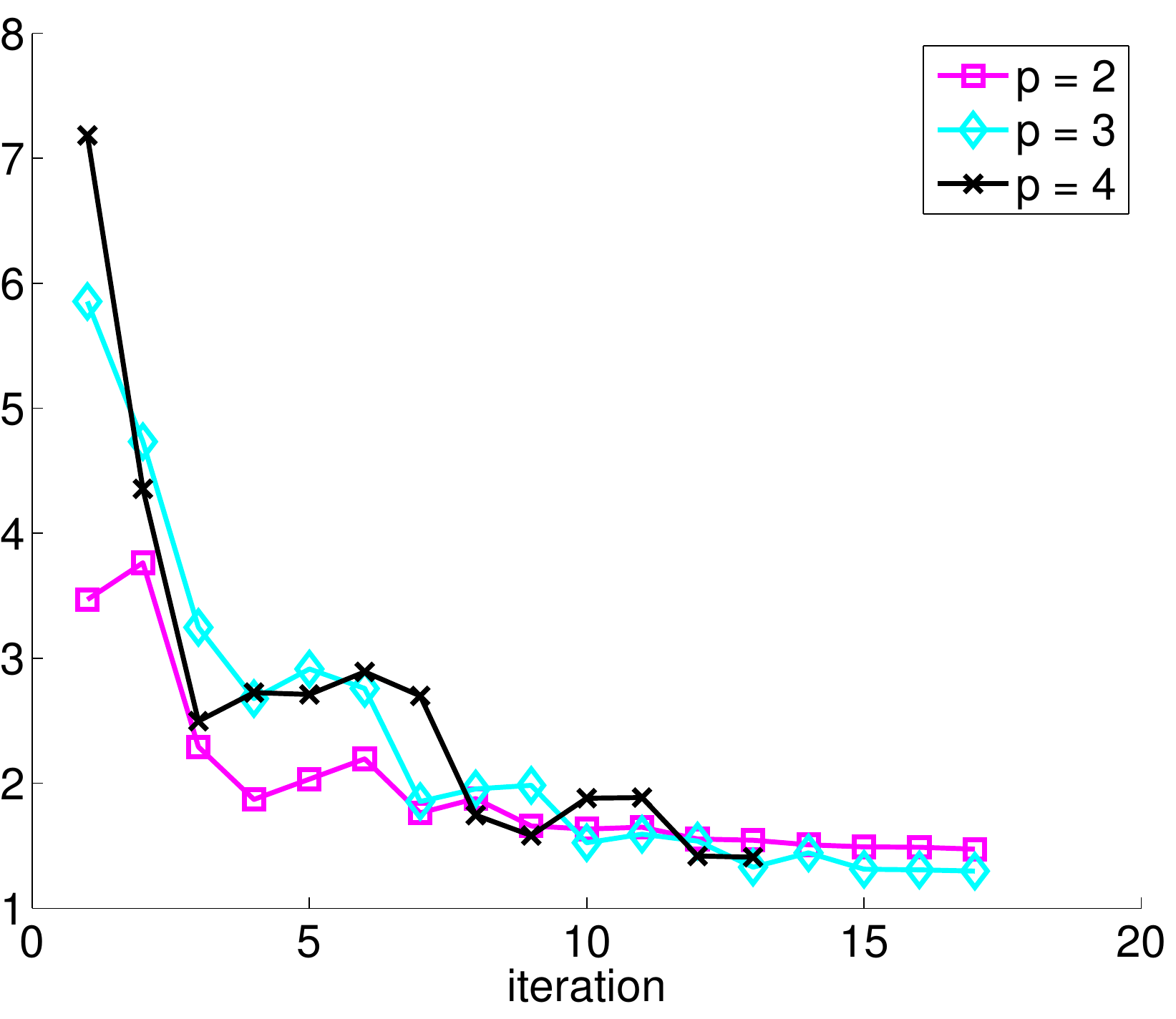} 
\includegraphics[width=.3\textwidth]{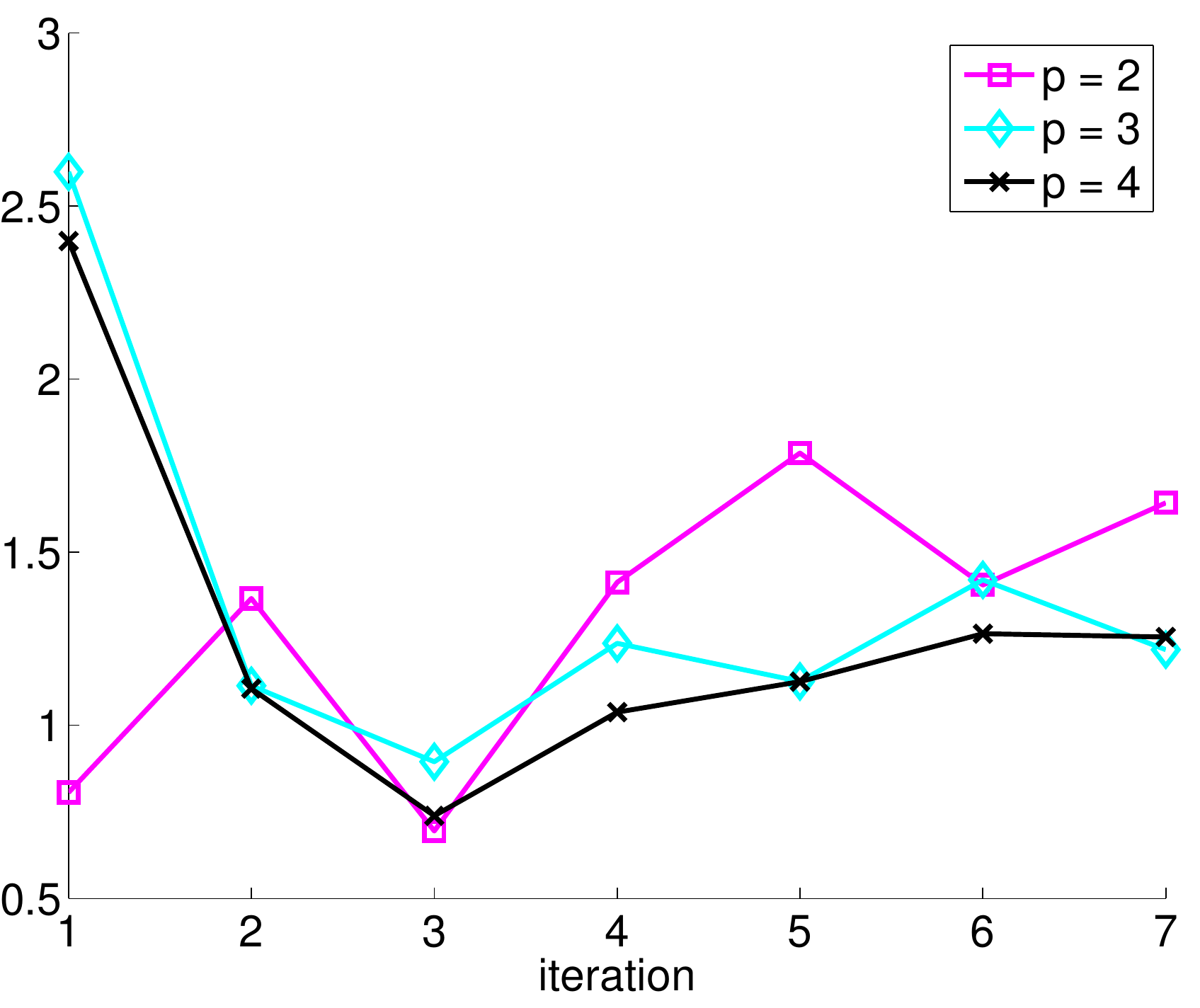}
\end{center}
\caption{\label{F:efficiency indices} Efficiency indices $\frac{\frac{1}{10}\left(\sum_{\beta\in\HH}\EE_\beta^2\right)^\frac12}{\|\nabla(u-U)\|_{L^2(\Omega)}}$; for Example~\ref{Ex:regular solution} (top left), Example~\ref{Ex:diagonal} (top middle), Example~\ref{Ex:Lshaped} (top right), Example~\ref{Ex:singular solution} (bottom left), Example~\ref{Ex:ring} (bottom middle), Example~\ref{Ex:cube} (bottom right).}
\end{figure} 

\section*{Acknowledgements}

{\small A. Buffa was partially supported by ERC AdG project CHANGE n. 694515, by MIUR PRIN project ``Metodologie innovative nella modellistica differenziale numerica'', and by Istituto Nazionale di Alta Matematica (INdAM). E.M. Garau was partially supported by CONICET through grant PIP 112-2011-0100742, by
Universidad Nacional del Litoral through grants CAI+D 500 201101 00029 LI, 501 201101 00476 LI, by Agencia Nacional
de Promoci\'on Cient\'ifica y Tecnol\'ogica, through grants
PICT-2012-2590 and PICT-2014-2522 (Argentina). This support is gratefully acknowledged.
}

\bibliographystyle{apalike}

\end{document}